\documentclass[a4paper,11pt]{amsart}

\usepackage{latexsym}
\usepackage{a4wide}
\usepackage{amscd}
\usepackage{graphics}
\usepackage{amsmath}
\usepackage{amssymb}
\usepackage{mathrsfs}
\usepackage{enumerate}
\usepackage{hyperref}
\usepackage{pgf,tikz}

\newcommand{\R}{\mathbb R}
\newcommand{\C}{\mathbb C}
\newcommand{\N}{\mathbb N}
\newcommand{\Z}{\mathbb Z}

\newcommand{\eps}{\varepsilon}
\newcommand{\abs}[1]{\left\vert #1 \right\vert}

\newcommand{\nor}[1]{\left\Vert #1 \right\Vert}

\newcommand{\Div}[1]{\mathrm{div}#1}
\newcommand{\Di}{\mathcal{D}^{1,2}}
\renewcommand{\Re}{\mathop{\mathfrak{Re}}}
\renewcommand{\Im}{\mathop{\mathfrak{Im}}}

\newtheorem{Theorem}{Theorem}[section]
\newtheorem{Corollary}[Theorem]{Corollary}
\newtheorem{Lemma}[Theorem]{Lemma}
\newtheorem{Proposition}[Theorem]{Proposition}
 \theoremstyle{definition} \newtheorem{Definition}[Theorem]{Definition}

 \newtheorem{remark}[Theorem]{Remark}

\begin{document}

\title[Sharp estimates for Aharonov-Bohm varying pole]{Sharp asymptotic estimates
  for eigenvalues of Aharonov-Bohm operators with varying poles}
\author{Laura Abatangelo, Veronica Felli}

\address{
\hbox{\parbox{5.7in}{\medskip\noindent
  L. Abatangelo, V. Felli\\
Dipartimento di Matematica e Applicazioni,\\
 Universit\`a di Milano Bicocca, \\
Via Cozzi 55, 20125 Milano (Italy). \\[2pt]
         {\em{E-mail addresses: }}{\tt laura.abatangelo@unimib.it, veronica.felli@unimib.it.}}}
}

\date{April 1, 2015}

\thanks{ 
The authors have been partially supported by the \emph{Gruppo Nazionale per
  l'Analisi Matematica, la Probabilit\`{a} e le loro Applicazioni
  (GNAMPA)} of the \emph{Istituto Nazionale di Alta Matematica
  (INdAM)}, 2014 INdAM-GNAMPA research project ``Stabilit\`{a}
spettrale e analisi asintotica per problemi singolarmente perturbati'',
and by the project ERC Advanced Grant 2013 
``Complex Patterns for Strongly Interacting Dynamical Systems - COMPAT''.
}

\begin{abstract}
  We investigate the behavior of eigenvalues for a magnetic
  Aharonov-Bohm operator with half-integer circulation and Dirichlet
  boundary conditions in a planar domain.  We provide sharp
  asymptotics for eigenvalues as the pole is moving in the interior of
  the domain, approaching a zero of an eigenfunction of the limiting
  problem along a nodal line.  As a consequence, we verify
  theoretically some conjectures arising from numerical evidences in
  preexisting literature.  The proof relies on an Almgren-type
  monotonicity argument for magnetic operators together with a sharp
  blow-up analysis.
\end{abstract}

\maketitle

\section{Introduction}\label{sec:introduction}

The aim of this paper is to investigate the behavior of the
eigenvalues of Aharonov-Bohm operators with moving poles.  For
$a=(a_1,a_2)\in\R^2$ and $\alpha\in \R\setminus\Z$, we consider the
vector potential
\[
A_{a}^\alpha(x)=\alpha\bigg(\frac{-(x_2-a_2)}{(x_1-a_1)^2+(x_2-a_2)^2},
\frac{x_1-a_1}{(x_1-a_1)^2+(x_2-a_2)^2}\bigg),\quad
x=(x_1,x_2)\in\R^2\setminus\{a\},
\]
which generates the Aharonov-Bohm magnetic field in $\R^2$ with pole
$a$ and circulation $\alpha$; such a field is produced by an
infinitely long thin solenoid intersecting perpendicularly the plane
$(x_1,x_2)$ at the point $a$, as the radius of the solenoid goes to
zero and the magnetic flux remains constantly equal to $\alpha$ (see
e.g. \cite{AT,AB,MOR}).
   
In this paper we will focus on the case of half-integer circulation,
so we will assume $\alpha=1/2$ and denote
\[
A_{a}(x)=A_{a}^{1/2}(x) =A_0(x-a),\quad\text{where}\quad
A_0(x_1,x_2)=\frac12\bigg(-\frac{x_2}{x_1^2+x_2^2},\frac{x_1}{x_1^2+x_2^2}\bigg).
\]
In the spirit \cite{BNNNT}, \cite{NNT} and \cite{NT}, we are
interested in studying the dependence on the pole $a$ of the spectrum
of Schr\"odinger operators with Aharonov-Bohm vector potentials,
i.e. of operators $(i\nabla +A_{a})^2$ acting on functions
$u:\R^2\to\C$ as
\[
(i\nabla +A_{a})^2 u=-\Delta u+2iA_{a}\cdot\nabla u+|A_{a}|^2 u.
\]
The interest in Aharonov-Bohm operators with half-integer circulation
$\alpha=1/2$ is motivated by the fact that nodal domains of
eigenfunctions of such operators are strongly related to spectral
minimal partitions of the Dirichlet laplacian with points of odd
multiplicity, see \cite{BNHHO09,NT}. We refer to papers \cite{BNH11, BNL14,
  H10, HHOO99, HHO10, HHO13, HHOT09, HHOT10', HHOT10} for details on the deep
relation between behavior of eigenfunctions, their nodal domains, and
spectral minimal partitions.  Furthermore, the investigation carried
out in \cite{BNNNT,lena,NNT,NT} highlighted a strong connection
between nodal properties of eigenfunctions and the critical points of
the map which associates eigenvalues of the operator $A_a$ to the
position of pole $a$. Motivated by this, in the present paper we
deepen the investigation started in \cite{BNNNT, NNT} about the
dependence of eigenvalues of Aharonov-Bohm operators on the pole
position, aiming at proving sharp asymptotic estimates for the
convergence of eigenvalues associated to operators with a moving pole.

Let $\Omega\subset\R^2$ be a bounded, open and simply connected
domain. For every $a\in\Omega$, we introduce the functional space
$H^{1 ,a}(\Omega,\C)$ as the completion of
\[
\{u\in
H^1(\Omega,\C)\cap C^\infty(\Omega,\C):u\text{ vanishes in a
  neighborhood of }a\}
\]
 with respect to the norm 
 $$
 \|u\|_{H^{1,a}(\Omega,\C)}=\left(\left\|\nabla u\right\|^2
   _{L^2(\Omega,\C^2)} +\|u\|^2_{L^2(\Omega,\C)}+\Big\|\frac{u}{|x-a|}
   \Big\|^2_{L^2(\Omega,\C)}\right)^{\!\!1/2}.
$$
It is easy to verify that 
$H^{1,a}(\Omega,\C)=\big\{u\in H^1(\Omega,\C):\frac{u}{|x-a|}\in
  L^2(\Omega,\C)\big\}$. 
We also observe that, in view of the
Hardy type inequality proved in \cite{LW99} (see \eqref{eq:hardy}), 
an equivalent norm in $H^{1,a}(\Omega,\C)$ is
given by 
\begin{equation}\label{eq:norma}
  \left(\left\|(i\nabla+A_{a}) u\right\|^2
    _{L^2(\Omega,\C^2)} +\|u\|^2_{L^2(\Omega,\C)}\right)^{\!\!1/2}.
\end{equation}
We also consider the space $H^{1 ,a}_{0}(\Omega,\C)$ as the completion
of $C^\infty_{\rm c}(\Omega\setminus\{a\},\C)$ with respect to the
norm $\|\cdot\|_{H^{1}_{a}(\Omega,\C)}$, so that
$H^{1,a}_{0}(\Omega,\C)=\big\{u\in H^1_0(\Omega,\C):\frac{u}{|x-a|}\in
L^2(\Omega,\C)\big\}$.

For every $a\in\Omega$, we consider the eigenvalue problem
\begin{equation}\label{eq:eige_equation_a}\tag{$E_a$}
  \begin{cases}
   (i\nabla + A_{a})^2 u = \lambda u,  &\text{in }\Omega,\\
   u = 0, &\text{on }\partial \Omega,
 \end{cases}
\end{equation}
in a weak sense, i.e.  we say that $\lambda\in\R$ is an eigenvalue of
problem \eqref{eq:eige_equation_a} if there exists $u\in
H^{1,a}_{0}(\Omega,\C)\setminus\{0\}$ (called eigenfunction) such that
\[
\int_\Omega (i\nabla u+A_{a} u)\cdot \overline{(i\nabla v+A_{a}
  v)}\,dx=\lambda\int_\Omega u\overline{ v}\,dx \quad\text{for all }v\in H^{1,a}_{0}(\Omega,\C).
\]
From classical spectral theory, the eigenvalue problem $(E_a)$ admits
a sequence of real diverging eigenvalues $\{\lambda_k^a\}_{k\geq 1}$
with finite multiplicity; in the enumeration
\[
\lambda_1^a \leq \lambda_2^a\leq\dots\leq \lambda_j^a\leq\dots,
\]
we repeat each eigenvalue as many times as its multiplicity.  We are
interested in the behavior of the function $a\mapsto \lambda_j^a$ in a
neighborhood of a fixed point $b\in\Omega$. Up to a translation, it is
not restrictive to consider $b=0$. Thus, we assume that $0\in\Omega$.

In \cite[Theorem 1.1]{BNNNT} and \cite[Theorem 1.2]{lena} it is proved
that, for all $j\geq1$, 
\begin{equation}\label{eq:74}
\text{the function }a\mapsto \lambda_j^a
\text{ is continuos in $\Omega$}.   
\end{equation}
A strong improvement of the regularity \eqref{eq:74} holds under
simplicity of the eigenvalue. Indeed in \cite[Theorem 1.3]{BNNNT} it
is proved that, if there exists $n_0\geq 1$ such that
\begin{equation}\label{eq:1}
  \lambda_{n_0}^0\quad\text{is simple},
\end{equation}
then the function $a\mapsto \lambda_{n_0}^a$ is of class $C^\infty$ in
a neighborhood of $0$; this regularity result is improved in
\cite[Theorem 1.3]{lena}, where, in the more general setting of
Aharonov-Bohm operators with many singularities, it is shown that,
under assumption \eqref{eq:1} the function $a\mapsto \lambda_{n_0}^a$
is analytic in a neighborhood of $0$.  Then the question of what is
the leading term in the asymptotic expansion of such a function (at
least on a single straight path around the limit point $0$) naturally
arises.  The main purpose of the present paper is to answer to such a
question. This may also shed some light on the nature of $0$ as a
critical point for the map $a \mapsto \lambda_a$ when the limit
eigenfunction has in $0$ a zero of order $k/2$ with $k\geq 3$ odd.

At a deep insight into the problem, papers \cite{BNNNT} and \cite{NT}
suggest a high reliability of the behavior of the eigenvalue
$\lambda_{n_0}^a$ on the structure of the nodal lines of \emph{the}
eigenfunction relative to $\lambda_{n_0}^0$. In order to enter into
the issue, let us establish the setting and some notation.

Let us assume that there exists $n_0\geq 1$ such that \eqref{eq:1}
holds and denote
\[
\lambda_0= \lambda_{n_0}^0
\]
and, for any $a\in\Omega$,
\[
 \lambda_a= \lambda_{n_0}^a.
 \]
 From \eqref{eq:74} it follows that, if $a\to 0$, then
\begin{equation*}
  \lambda_a\to \lambda_0.
\end{equation*}
Let $\varphi_0\in H^{1,0}_{0}(\Omega,\C)\setminus\{0\}$ be an
eigenfunction of problem $(E_0)$ associated to the eigenvalue
$\lambda_0= \lambda_{n_0}^0$, i.e. solving
\begin{equation}\label{eq:equation_lambda0}
 \begin{cases}
   (i\nabla + A_0)^2 \varphi_0 = \lambda_0 \varphi_0,  &\text{in }\Omega,\\
   \varphi_0 = 0, &\text{on }\partial \Omega,
 \end{cases}
\end{equation}
such that 
\begin{equation}\label{eq:83}
  \int_\Omega |\varphi_0(x)|^2\,dx=1.
\end{equation}
In view of \cite[Theorem 1.3]{FFT} (see also Proposition
\ref{prop:fft} below) we have that
\begin{equation}\label{eq:37}
  \varphi_0 \text{ has at $0$ a zero
    or order $\frac k2$ for some odd $k\in \N$},
\end{equation}
see \cite[Definition 1.4]{BNNNT}.  
We recall from \cite[Theorem
1.3]{FFT} and \cite[Theorem 1.5]{NT} that \eqref{eq:37} implies that
the eigenfunction $\varphi_0$ has got exactly $k$ nodal lines meeting
at $0$ and dividing the whole angle into $k$ equal parts.

A first result relating the rate of convergence of $\lambda_a$ to
$\lambda_0$ with the order of vanishing of $\varphi_0$ at $0$ can be
found in \cite{BNNNT}, where the following estimate is proved.
\begin{Theorem}[\cite{BNNNT}, Theorem 1.7]\label{t:stimaBNNNT}
  If assumptions \eqref{eq:1} and \eqref{eq:37} with $k\geq 3$ are
  satisfied, then
 \[
  |\lambda_a -\lambda_0 | \leq C |a|^{\frac{k+1}{2}} \qquad \text{as }a\to0 
 \]
 for a constant $C>0$ independent of $a$.
\end{Theorem}
As already mentioned, the latter theorem pursue the idea that the
asymptotic expansion of the function $a\mapsto \lambda_a$ has to do
with the nodal properties of the related limit eigenfunction.

The main result of the present paper establishes the exact order of
the asymptotic expansion of $\lambda_a -\lambda_0$ along a suitable
direction as $|a|^k$, where $k$ is the number of nodal lines of
$\varphi_0$ at $0$ which coincides with twice the order of vanishing
of $\varphi_0$ 
in assumption \eqref{eq:37}.  In addition, we detect 
the sharp coefficient of the asymptotics, which can be characterized in
terms of the 
limit profile of a 
blow-up sequence obtained by a 
suitable scaling of approximating eigenfunctions.

In order to state our main result, we need to recall some known facts
and to introduce some additional notation. 
By \cite[Theorem 1.3]{FFT} (see Proposition
\ref{prop:fft} below), 
if $\varphi_0$ is an eigenfunction of $(i\nabla +
A_0)^2$ on $\Omega$ satisfying  assumption \eqref{eq:37}, 
there exist $\beta_1,\beta_2\in\C$ such that
$(\beta_1,\beta_2)\neq(0,0)$ and
\begin{equation}\label{eq:131}
  r^{-k/2} \varphi_0(r(\cos t,\sin t)) \to 
  \beta_1
  \frac{e^{i\frac t2}}{\sqrt{\pi}}\cos\Big(\frac k2
  t\Big)+\beta_2 
  \frac{e^{i\frac t2}}{\sqrt{\pi}}\sin\Big(\frac k2
  t\Big) \quad \text{in }C^{1,\tau}([0,2\pi],\C)
\end{equation}
as $r\to0^+$ for any $\tau\in (0,1)$.

Let $s_0$ be the positive half-axis $s_0=\{(x_1,x_2)\in\R^2:
x_2=0\text{ and }x_1\geq0\}$.  We observe that, for every odd natural
number $k$, there exists a unique (up to a multiplicative constant)
function $\psi_k$ which is harmonic on $\R^2\setminus s_0$,
homogeneous of degree $k/2$ and vanishing on $s_0$.  Such a function
is given by
\begin{equation}\label{eq:psi_k}
  \psi_k(r\cos t,r\sin t)= r^{k/2} \sin
  \bigg(\frac{k}{2}\,t\bigg),\quad 
  r\geq0,\quad t\in[0,2\pi].
\end{equation}
Let $s:=\{(x_1,x_2)\in\R^2: x_2=0\text{ and }x_1\geq 1\}$ and
$\R^2_+=\{(x_1,x_2)\in\R^2:x_2>0)\}$.  We denote as $\Di_s(\R^2_+)$
the completion of $C^\infty_{\rm c}(\overline{\R^2_+} \setminus s)$
under the norm $( \int_{\R^2_+} |\nabla u|^2\,dx )^{1/2}$.  From the
Hardy type inequality proved in \cite{LW99} (see \eqref{eq:hardy}) and
a change of gauge, it follows that functions in $\Di_{s}(\R^2_+)$
satisfy the following Hardy type inequality:
\[ 
\int_{\R^2} |\nabla \varphi(x)|^2\,dx
\geq \frac14 \int_{\R^2} \dfrac{|\varphi(x)|^2}{|x-{\mathbf e}|^2}\,dx,
\quad\text{for all }\varphi\in \Di_{s}(\R^2_+),
 \] 
where ${\mathbf e}=(1,0)$.
Then 
\[
\Di_{s}(\R^2_+)=\Big\{ u\in L^1_{\rm loc}(\overline{\R^2_+} \setminus s):\
 \nabla u\in L^2(\R^2_+), 
\ \tfrac{u}{|x-{\mathbf e}|}\in L^2(\R^2_+), \text{ and } u=0 \text{
  on }s\Big\}.
\]
The functional
\begin{equation}\label{eq:Jk}
 J_k: \Di_{s}(\R^2_+)\to\R,\quad 
J_k(u) = \frac12 \int_{\R^2_+} |\nabla u(x)|^2 \,dx-
 \int_{ \partial
 \R^2_+\setminus s} u(x_1,0)\frac{\partial \psi_k}{\partial x_2}(x_1,0)\,dx_1,
\end{equation}
is well-defined on the space $\Di_{s}(\R^2_+)$; we notice that
$\frac{\partial \psi_k}{\partial x_2}$ vanishes on $ \partial
\R^2_+\setminus s_0$, so that $\int_{ \partial \R^2_+\setminus s}
u(x_1,0)\frac{\partial \psi_k}{\partial x_2}(x_1,0)\,dx_1=\int_0^1
u(x_1,0)\frac{\partial \psi_k}{\partial x_2}(x_1,0)\,dx_1$.  By
standard minimization methods, $J_k$ achieves its minimum over the
whole space $\Di_{s}(\R^2_+)$ at some function $w_k\in
\Di_{s}(\R^2_+)$, i.e.  there exists $w_k\in \Di_{s}(\R^2_+)$ such
that
\begin{equation}\label{eq:Ik}
{\mathfrak m}_k=\min_{u\in \Di_{s}(\R^2_+)}J_k(u)=J_k(w_k).
\end{equation}
We note that 
\begin{equation}\label{eq:segno_mk}
{\mathfrak m}_k=J_k(w_k)=-\frac12  \int_{\R^2_+} |\nabla
w_k(x)|^2 \,dx=-\frac12 \int_0^1 \dfrac{\partial_+ \psi_k}{\partial x_2}(x_1,0)\,w_k (x_1,0)\,dx_1  
<0,
\end{equation}
where, for all $x_1>0$, $\frac{\partial_+ \psi_k}{\partial
  x_2}(x_1,0)=\lim_{t\to0^+}\frac{\psi_k(x_1,t)-\psi_k(x_1,0)}{t}=\frac
k2 x_1^{\frac k2-1}$.

We are now in a position to state our main theorem.
\begin{Theorem}\label{t:main_asy_eige}
Let
$\Omega\subset\R^2$ be a bounded, open and simply connected
domain such that $0\in\Omega$ and let $n_0\geq 1$ be such that 
the $n_0$-th eigenvalue $\lambda_0=\lambda_{n_0}^0$ of $(i\nabla +
A_0)^2$ on $\Omega$ is simple  with associated eigenfunctions having  in
  $0$ a zero of order $k/2$ with $k\in\N$ odd. For $a\in\Omega$ let $\lambda_a=\lambda_{n_0}^a$ be
  the $n_0$-th eigenvalue of $(i\nabla +
A_a)^2$ on $\Omega$. 
Let ${\mathfrak r}$ be the half-line tangent to a nodal line of eigenfunctions associated to $\lambda_0$
  ending at $0$. 
Then, as $a\to0$ with $a\in{\mathfrak r}$,  
\[
\frac{\lambda_0-\lambda_a}{|a|^k}\to
-4 \frac{|\beta_1|^2+|\beta_2|^2}{\pi}\,{\mathfrak m}_k
\]
with $(\beta_1,\beta_2)\neq(0,0)$ being as in \eqref{eq:131} and
${\mathfrak m}_k$ being as in \eqref{eq:Ik}--\eqref{eq:segno_mk}.
\end{Theorem}

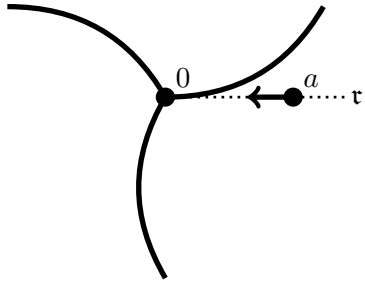
\begin{figure}\centering
\begin{tikzpicture}[scale=1.2]
    \draw[line width=2pt] (0,0) to [out=0, in=-120] (1.732,1);
    \draw[line width=2pt] (0,0) to [out=120, in=0] (-1.732,1);
    \draw[line width=2pt] (0,0) to [out=240, in=120] (0,-2);
\fill (0,0) circle (3pt) node[above right] {$0$};
\draw[dotted,line width=1pt]  (0,0) -- (2,0);
\draw[line width=2pt,->] (1.4,0) -- (0.9,0);
\fill (1.4,0) circle (3pt) node[above right] {$a$};
\node at (2.1,0) {${\mathfrak r}$};
                              \end{tikzpicture}
\caption{$a$ approaches $0$ along the tangent ${\mathfrak r}$ to a
  nodal line of $\varphi_0$.}
\label{fig:1}
\end{figure}

We remark that Theorem \ref{t:main_asy_eige} is significant non only
from a pure ``analytic'' point of view (detecting of sharp
asymptotics), but also from a quite theoretical point of view.
Indeed, several numerical simulations presented in \cite{BNNNT} are
validated and confirmed by Theorem \ref{t:main_asy_eige}.

\begin{remark}\label{rem:segnoauto}
  Due to the analyticity of the function $a\mapsto \lambda_a$
  established in \cite[Theorem 1.3]{lena}, from Theorem
  \ref{t:main_asy_eige} it follows that
  \[
\frac{\lambda_0-\lambda_a}{|a|^k}\to
4 \frac{|\beta_1|^2+|\beta_2|^2}{\pi}\,{\mathfrak m}_k
\]
as $a\to0$ along the half-line opposite to the tangent to a nodal line
of $\varphi_0$. In particular, we have that the restriction of the
function $\lambda_0-\lambda_a$ on the straight line tangent to a nodal
line of $\varphi_0$ changes sign at $0$ (is positive on the side of
the nodal line of $\varphi_0$ and negative on the opposite
side). Hence, if $\lambda_0$ is simple, then $0$ cannot be an extremal
point of the map $a\mapsto\lambda_a$.
\end{remark}

Theorem \ref{t:main_asy_eige} and the consequent Remark
\ref{rem:segnoauto} allow completing some results of papers
\cite{BNNNT,NNT,NT} concerning the investigation of critical and
extremal points of the map $a\mapsto \lambda_a$.  It is worth
recalling from \cite[Corollary 1.2]{BNNNT} that the function $a\mapsto
\lambda_a$ must have an extremal point in $\Omega$.  More precisely,
in \cite{BNNNT} the following result is proved.
\begin{Proposition}[{\cite{BNNNT}, Corollary 1.8}]\label{p:BNNNT}
  Fix any $j\in\N$. If $0$ is an extremal point of $a\mapsto
  \lambda_j^a$, then either $\lambda_j^0$ is not simple, or the
  eigenfunction of $(i\nabla + A_0)^2$ associated to $\lambda_j^0$ has
  at $0$ a zero of order $k/2$ with $k\geq 3$ odd.
\end{Proposition}

In view of Theorem \ref{t:main_asy_eige} and Remark
\ref{rem:segnoauto}, we can exclude the second alternative in
Proposition \ref{p:BNNNT}, obtaining the following result.
\begin{Corollary}\label{c:primo}
  Fix any $j\in\N$. If $0$ is an extremal point of $a\mapsto
  \lambda_j^a$, then $\lambda_j^0$ is not simple.
\end{Corollary}

Moreover, the aforementioned simulations \cite{BNNNT} suggest that
extremal points of the map $a\mapsto \lambda_a$ are generally attained
at points where the function itself is not differentiable.  Taking
into account Corollary \ref{c:primo}, we may conjecture that the
missed differentiability is produced by the dropping of assumption
\eqref{eq:1}.

As a further remark, Theorem \ref{t:main_asy_eige} proves that the
asymptotic expansion of $\lambda_0-\lambda_a$ has a leading term of
odd degree, hence, if $k\geq 3$, $0$ is a stationary inflexion point
along $k$ directions (corresponding to the nodal lines of
$\varphi_0$), as experimentally predicted by numerical simulations in
\cite[Section 7]{BNNNT}. More precisely, as a consequence of Theorem
\ref{t:main_asy_eige} and Remark \ref{rem:segnoauto} we can state the
following result.

\begin{Corollary}
  Under assumptions \eqref{eq:1} and \eqref{eq:37} with $k\geq 3$, $0$
  is a saddle point for the map $a\mapsto \lambda_a$. In particular,
  $0$ is a stationary and not extremal point.
\end{Corollary}

On the other hand, under assumptions \eqref{eq:1} and \eqref{eq:37}
with $k=1$, Theorem \ref{t:main_asy_eige} implies that the gradient of
the function $a\mapsto\lambda_a$ in $0$ is different from zero, then
$0$ is not a stationary point, a fortiori not even an extremal point;
we then recover a result stated in \cite[Corollary 1.2]{NT}.

From Theorem \ref{t:main_asy_eige}, we
directly obtain that, under the assumptions of Theorem
\ref{t:main_asy_eige}, the Taylor polynomials of the function
$a\mapsto \lambda_0-\lambda_a$ with center $0$ and degree strictly smaller than $k$ vanish,
since by Theorem \ref{t:main_asy_eige} they vanish on the $k$
independent directions corresponding to the
 nodal lines of $\varphi_0$ (see also \cite[Lemma 6.6]{BNNNT}). Then we obtain the following Taylor
 expansion at $0$:
\[
 \lambda_0-\lambda_a=P(a)+o(|a|^k),\quad\text{as }|a|\to0^+,
\]
for some 
\[
P\not\equiv0,\quad P(a)=P(a_1,a_2)=\sum_{j=0}^k \alpha_j a_1^{k-j}
a_2^j
\] 
homogeneous polynomial of degree $k$. The detection of the exact value
of all coefficients of the polynomial (and hence the sharp asymptotics
along any direction) is the object of a current investigation.

The proof of Theorem \ref{t:main_asy_eige} is based on the
Courant-Fisher minimax characterization of eigenvalues. The
asymptotics for eigenvalues is derived by combining estimates from
above and below of the Rayleigh quotient. To obtain sharp estimates,
we construct proper test functions for the Rayleigh quotient by
suitable manipulation of eigenfunctions. In this way, we obtain upper
and lower bounds whose limit as $a\to0$ can be explicitly computed
taking advantage of a fine blow-up analysis for scaled
eigenfunctions. More precisely, we prove (see Theorem \ref{t:blowup2})
that the blow-up sequence
\begin{equation}\label{eq:7}
\dfrac{\varphi_a(|a|x)}{|a|^{k/2}}
\end{equation}
converges as $|a|\to 0^+$, $a\in\mathfrak r$, to a limit profile,
which can be identified, up to a phase and a change of coordinates,
with $w_k+\psi_k$, being $w_k$ and $\psi_k$ as in \eqref{eq:Ik} and
\eqref{eq:psi_k} respectively. The proof of the energy estimates for
the blow-up sequence uses a monotonicity argument inspired by
\cite{almgren}, based on the study of an Almgren-type frequency
function given by the ratio of the local magnetic energy over mass
near the origin; see \cite{FFT,kurata,NNT} for Almgren-type
monotonicity formulae for elliptic operators with magnetic potentials.
We mention that a similar approach based on estimates of the Rayleigh
quotient, blow-up analysis and monotonicity formula was used in
\cite{AFT2} to prove a sharp control of the rate of convergence of the
eigenvalues and eigenfunctions of the Dirichlet laplacian in a
perturbed domain (obtained by attaching a shrinking handle to a smooth
region) to the relative eigenvalue and eigenfunction in the limit
domain (see also \cite{AFT1,FT12} for blow-up analysis and
monotonicity formula); however, in \cite{AFT1,AFT2,FT12} only the
particular case of limit eigenfunctions having at the singular point
the lowest vanishing order (corresponding to the case $k=1$ in our
setting) was considered. In the present paper we do not prescribe any
restriction on the order of the zero of the limit eigenfunction: this
produces significant additional difficulties with respect to
\cite{AFT2}, the main of which relies in the identification of the
limit profile of the blow-up sequence \eqref{eq:7}. Such a difficulty
is overcome here by fine energy estimates of the difference between
approximating and limit eigenfunctions, performed exploiting the
invertibility of an operator associated to the limit eigenvalue
problem.

\bigskip 

The paper is organized as follows.  Sections \ref{sec:2} and
\ref{sec:3} are devoted to set up the framework, recall some useful
known facts, introduce notation and prove some basic inequalities.
Section \ref{sec:4} contains the construction of a suitable limit
profile which will be used to describe the limit of the blowed-up
sequence.  The study of the behavior of such a blow-up sequence can
proceed thanks to the Almgren-type monotonicity argument which is
presented in section \ref{sec:5}.  Via the energy estimates proved
within section \ref{sec:5}, in section \ref{sec:6} we present some
preliminary upper and lower bound for the difference
$\lambda_0-\lambda_a$, relying on the well-known minimax
characterization for eigenvalues.  Section \ref{sec:7} contains energy
estimates of the difference between approximating and limit
eigenfunctions which are used to identify the limit profile in the
sharp blow-up analysis which is performed in section
\ref{sec:blow-up-analysis}.  Finally, section \ref{sec:9} concludes
the proof of Theorem \ref{t:main_asy_eige}.

\medskip
\noindent {\bf Notation. } We list below some notation used throughout
the paper.\par
\begin{itemize}
\item[-] For all $r>0$ and $a\in\R^2$, $D_r(a)=\{x\in\R^2:|x-a|<r\}$
  denotes the disk of center $a$ and radius $r$.
\item[-] For all $r>0$, 
$D_r=D_r(a)$ denotes the disk of center $0$ and
radius $r$.
\item[-] For every complex number $z\in\C$, $\overline{z}$ denotes its
  complex conjugate.
\item[-] For $z\in\C$, $\Re z$ denotes its real
  part and $\Im z$ its imaginary part.
\end{itemize}

\section{Preliminaries}\label{sec:2}

\subsection{Diamagnetic and Hardy inequalities}.  We recall from
\cite{LW99} (see also \cite[Lemma 3.1 and Remark 3.2]{FFT}) the
following Hardy type inequality
\begin{equation}\label{eq:hardy}
  \int_{D_r(a)} |(i\nabla+A_a)u|^2\,dx \geq \frac14 \int_{D_r(a)}\frac{|u(x)|^2}{|x-a|^2}\,dx,
\end{equation}
which holds for all $r>0$, $a\in \R^2$  and $u\in H^{1,a}(D_r(a),\C)$.

We also recall the well-known \emph{diamagnetic inequality} (see e.g. 
\cite{LL} or \cite[Lemma A.1]{FFT}for a proof): if $a\in\Omega$ and   $u\in
H^{1,a}(\Omega,\C)$, then 
\begin{equation}\label{eq:diamagnetic}
  |\nabla |u|(x)|\leq
  |i\nabla u(x)+A_a(x)u(x)| \quad {\rm for \
    a.e.} \ x\in \Omega.
\end{equation}

\subsection{Approximating eigenfunctions.}\label{sec:appr-eigenf}
 For every $b=(b_1,b_2)\in \R^2$, we define
\[
\theta_b:\R^2\setminus\{b\}\to [0,2\pi)\]
 as 
\begin{equation}\label{eq:theta_a}
\theta_b(x_1,x_2)=
\begin{cases}
  \arctan\frac{x_2-b_2}{x_1-b_1},&\text{if }x_1>b_1,\ x_2\geq b_2,\\ 
  \frac\pi2,&\text{if }x_1=b_1,\ x_2> b_2,\\ 
  \pi+\arctan\frac{x_2-b_2}{x_1-b_1},&\text{if }x_1<b_1,\\ 
  \frac32\pi,&\text{if }x_1=b_1,\ x_2< b_2,\\ 
  2\pi+\arctan\frac{x_2-b_2}{x_1-b_1},&\text{if }x_1>b_1,\ x_2< b_2,\\ 
\end{cases}
\end{equation}
so that 
\[
\theta_b(b+r(\cos t,\sin t))=t\quad\text{for all }r>0\text{ and
}t\in[0,2\pi).
\] 
For all $a\in\Omega$, let $\varphi_a\in
H^{1,a}_{0}(\Omega,\C)\setminus\{0\}$ be an eigenfunction of problem
\eqref{eq:eige_equation_a} associated to the eigenvalue $\lambda_a$,
i.e. solving
\begin{equation}\label{eq:equation_a}
 \begin{cases}
   (i\nabla + A_a)^2 \varphi_a = \lambda_a \varphi_a,  &\text{in }\Omega,\\
   \varphi_a = 0, &\text{on }\partial \Omega,
 \end{cases}
\end{equation}
such that 
\begin{equation}\label{eq:6}
  \int_\Omega |\varphi_a(x)|^2\,dx=1 \quad\text{and}\quad 
  \int_\Omega e^{\frac i2(\theta_0-\theta_a)(x)}\varphi_a(x)\overline{\varphi_0(x)}\,dx\text{ is a
    positive real number},
\end{equation}
where $\varphi_0$ is as in (\ref{eq:equation_lambda0}--\ref{eq:83});
we observe that, given an eigenfunction $v$ of $(E_a)$ associated to
$\lambda_a$, to obtain an eigenfunction $\varphi_a$ satisfying the
normalization conditions \eqref{eq:6} it is enough to consider
$(\int_\Omega |v|^2\,dx)^{-1}e^{i\vartheta}v$ where
$\vartheta=\arg\left[ \big(\int_\Omega |v|^2\,dx\big)\big(\int_\Omega
  e^{i(\theta_0-\theta_a)/2}v\overline{\varphi_0}\,dx\big)^{-1}\right]$.
Using \eqref{eq:1}, \eqref{eq:equation_lambda0},
\eqref{eq:equation_a}, \eqref{eq:6}, and standard elliptic estimates,
it is easy to prove that
\begin{equation}\label{eq:21}
\varphi_a\to \varphi_0\quad \text{in }H^1(\Omega,\C)\text{ and in
}C^2_{\rm loc}(\Omega\setminus\{0\},\C)
\end{equation}
and 
\begin{equation}\label{eq:22}
  \int_{\Omega} \abs{(i\nabla+A_a)\varphi_a(x)}^2\,dx\to 
  \int_{\Omega} \abs{(i\nabla+A_0)\varphi_0(x)}^2\,dx
\end{equation}
as $a\to 0$. We notice that \eqref{eq:21} and \eqref{eq:22} imply that 
\begin{equation}\label{eq:45}
  (i\nabla+A_a)\varphi_a\to (i\nabla+A_0)\varphi_0\quad\text{in }L^2(\Omega,\C).
\end{equation}

\subsection{Local asymptotics of eigenfunctions.}
We recall from \cite{FFT} the description of the asymptotics at the
singularity of solutions to elliptic equations with Aharonov-Bohm
potentials. In the case of Aharonov-Bohm potentials with circulation
$\frac12$, such asymptotics is described in terms of eigenvalues and
eigenfunctions of the following operator $\mathfrak L$ acting on
$2\pi$-periodic functions
\begin{equation}\label{eq:frakl}
\mathfrak L\psi=
-\psi''+i\psi'+\frac14\psi.
\end{equation}
It is easy to verify that the eigenvalues of $\mathfrak L$ are
$\big\{\frac {j^2}4:j\in \N,\ j\text{ is odd}\big\}$; moreover each
eigenvalue $\frac{j^2}4$ has multiplicity $2$ and an
$L^2((0,2\pi),\C)$-orthonormal basis of the eigenspace associated to
the eigenvalue $\frac{j^2}4$ is formed by the functions
\begin{equation}\label{eq:9}
  \psi_1^j(t)=\frac{e^{i\frac t2}}{\sqrt{\pi}}\cos\Big(\frac j2
  t\Big),\quad 
  \psi_2^j(t)=\frac{e^{i\frac t2}}{\sqrt{\pi}}\sin\Big(\frac j2
  t\Big).
\end{equation}

\begin{Proposition}[\cite{FFT}, Theorem 1.3]\label{prop:fft}
 Let $\Omega \subset \R^2$ be a bounded open set containing   $b$ and
 $h\in L^\infty_{\rm loc}(\Omega\setminus\{0\},\C)$ such that 
 \begin{equation}\label{eq:3}
   \abs{h(x)}=O(\abs{x}^{-2+\eps})\quad  \text{as }|x|\to 0 \quad
   \text{for some }\eps>0.
 \end{equation}
Let $u \in H^{1,b}(\Omega,\C)$ be a
 nontrivial weak solution to the problem
 \begin{equation}\label{eq:2}
   (i\nabla+A_b)^2 u = hu,  \quad \text{in }\Omega,
 \end{equation}
i.e. 
 \begin{equation}\label{eq:8}
   \int_\Omega (i\nabla u+A_b u)\cdot  \overline{(i\nabla v+A_b
     v)}\,dx=\int_\Omega hu\overline{ v}\,dx
   \quad\text{for all }v\in H^{1,b}_{0}(\Omega,\C).
\end{equation}
Then there exists an odd $j\in\N$ such that 
\begin{equation}\label{eq:4}
  \lim_{r\to 0^+} \dfrac{r\int_{D_r(b)} \big(\abs{(i\nabla+A_b)u(x)}^2 -
    (\Re{h(x)})\abs{u(x)}^2\big)\,dx}{\int_{\partial D_r(b)} \abs{u}^2\, ds} =
  \frac j2.
\end{equation}
Furthermore, 
\begin{equation}\label{eq:5}
  r^{-j/2} u(b+r(\cos t,\sin t)) \to 
  \beta_j^1(b,u,h) \psi_1^j(t)+\beta_j^2 (b,u,h) \psi_2^j(t) \quad \text{in }C^{1,\alpha}([0,2\pi],\C)
\end{equation}
as $r\to0^+$ for any $\alpha\in (0,1)$, where, for $\ell=1,2$,
\begin{multline}\label{def_beta}
  \beta_j^\ell(b,u,h)= \int_0^{2\pi}\bigg[(R^{-\frac j2} u(b+R(\cos
  t,\sin t)) \\+ \int_0^R \frac{h(b+s(\cos t,\sin t)) u(b+s(\cos
    t,\sin t))}{j} \bigg( s^{1-\frac j2} -\frac{s^{1+\frac j2}}{R^{j}
  } \bigg)\,ds \bigg] \overline{\psi_\ell^j(t)}\,dt
\end{multline}
for all $R>0$ such that $\{x\in\R^2:|x-b|\leq
R\}\subset \Omega$ and 
$(\beta_j^1(b,u,h) ,\beta_j^2(b,u,h)) \neq (0,0)$.
\end{Proposition}

\noindent From Proposition \ref{prop:fft} we have that, under
assumption \eqref{eq:37}, 
\begin{equation*}
  r^{-k/2} \varphi_0(r(\cos t,\sin t)) \to 
  \frac{e^{i\frac
      t2}}{\sqrt{\pi}}\Big(\beta_k^1(0,\varphi_0,\lambda_0)
  \cos\big(\tfrac k2 t\big)+\beta_k^2 (0,\varphi_0,\lambda_0)
  \sin\big(\tfrac k2 t\big)\Big)  
\end{equation*}
in $C^{1,\alpha}([0,2\pi],\C)$ as $r\to0^+$ for any $\alpha\in (0,1)$ 
with 
\begin{equation}\label{eq:70}
  (\beta_k^1(0,\varphi_0,\lambda_0),
  \beta_k^2(0,\varphi_0,\lambda_0))\neq(0,0),
\end{equation}
where $\beta_k^\ell(0,\varphi_0,\lambda_0)$ are defined as in
\eqref{def_beta}.
We observe that, from \cite{HHOO99} (see also \cite[Lemma 2.3]{BNNNT}),
the function 
\begin{equation*}
  e^{-i\frac t2}\varphi_0(r(\cos t,\sin t))
\end{equation*}
is a multiple of a real-valued function and therefore 
\begin{equation}\label{eq:77}
\text{either
$\beta_k^1(0,\varphi_0,\lambda_0)=0$ or
$\tfrac{\beta_k^2(0,\varphi_0,\lambda_0)}{\beta_k^1(0,\varphi_0,\lambda_0)}$
is real.}
\end{equation}
Since \eqref{eq:70} and \eqref{eq:77} hold, the function  
\[
t\mapsto\beta_k^1(0,\varphi_0,\lambda_0) \cos\big(\tfrac
{k}2t\big)+\beta_k^2 (0,\varphi_0,\lambda_0) \sin\big(\tfrac
{k}2t\big)
\]
 has exactly $k$ zeroes
$t_1,t_2,\dots,t_k$ in $[0,2\pi)$. Up to a change of coordinates in
$\R^2$, it is not restrictive to assume that $0\in\{t_1,t_2,\dots,t_k\}$, i.e. to assume that 
\begin{equation}\label{eq:54}
  \beta_k^1(0,\varphi_0,\lambda_0)=0.
\end{equation}
\begin{remark}\label{rem_beta1=0}
  Condition \eqref{eq:54} can be interpreted as a suitable change of
  the cartesian coordinate system $(x_1,x_2)$ in $\R^2$: we rotate the
  axes in such a way that the positive $x_1$-axis is tangent to one of
  the $k$ nodal lines of $\varphi_0$ ending at $0$ (see \cite[Theorem
  1.5]{NT} for the description of nodal lines of eigenfunctions near
  the pole). It is easy to verify that, besides the alignment of a
  nodal line of $\varphi_0$ along the $x_1$-axis, such a change of
  coordinates has also the effect of rotating the vector
  $(\beta_k^1(0,\varphi_0,\lambda_0),
  \beta_k^2(0,\varphi_0,\lambda_0))$; hence, since in the asymptotics
  stated in Theorem \ref{t:main_asy_eige} only the norm of such a
  vector is involved, it is enough to prove the theorem for
  $\beta_k^1(0,\varphi_0,\lambda_0)=0$.
\end{remark}

By Proposition \ref{prop:fft}, under conditions \eqref{eq:70} and \eqref{eq:54},
$\beta_k^2(0,\varphi_0,\lambda_0)\neq 0$ can be also  characterized as 
\begin{equation}\label{eq:84}
  \beta_k^2(0,\varphi_0,\lambda_0)=\frac1{\sqrt{\pi}}\lim_{r\to 0^+}
  r^{-k/2}\int_0^{2\pi} \varphi_0(r(\cos t,\sin t)) 
e^{-i\frac t2}\sin\big(\tfrac k2 t\big)\,dt.
\end{equation}

\subsection{Fourier coefficients of angular components of solutions.}\label{sec:four-coeff-angul}
Let $U\subseteq\R^2$ be an open set, $b\in U$ and $u\in H^{1,b}(U,\C)$
be a weak solution (in the sense of \eqref{eq:8}) to the problem
\begin{equation}\label{eq_generica}
   (i\nabla + A_b)^2 u = \lambda u,\quad   \text{in }U,
\end{equation}
for some $\lambda\in \R$. 
If $b\in \R^2$ is of the form $b=(|b|,0)$, letting $\theta_b$
   as in \eqref{eq:theta_a}, we have  that 
$\theta_b\in C^\infty(\R^2\setminus \{(x_1,x_2)\in\R^2:
x_2=0, x_1\geq |b|\})$
 and
$\nabla\theta_b$ can be extended to be in $C^\infty(\R^2\setminus \{b\})$ with 
 $\nabla \big(\frac{\theta_b}{2}\big)=A_b$ in $\R^2\setminus \{b\}$.

 Let $b=(|b|,0)\in U$ and let $u\in H^{1,b}(U,\C)$ be a weak solution to
 \eqref{eq_generica}. Let $R>0$ be such that $R>|b|$ and $D_R\subset U$.  For
 $\ell\in\{1,2\}$ and $j$ odd natural number we define, for all
 $r\in(|b|,R)$,
\begin{equation}\label{coeff_fourier_centrata0}
  v_\ell^j(r): = \int_0^{2\pi}  u(r(\cos t,\sin t))e^{-\frac i2
    \theta_b(r\cos t,r\sin t)}e^{i\frac t2} \overline{\psi_\ell^j(t)}\,dt. 
\end{equation}
We notice that $\{v_\ell^j(r)\}_{j,\ell}$ are the Fourier coefficients of  the
function 
\[
t\mapsto  u(r(\cos t,\sin t)) e^{-\frac i2
  (\theta_b-\theta_0)(r\cos t,r\sin t)}
\]
with respect to the orthonormal basis of the space of
periodic-$L^2((0,2\pi),\C)$ functions given in \eqref{eq:9}.  Since
the function $w=u e^{-\frac i2 \theta_b}$ solves $-\Delta w=\lambda w$
in $\R^2\setminus\{(x_1,0):x_1\geq|b|\}$ and jumps to its opposite
across the crack $\{(x_1,0):x_1\geq|b|\}$ (as well as its derivative
$\frac{\partial w}{\partial x_2}$), we have that $v_\ell^j$ is a
solution to the equation
\[ 
-(v_\ell^j )'' - \frac1r (v_\ell^j )' + \frac{j^2}{4r^2}v_\ell^j  =
\lambda v_\ell^j,  \quad \text{in }(|b|,R)
\]
which can be rewritten as
\begin{equation}\label{eq:fourier_centrata0}
  -\left(r^{1+j}\left(r^{-\frac j2}v_\ell^j \right)'\right)'= \lambda\,
  r^{1+\frac j2}v_\ell^j,  \quad \text{in }(|b|,R).
\end{equation}

\section{Poincar\'{e} type inequalities}\label{sec:3}

In this section we establish the validity of some Poincar\'{e} type
inequalities uniformly with respect to varying poles.

\begin{Lemma}[Poincar\'{e} inequality]\label{poincare_inequality}
  Let $r>0$ and $a\in D_r$. For any $u\in H^{1,a}(\Omega,\C)$ the
  following inequality holds true
\begin{equation}
  \dfrac{1}{r^2} \int_{D_r} \abs{u}^2\,dx \leq \dfrac1r \int_{\partial
    D_r} \abs{u}^2 \,ds+ 
  \int_{D_r} \abs{(i\nabla +A_{ a})u}^2\,dx.
\end{equation} 
\end{Lemma}
\begin{proof}
  From the Divergence Theorem, the Young inequality, and the
  diamagnetic inequality \eqref{eq:diamagnetic}, it follows that
\begin{align*}
  \frac2{r^2}\int_{D_r} \abs{u}^2\,dx
  &=\frac1{r^2}\int_{D_r}\Big(\Div(|u|^2x)-2|u|\nabla|u|\cdot x\Big)\,dx\\
  &=\frac1r \int_{\partial D_r} \abs{u}^2 \,ds-\frac2{r^2}
  \int_{D_r}|u|\nabla |u|\cdot x\,dx\\
  &\leq \frac1r \int_{\partial D_r} \abs{u}^2 \,ds+\int_{D_r} |\nabla
  |u||^2\,dx+\frac1{r^2}\int_{D_r} \abs{u}^2\,dx\\
  &\leq \frac1r \int_{\partial D_r} \abs{u}^2 \,ds+\int_{D_r}
  |(i\nabla +A_a) u|^2\,dx+\frac1{r^2}\int_{D_r} \abs{u}^2\,dx
\end{align*}
which yields the conclusion.
\end{proof}

For every $b\in D_1$ we define
\begin{equation}\label{eq:m_b}
  m_{b}: = \inf_{\stackrel{v\in H^{1,b}(D_1,\C)}{v\not\equiv 0}} 
  \dfrac{\int_{D_1} \abs{(i\nabla +A_{b})v}^2\,dx}{\int_{\partial D_1}
    \abs{v}^2\,ds}.
\end{equation} 
\begin{Lemma}\label{l:mbattainedpositive}
  For every $b\in D_1$, the infimum $m_b$ defined in \eqref{eq:m_b} is
  attained and $m_b>0$.
\end{Lemma}
\begin{proof}
 Let $v_n$ be a minimizing sequence such that
  \[ 
  \int_{\partial D_1} \abs{v_n}^2\,dx=1 \quad \text{and} \quad
  \int_{D_1} \abs{(i\nabla +A_{b})v_n}^2 =m_b+o(1)\quad\text{as
  }n\to\infty.
\] 
Then, by Lemma \ref{poincare_inequality}, we have that
$\{v_n\}_{n\in\N}$ is bounded in $H^{1,b}(D_1,\C)$. Hence there exists
a subsequence $v_{n_k}$ converging to some $v\in H^{1,b}(D_1,\C)$
weakly in $H^{1,b}(D_1,\C)$ and (by compactness of the trace embedding
$H^{1,b}(D_1,\C)\hookrightarrow L^2(\partial D_1,\C)$) strongly in
$L^2(\partial D_1,\C)$. Strong convergence in $L^2(\partial D_1,\C)$
implies that $\int_{\partial D_1} \abs{v}^2\,dx=1$, so that
$v\not\equiv 0$; moreover weak lower semicontinuity of the
$H^{1,b}(D_1,\C)$-norm implies that $v$ attains $m_b$.

If, by contradiction, $m_{b}=0$, then, via the diamagnetic inequality
\eqref{eq:diamagnetic},
\[ 
0= \int_{D_1} \abs{(i\nabla + A_{b}) v }^2 \,dx\geq \int_{D_1}
\abs{\nabla \abs{ v}}^2\, dx
\] 
which implies that $|v| \equiv C$, being $C\geq 0$ a real
constant. Since $v\not\equiv 0$, we have that $C>0$ and then
$\int_{D_1}\frac{|v|^2}{|x-b|^2}\,dx=+\infty$, thus contradicting the
fact that $v\in H^{1,b}(D_1,\C)$.
\end{proof}

\begin{Lemma}\label{lemma_Sobolev_inequality}
  Let $r>0$ and $a\in D_r$. Then
\begin{equation}\label{poincare2}
  \frac{m_{a/r}}{r} \int_{\partial D_r} \abs{u}^2 \,ds\leq \int_{D_r}
  \abs{(i\nabla +A_{a})u}^2\,dx\quad
  \text{for all }u\in H^{1,a}(D_r,\C),
\end{equation}
with $m_{a/r}$ as in \eqref{eq:m_b} with $b=\frac ar$.
\end{Lemma}
\begin{proof}
  It follows from \eqref{eq:m_b} and a standard dilation argument.
\end{proof}

\begin{Lemma}\label{lemma_continuity_mb}
  The function $b\mapsto m_b$, with $m_b$ defined in \eqref{eq:m_b},
  is continuous in $D_1$. Moreover $m_0=\frac12$.
\end{Lemma}
\begin{proof}
 Let us consider $b\in D_1$ and a sequence $\{b_n\}_{n\in\N}$ such
 that $b_n \to b$. From Lemma \ref{l:mbattainedpositive}, $m_{b_n}$
 is attained by some $v_{n} \in H^{1,b_n}(D_1,\C)$ with 
\[
\int_{\partial D_1} \abs{v_{n}}^2\,ds=1\quad\text{and}\quad \int_{D_1}
\abs{(i\nabla +A_{b_n})v_n}^2 =m_{b_n}.
\] 
Moreover, by the definition of $m_b$ and density of smooth functions
vanishing in a neighborhood of $b$ in $H^{1,b}(D_1,\C)$, for every
$\eps>0$ there exists a function $\varphi \in C^\infty_{\rm
  c}({\overline D_1 \setminus \{b\}})$ such that
\[
\dfrac{\int_{D_1} \abs{(i\nabla +A_{b})\varphi}^2\,dx}{\int_{\partial
    D_1} \abs{\varphi}^2\,ds}<m_b+\eps.
\]
Since $b_n \to b$, we have that $\varphi \in H^{1,b_n}(D_1,\C)$ if $n$
is sufficiently large, hence
\[
m_{b_n}\leq \dfrac{\int_{D_1} \abs{(i\nabla
    +A_{b_n})\varphi}^2\,dx}{\int_{\partial D_1} \abs{\varphi}^2\,ds}.
\]
Since $A_{b_n}$ is bounded in the support of $\varphi$ uniformly with
respect to $n$, the Dominated Convergence Theorem ensures that
$\int_{D_1} \abs{(i\nabla +A_{b_n})\varphi}^2\,dx\to \int_{D_1}
\abs{(i\nabla +A_{b})\varphi}^2\,dx$ as $n\to\infty$ and therefore
\[
\limsup_{n\to\infty}m_{b_n}\leq \dfrac{\int_{D_1} \abs{(i\nabla
    +A_{b})\varphi}^2\,dx}{\int_{\partial D_1}
  \abs{\varphi}^2\,ds}<m_b+\eps.
\]
We conclude that
\[
\limsup_{n\to\infty}m_{b_n}\leq m_b.
\]
This in particular implies that $m_{b_n}=\int_{D_1}\abs{(i\nabla +
  A_{b_n})v_{n}}^2$ is bounded and hence, in view of Lemma
\ref{poincare_inequality} and \eqref{eq:hardy}, the sequence $\{v_n\}$ is bounded in
$H^1(D_1,\C)$. Then there exists a subsequence $\{v_{n_k}\}$
converging to some $v\in H^1(D_1,\C)$ weakly in $H^1(D_1,\C)$, a.e. in
$D_1$, and (by the Rellich-Kondrakov Theorem) strongly in $L^2(
D_1,\C)$. Moreover, by compactness of the trace embedding
$H^{1}(D_1,\C)\hookrightarrow L^2(\partial D_1,\C)$ we have that
$\int_{\partial D_1} \abs{v}^2\,dx=1$.

We notice that $A_{ b_{n_k}}v_{n_k} \to A_{b}v $ almost everywhere in
$D_1$ and, for any radius $r\in(0,1-|b|)$ e $n$ sufficiently large,
\begin{align*}
  \nor{A_{ b_n}v_{n}}_{L^2(D_1)}^2
  &= \int_{D_r(b_n)} \abs{A_{ b_n}v_{n}}^2\,dx + \int_{D_1 \setminus D_r(b_n)} \abs{A_{ b_n}v_{n}}^2\,dx  \\
  \notag&= \frac14\int_{D_r(b_n)} \frac{|v_{n}(x)|^2}{|x-b_n|^2}\,dx +
  \int_{D_1 \setminus D_r(b_n)}
  \abs{A_{ b_n}v_{n}}^2\,dx  \\
  \notag&\leq \int_{D_r(b_n)} \abs{(i\nabla + A_{b_n})v_{n}}^2 +
  \int_{D_1 \setminus D_r(b_n)} \abs{A_{ b_n}v_{n}}^2 \leq {\rm const}
\end{align*}
via the Hardy inequality \eqref{eq:hardy}. Therefore $A_{b}v\in
L^2(D_1,\C)$ and $A_{ b_{n_k}}v_{n_k}\rightharpoonup A_{b}v$ weakly in
$L^2(D_1,\C)$.  In particular $v\in H^{1,b}(D_1,\C)$.

Therefore $(i\nabla +A_{ b_{n_k}})v_{n_k} \rightharpoonup
(i\nabla +A_{b})v$ weakly in $L^2(D_1,\C)$. The weak lower semi-continuity of the $L^2$-norm yields
\[ m_{b}  \leq \int_{D_1} \abs{(i\nabla +A_{b})v}^2\,dx 
\leq \liminf_{k\to  \infty} m_{b_{n_k}} 
\leq \limsup_{k\to  \infty} m_{b_{n_k}}  \leq m_{b}.
\]
Then $\lim_{k\to  \infty} m_{b_{n_k}}=m_{b}$. From the Urysohn
property we conclude that $\lim_{n\to  \infty} m_{b_{n}}=m_{b}$ and
hence the function $b\mapsto m_b$ is continuos.

From Lemma \ref{l:mbattainedpositive}, the infimum
\[ 
m_0 = \inf_{\stackrel{v\in H^{1,0}(D_1,\C)}{v\not\equiv 0}} 
\dfrac{\int_{D_1} \abs{(i\nabla +A_{0})v}^2\,dx}{\int_{\partial D_1}
  \abs{v}^2\,ds} 
\]
is attained by a function $v_0\in H^{1,0}(D_1,\C)\setminus\{0\}$,
which weakly solves  $(i\nabla+A_0)^2 v_0 = 0$ in $D_1$ in the sense
of \eqref{eq:8}. From  \cite [Lemma 5.4]{FFT}, we have that
\begin{align*}
  N(v_0,r):= \dfrac{r\int_{D_r}\abs{(i\nabla +
      A_0)v_0}^2\,dx}{\int_{\partial D_r}\abs{v_0}^2\,ds} \quad
  \text{is monotone nondecreasing w.r.t. }r;
\end{align*}
furthermore (see Proposition \ref{prop:fft}) $\lim_{r\to
  0^+}N(v_0,r)\geq\frac12$. Hence $m_0=N(v_0,1)\geq \frac12$.  It is
easy to verify that, letting $\tilde v(r\cos t,r\sin
t)=r^{1/2}e^{i\frac t2}\sin(\frac t2)$, we have that $\tilde v\in
H^{1,0}(D_1,\C)$ and
\[
\frac12=\dfrac{\int_{D_1} \abs{(i\nabla +A_{0})\tilde v}^2\,dx}{\int_{\partial D_1}
  \abs{\tilde v}^2\,ds}\geq m_0,
\]
thus implying $m_0=\frac12$.
The proof is thereby complete. 
\end{proof}

As a direct consequence of Lemma \ref{lemma_continuity_mb}, we obtain
the following result which provides a Poincar\'{e} type inequality
with a control on the best constant which is uniform with respect to
the variation of the pole.

\begin{Corollary}\label{corollary_mb}
  For any $\delta \in(0,\frac12)$, there exists some sufficiently large
  $\mu_\delta>1$ such that $m_b \geq \frac12 -\delta$ for every $b\in
  D_1$ with $|b|<\frac{1}{\mu_\delta}$.
\end{Corollary}
\begin{proof}
 The proof is a straightforward consequence of Lemma
\ref{lemma_continuity_mb}. 
\end{proof}

\section{Limit profile}\label{sec:4}

In the present section we construct the  profile which will
be used to describe the limit of
blowed-up sequences of eigenfunctions with poles approaching $0$ along the half-line
tangent  to a nodal line of $\varphi_0$.

\begin{Lemma}\label{lemma_Phi}
 For every odd natural number $k$ there exists $\Phi_k \in \bigcup_{R>0}H^1(D_R^+)$ (where
 $D_R^+$ denotes the half-disk $\{(x_1,x_2)\in D_R(0):x_2>0\}$)  such that 
\[
\begin{cases}
  -\Delta \Phi_k =0, & \text{in } \R^2_+  \text{ in a distributional sense},\\[3pt]
  \Phi_k =0, &\text{on } s , \\[3pt]
  \frac{\partial \Phi_k}{\partial \nu}=0, &\text{on } \partial
  \R^2_+\setminus s,\\[3pt]
  \Phi_k-\psi_k\in \Di_{s}(\R^2_+),
\end{cases}
\]
where $\nu=(0,-1)$ is the outer normal unit vector on $\partial \R^2_+$. 
\end{Lemma}
\begin{proof}
The function $w_k\in \Di_{s}(\R^2_+)$ minimizing the functional $J_k$
defined in \eqref{eq:Jk}  weakly
solves 
\begin{equation}\label{eq:119}
 \begin{cases}
   -\Delta w_k=0, &\text{in }\R^2_+, \\
   w_k=0, &\text{on }s, \\
   \frac{\partial w_k}{\partial \nu}=-\frac{\partial \psi_k}{\partial
     \nu}, &\text{on }\partial \R^2_+\setminus s.
 \end{cases}
\end{equation}
 Taking 
 \begin{equation}\label{eq:44}
\Phi_k=\psi_k + w_k
\end{equation}
 we reach the conclusion.
\end{proof}

\noindent 
From now on, with a little abuse of notation, $\Phi_k$ will denote the
even extension of the function $\Phi_k$ in the previous Lemma
\ref{lemma_Phi} on the whole $\R^2$,
i.e. $\Phi_k(x_1,-x_2)=\Phi_k(x_1,x_2)$.

Let us now set 
\[
{\mathbf e}=(1,0)
\] 
and define, for every odd natural number $k$, 
\begin{equation}\label{eq:10}
\Psi_k=e^{i\frac{\theta_{\mathbf e}}{2}}\Phi_k,
\end{equation}
where $\theta_{\mathbf e}$ is as in \eqref{eq:theta_a} (with
$b={\mathbf e}$) and $\Phi_k$ is the extension (even in $x_2$) of the
function in Lemma \ref{lemma_Phi}.

We denote as $H^{1 ,{\mathbf e}}_{\rm loc}(\R^2,\C)$ the space of
functions belonging to $H^{1 ,{\mathbf e}}(D_r,\C)$ for all $r>0$, as
$\mathcal D^{1,2}_{s}(\R^2)$ the completion of $C^\infty_{\rm c}(\R^2
\setminus s)$ with respect to the norm $( \int_{\R^2} |\nabla u|^2\,dx
)^{1/2}$ and as $\mathcal D^{1,2}_{\mathbf e}(\R^2)$ the completion of
$C^\infty_{\rm c}(\R^2 \setminus \{{\mathbf e}\})$ with respect to the
norm $( \int_{\R^2} |(i\nabla +A_{\mathbf e}) u|^2\,dx )^{1/2}$.

\begin{Proposition}\label{prop_Psi}
The functions $\Psi_k$ defined in \eqref{eq:10} satisfies the following
properties:
\begin{align}
\label{eq:15}&\Psi_k\in H^{1 ,{\mathbf e}}_{\rm loc}(\R^2,\C);\\
\label{eq:16}&(i\nabla +  A_{\mathbf e})^2 \Psi_k=0\quad\text{in
  $\R^2$ in a weak $H^{1 ,{\mathbf e}}$-sense};\\
\label{eq:17}&\int_{\R^2} \big|(i\nabla + A_{\mathbf e})(\Psi_k -
e^{i\theta_{\mathbf e}/{2}}\psi_k )\big|^2\,dx < +\infty;\\
\label{eq:75}&e^{i\frac{\theta_{\mathbf
      e}(x)}{2}}w_k =\Psi_k(x)-e^{i\frac{\theta_{\mathbf
      e}(x)}{2}}\psi_k(x)=O(|x|^{-1/2}),\quad\text{as }|x|\to+\infty.
\end{align}
\end{Proposition}
\begin{proof}
  Statements (\ref{eq:15}--\ref{eq:16}) follow by direct calculations
  together with the asymptotic expansion of solutions to elliptic
  problems with cracks which is proved in \cite{CDD} and which yields
  that $\Phi_k({\mathbf e}+r(\cos t,\sin t))=O(r^{1/2})$ as $r\to
  0^+$. \eqref{eq:17} follows from Lemma \ref{lemma_Phi} and direct
  calculations.

To prove~\eqref{eq:75}, we write 
\[
\Psi_k=e^{i\frac{\theta_{\mathbf
      e}}{2}}\psi_k+v
\]
 where
 $v=e^{i\frac{\theta_{\mathbf e}}{2}}(\Phi_k-\psi_k)$. We note that
 $w_k=\Phi_k-\psi_k\in \mathcal D^{1,2}_s(\R^2)$ and
 hence $v\in \mathcal D^{1,2}_{\mathbf e}(\R^2)$. Since $w_k$ weakly solves $-\Delta w_k=0$ in
 $\R^2\setminus s_0$, 
 its Kelvin transform $\tilde w_k(x)=w_k(\frac{x}{|x|^2})$ weakly solves
 $-\Delta  \tilde w_k=0$ in $D_1\setminus\{(x_1,0):0\leq x_1<1\}$ and
 vanishes on $\{(x_1,0):0\leq x_1<1\}$. From the asymptotics of
  solutions to elliptic problems with cracks proved in \cite{CDD} it follows
  that $|\tilde w_k(x)|=O(|x|^{1/2})$ as $|x|\to 0^+$, which yields
$|w_k(x)|=O(|x|^{-1/2})$ as $|x|\to +\infty$. Therefore we have that 
\begin{equation}\label{eq:11}
|v(x)|=O(|x|^{-1/2})\quad\text{ as }|x|\to +\infty,
\end{equation}
thus proving \eqref{eq:75}. 
\end{proof}

The following result establishes that $\Psi_k$ is the unique function 
satisfying \eqref{eq:15}, \eqref{eq:16}, and~\eqref{eq:17}.

\begin{Proposition}\label{p:uniqueness}
  If $\Phi\in H^{1,{\mathbf e}}_{\rm loc}(\R^2)$ weakly satisfies
\begin{equation}\label{equation_Psi}
  (i\nabla +  A_{\mathbf e})^2 \Phi=0, \quad \text{in $\R^2 $},
\end{equation}
and 
\begin{equation}\label{eq:14}
  \int_{\R^2} |(i\nabla + A_{\mathbf e})(\Phi - e^{\frac
    i2\theta_{\mathbf e}}\psi_k )|^2 < +\infty,
\end{equation}
then $\Phi=\Psi_k$, with  $\Psi_k$ being the function defined in \eqref{eq:10}. 
\end{Proposition}
\begin{proof}
  Suppose that $\Phi\in H^{1,{\mathbf e}}_{\rm loc}(\R^2)$ satisfies
  \eqref{equation_Psi} and \eqref{eq:14}.  Then, in view of
  \eqref{eq:16}, the difference $\Psi=\Phi-\Psi_k$ weakly solves
  $(i\nabla + A_{\mathbf e})^2 \Psi=0$ in $\R^2$. Moreover from
  \eqref{eq:17} and \eqref{eq:14} it follows that
\begin{equation}
  \label{eq:53}
    \int_{\R^2} |(i\nabla + A_{\mathbf e})\Psi(x)|^2dx < +\infty,
\end{equation}
 which, in view of \eqref{eq:hardy}, implies that 
\begin{equation}\label{eq:per_unica_Psi}
\int_{\R^2} \frac{|\Psi(x)|^2}{|x-\mathbf e|^2}\,dx <+\infty.
\end{equation}
For $R>1$, let $\eta_R:\R^2\to\R$ be a smooth cut-off function such
that
\begin{equation}\label{eq:88}
  \eta_R\equiv 0\text{ in }D_{R/2},\quad 
  \eta_R\equiv 1 \text{ on }\R^2\setminus D_{R},\quad
  |\nabla\eta_R|\leq4/R\text{ in }\R^2.
\end{equation}
Testing the equation for $\Psi$ by $(1-\eta_R)^2\Psi$ we obtain that 
\begin{align*}
  \int_{\R^2} (1-\eta_R)^2|(i\nabla+A_{\mathbf e}) \Psi|^2\,dx &= -2i
  \int_{\R^2} (1-\eta_R) \overline{\Psi} (i\nabla+A_{\mathbf e})
  \Psi\cdot \nabla\eta_R \,dx \\
  & \leq \frac12 \int_{\R^2} (1-\eta_R)^2|(i\nabla+A_{\mathbf e})
  \Psi|^2\,dx+ 2 \int_{\R^2} |\Psi|^2 |\nabla \eta_R|^2\,dx
\end{align*}
which implies that 
\begin{align*}
  \int_{D_{R/2}}|(i\nabla+A_{\mathbf e}) \Psi|^2\,dx&\leq \int_{\R^2}
  (1-\eta_R)^2|(i\nabla+A_{\mathbf e})
  \Psi|^2\,dx  \leq 4 \int_{\R^2} |\Psi|^2 |\nabla \eta_R|^2\,dx\\
  &\leq \dfrac{64}{R^2} \int_{D_{R} \setminus D_{R/2}} |\Psi|^2dx \leq
  64\dfrac{(R+1)^2}{R^2}\int_{D_{R} \setminus D_{R/2}}
  \frac{|\Psi|^2}{|x-\mathbf e|^2}\,dx \to 0
\end{align*}
as $R\to +\infty$ thanks to \eqref{eq:per_unica_Psi}. It follows that
$\int_{\R^2}|(i\nabla+A_{\mathbf e}) \Psi|^2\,dx=0$, which implies
that $\int_{\R^2} |x-\mathbf e|^{-2}|\Psi(x)|^2\,dx=0$ in view of
\eqref{eq:hardy}. Hence $\Psi\equiv 0$ in $\R^2$ and $\Phi=\Psi_k$.
\end{proof}

The following lemma establishes a deep relation between the function
$\Psi_k$ and the constant ${\mathfrak m}_k$ introduced in \eqref{eq:Ik}.
\begin{Lemma}\label{l:legame_compl}
Let  $\Psi_k$ be the function defined in \eqref{eq:10}. Then
\begin{equation}\label{eq:step3}
\pi-
\int_0^{2\pi} \Psi_k(\cos t,\sin
  t)e^{-\frac i2 \theta_{\mathbf e}(\cos t,\sin t)}\sin\big(\tfrac
  k2t\big)\,dt=\frac{4}{k} {\mathfrak m}_k
\end{equation}
with ${\mathfrak m}_k$ as in \eqref{eq:Ik}. 
\end{Lemma}
\begin{proof}
Let $w_k$ be the function introduced in  \eqref{eq:Ik} and
\eqref{eq:119},
extended evenly in $x_2$ to the whole $\R^2$
(i.e. $w_k(x_1,-x_2)=w_k(x_1,x_2)$); from \eqref{eq:119} we have that
$w_k$ is harmonic on $\R^2\setminus s_0$. 
Taking into account  \eqref{eq:9},  \eqref{eq:44}, and
\eqref{eq:psi_k}, we have that 
\begin{multline*}
  -\frac1{\sqrt{\pi}}\bigg( \pi- \int_0^{2\pi} \Psi_k(\cos t,\sin
  t)e^{-\frac i2 \theta_{\mathbf e}(\cos t,\sin t)}\sin\big(\tfrac
  k2t\big)\,dt\bigg)\\
  = \int_0^{2\pi}w_k(\cos t,\sin t)e^{i\frac
    t2}\overline{\psi_2^k(t)}\,dt=\omega(1)
\end{multline*}
where
\begin{equation*}
  \omega(r):= \int_{0}^{2\pi} w_k(r\cos t,r\sin t)e^{i\frac t2}\overline{\psi_2^k(t)}\,dt.
\end{equation*}
As observed in \S \ref{sec:four-coeff-angul}, $\omega(r)$ satisfies,
for some $C_\omega\in\C$,
\[ 
\big( r^{-k/2}\omega(r) \big)'= \dfrac{C_\omega}{r^{1+k}}, \quad
\text{for }r>1.
\]
Integrating the previous equation over $(1,r)$ we obtain that
\[
\dfrac{\omega(r)}{r^{k/2}} - \omega(1) = \frac{C_\omega}{k}\left(
  1-\frac1{r^k} \right),\quad\text{for all }r\geq1.
 \]
 From \eqref{eq:75} it follows that $\omega(r)=O(r^{-1/2})$ as
 $r\to+\infty$, hence, letting $r\to+\infty$ in the previous identity,
 we obtain that necessarily $C_\omega=-k\omega(1)$ and then
\begin{equation}\label{eq:120}
 \omega(r)= \omega(1) r^{-k/2},\quad 
 \omega'(r)= -\frac k2\omega(1) r^{-\frac k2 -1},\quad\text{for all }r\geq1.
\end{equation}
On the other hand, 
\begin{equation}\label{eq:121}
  \omega'(r)= \frac{r^{-\frac k2-1}}{\sqrt\pi} \int_{\partial D_r}
  \frac{\partial w_k}{\partial \nu}\,\psi_k\,ds.
\end{equation} 
Combining \eqref{eq:120} and \eqref{eq:121} we obtain that 
\begin{equation}\label{eq:omega1}
  \omega(1)=-\dfrac{2}{k\sqrt\pi} \int_{\partial D_1} \dfrac{\partial w_k}{\partial \nu}\,\psi_k\,ds.
\end{equation}
Multiplying the equation $-\Delta w_k=0$ (which is weakly satisfied in
$\R^2\setminus s_0$) by $\psi_k$ and integrating by parts on
$D_1\setminus s_0$, we obtain that
\begin{equation}\label{eq:primopezzostep3}
 \int_{\partial D_1} \dfrac{\partial w_k}{\partial \nu}\,\psi_k\,ds = \int_{D_1} \nabla w_k \cdot \nabla \psi_k\,dx,
\end{equation}
whereas multiplying $-\Delta\psi_k=0$ (which is weakly satisfied in
$\R^2\setminus s_0$) by $w_k$ and integrating by parts on
$D_1\setminus s_0$ we obtain that
\begin{equation}\label{eq:secondopezzostep3}
  \int_{\partial D_1} \dfrac{\partial \psi_k}{\partial \nu}\,w_k\,ds - 
  2\int_0^1 \dfrac{\partial_+ \psi_k}{\partial x_2}(x_1,0)\,w_k (x_1,0)\,dx_1  
  = \int_{D_1} \nabla w_k \cdot \nabla \psi_k\,dx.
\end{equation}
Collecting \eqref{eq:primopezzostep3} and \eqref{eq:secondopezzostep3}
we have that
\[ 
 \int_{\partial D_1} \dfrac{\partial w_k}{\partial \nu}\,\psi_k\,ds
= \int_{\partial D_1} \dfrac{\partial \psi_k}{\partial \nu}\,w_k\,ds - 
2\int_0^1 \dfrac{\partial_+ \psi_k}{\partial x_2}(x_1,0)\,w_k (x_1,0)\,dx_1  
.
  \]
Since 
\[
\int_{\partial D_1} \dfrac{\partial \psi_k}{\partial
  \nu}\,w_k\,ds=\frac{k\sqrt\pi}{2}\omega(1),
\]
\eqref{eq:omega1} now reads
\[
\omega(1)=-\omega(1) + \frac{4}{k\sqrt\pi} \int_0^1 \dfrac{\partial_+
  \psi_k}{\partial x_2}(x_1,0)\,w_k (x_1,0)\,dx_1
\] and thus
\begin{equation*}
  \omega(1)= \frac{2}{k\sqrt\pi} \int_0^1 \dfrac{\partial_+ \psi_k}{\partial x_2}(x_1,0)\,w_k (x_1,0)\,dx_1  .
\end{equation*}
Letting ${\mathfrak m}_k$ as in \eqref{eq:Ik}, in view of
\eqref{eq:segno_mk} we conclude that 
\begin{equation*}
  \omega(1)= -\frac{4}{k\sqrt\pi} {\mathfrak m}_k,
\end{equation*}
thus concluding the proof of \eqref{eq:step3}.
  
\end{proof}

\section{Monotonicity formula and energy estimates for blow-up sequences}\label{sec:5}

In this section we prove some 
energy estimates for eigenfunctions using an adaption of the Almgren
monotonicity argument inspired by \cite[Section 5]{NNT} and \cite{FFT}.

\begin{Definition}
  Let $\lambda \in \R$, $b \in \R^2$, and $u \in
  H^{1,b}(D_r,\C)$. For any  $r>|b|$, we define the Almgren-type frequency function as
\[
 N(u,r,\lambda,A_b) = \dfrac{E(u,r,\lambda,A_b)}{H(u,r)},
 \]
where
\begin{align}
  \label{eq:39}& E(u,r,\lambda,A_b) = \int_{D_r} \abs{(i\nabla +
    A_b)u}^2\,dx -
  \lambda \int_{D_r} \abs{u}^2\,dx, \\
\label{eq:40}& H(u,r) = \dfrac1r \int_{\partial D_r} \abs{u}^2\,ds .
\end{align}
\end{Definition}
When we study the quotient $N=E/H$ for any magnetic
eigenfunction, we find several specific relations to hold true.  We
are interested in the derivative of such a
quotient, since it provides some information about the
possible vanishing behavior of eigenfunctions near the pole of the
magnetic potential.

For all $1\leq j\leq n_0$ and $a\in\Omega$, let $\varphi_j^a\in
H^{1,a}_{0}(\Omega,\C)\setminus\{0\}$ be an eigenfunction of problem
\eqref{eq:eige_equation_a} associated to the eigenvalue $\lambda_j^a$,
i.e. solving
\begin{equation}\label{eq_eigenfunction}
 \begin{cases}
   (i\nabla + A_a)^2 \varphi_j^a = \lambda_j^a \varphi_j^a,  &\text{in }\Omega,\\
   \varphi_j^a = 0, &\text{on }\partial \Omega,
 \end{cases}
\end{equation}
such that 
\begin{equation}\label{eq:23}
  \int_\Omega |\varphi_j^a(x)|^2\,dx=1\quad\text{and}\quad 
  \int_\Omega \varphi_j^a(x)\overline{\varphi_\ell^a(x)}\,dx=0\text{ if }j\neq\ell.
\end{equation}
For $j=n_0$, we choose 
\begin{equation}\label{eq:89}
  \varphi_{n_0}^a=\varphi_a,
\end{equation}
with $\varphi_a$ as in \eqref{eq:equation_a}--\eqref{eq:6}.  We
observe that, since $a\in\Omega\mapsto \lambda_j^a$ admits a
continuous extension on $\overline{\Omega}$ as proved in \cite[Theorem
1.1]{BNNNT}, we have that
\begin{equation}
  \Lambda:=\sup_{\substack{
      a\in\Omega\\1\leq j\leq n_0}}\lambda_j^a\in(0,+\infty).\label{eq:19}
\end{equation}

\begin{Lemma}\label{l:Hpos}\quad

\begin{enumerate}[\rm (i)]
\item There exists $R_0\in(0,(5\Lambda)^{-1/2})$ such that
  $D_{R_0}\subset\Omega$ and, if $|a|<R_0$,
\begin{equation*}
  H(\varphi_j^a,r)>0\quad\text{for all }r\in(|a|,R_0) \text{ and }1\leq
  j\leq n_0.
\end{equation*}
\item There exist $C_0>0$ and $\alpha_0\in(0,R_0)$ such that 
\begin{equation*}
H(\varphi_j^a,R_0)\geq C_0\quad\text{for all $a$ with }|a|<\alpha_0 \text{ and }1\leq
j\leq n_0.
\end{equation*}
  \end{enumerate}
\end{Lemma}
\begin{proof}
  To prove (i) we argue by contradiction and assume that, for all $n$
  sufficiently large, there exist $a_n\in\Omega$ with $|a_n|<\frac1n$,
  $r_n\in \big(|a_n|,\frac1n\big)$, and $j_n\in\{1,\dots,n_0\}$  such
  that  $H(\varphi_{j_n}^{a_n},r_n)=0$,
  i.e. $\varphi_{j_n}^{a_n}\equiv 0$ on $\partial D_{r_n}$. Testing
  \eqref{eq_eigenfunction} with $\varphi_{j_n}^{a_n}$ and integrating
  on $D_{r_n}$, in view of Lemma \ref{poincare_inequality} we obtain 
  \begin{equation*}
    0=\int_{D_{r_n}}
  \Big(|(i\nabla + A_{a_n}) \varphi_{j_n}^{a_n}|^2
    - \lambda_{j_n}^{a_n} |\varphi_{j_n}^{a_n}|^2 \Big)\,dx
\geq (1-\Lambda r_n^2)\int_{D_{r_n}}
 |(i\nabla + A_{a_n}) \varphi_{j_n}^{a_n}|^2\,dx.
\end{equation*}
Since $r_n\to 0$, for $n$ sufficiently large $1-\Lambda r_n^2>0$ and
hence the above inequality yields that $\int_{D_{r_n}} |(i\nabla +
A_{a_n}) \varphi_{j_n}^{a_n}|^2\,dx=0$. Lemma
\ref{poincare_inequality} then implies that
$\|\varphi_{j_n}^{a_n}\|_{H^{1,a_n}(D_{r_n},\C)}=0$ and hence
$\varphi_{j_n}^{a_n}\equiv 0$ in $D_{r_n}$. From the unique
continuation principle (see \cite[Corollary 1.4]{FFT}) we conclude
that $\varphi_{j_n}^{a_n}\equiv 0$ in $\Omega$, a contradiction.

To prove (ii), we argue by contradiction and assume that, for all $n$
sufficiently large, there exist $a_n\in\Omega$ with $a_n\to 0$ and
$j_n\in\{1,\dots,n_0\}$ such that
  \begin{equation}\label{eq:26}
    \lim_{n\to\infty}H(\varphi_{j_n}^{a_n},R_0)=0. 
  \end{equation}
  Letting $\varphi_n:=\varphi_{j_n}^{a_n}$ and
  $\lambda_n:=\lambda_{j_n}^{a_n}$, using \eqref{eq_eigenfunction} and
  \eqref{eq:23} it is easy to prove that, along a subsequence,
  $\lambda_{n_k}\to \lambda_{j_0}^0$ for some $j_0\in\{1,\dots,n_0\}$
  and $\varphi_{n_k} \to \varphi$ weakly in $H^1(\Omega,\C)$ for some
  $\varphi\in H^{1,0}_0(\Omega,\C)$ satisfying
\begin{equation}\label{eq:24}
  (i\nabla + A_0)^2 \varphi= \lambda_{j_0}^0 \varphi, \quad
  \text{in }\Omega,
\end{equation}
in a weak sense 
and 
\begin{equation}\label{eq:25}
\int_\Omega |\varphi(x)|^2\,dx=1.
\end{equation}
Furthermore, by \eqref{eq:26} and compactness of the trace embedding
$H^{1}(D_{R_0},\C)\hookrightarrow L^2(\partial D_{R_0},\C)$, we have
that
\[
0=\lim_{k\to\infty}\frac1{R_0}\int_{\partial
  D_{R_0}}|\varphi_{n_k}|^2\,ds=\frac1{R_0}\int_{\partial
  D_{R_0}}|\varphi|^2\,ds,
\]
which implies that $\varphi=0$ on $\partial D_{R_0}$. 
Testing
  \eqref{eq:24} with $\varphi$ and integrating
  on $D_{R_0}$, in view of Lemma \ref{poincare_inequality} we obtain 
  \begin{equation*}
    0=\int_{D_{R_0}}
    \Big(|(i\nabla + A_{0}) \varphi|^2
    - \lambda_{j_0}^{0} |\varphi|^2 \Big)\,dx
    \geq (1-\Lambda R_0^2)\int_{D_{R_0}}
    |(i\nabla + A_{0}) \varphi|^2\,dx.
\end{equation*}
Since $1-\Lambda R_0^2>0$, we deduce that $\int_{D_{R_0}} |(i\nabla +
A_{0}) \varphi|^2\,dx=0$. Lemma \ref{poincare_inequality} then implies
that $\varphi\equiv 0$ in $D_{R_0}$. From the unique continuation
principle (see \cite[Corollary 1.4]{FFT}) we conclude that
$\varphi\equiv 0$ in $\Omega$, thus contradicting \eqref{eq:25}.
\end{proof}

We notice that, thanks to Lemma \ref{l:Hpos}, the function $r\mapsto
N(\varphi_j^a,r,\lambda_j^a,A_a)$ is well defined in $(|a|,R_0)$.

\begin{Lemma}\label{lemma_derivata_H}
  Let $1\leq j\leq n_0$, $a\in\Omega$, and $\varphi_j^a\in
  H^{1,a}_{0}(\Omega,\C)$ be a solution to
  \eqref{eq_eigenfunction}--\eqref{eq:23}. Then $r\mapsto
  H(\varphi_j^a,r)$ is smooth in $(|a|,R_0)$ and
\[ 
\dfrac{d}{dr} H(\varphi_j^a,r) = \dfrac2r
E(\varphi_j^a,r,\lambda_j^a,A_a). 
\]
\end{Lemma}
\begin{proof}
Since the proof is similar to that of \cite[Lemma 5.2]{NNT}, we omit it.    
\end{proof}

\begin{Lemma}\label{lemma_stima_H_sotto}
  For $\delta\in(0,1/4)$, let $\mu_\delta$ be as in Corollary
  \ref{corollary_mb}.  Let $r_0 \leq R_0$ and $j\in\{1,\dots,n_0\}$.
  If $\mu_\delta |a| \leq r_1 < r_2 \leq r_0$ and $\varphi_j^a$ is a
  solution to \eqref{eq_eigenfunction}--\eqref{eq:23}, then
\begin{equation*}
  \frac{H(\varphi_j^a,r_2)}{H(\varphi_j^a,r_1)} 
  \geq e^{-\frac52 \Lambda   r_0^2} \left(\dfrac{r_2}{r_1}\right)^{1-2\delta}.
\end{equation*}
\end{Lemma}
\begin{proof}
  Combining Lemma \ref{poincare_inequality} with Lemma
  \ref{lemma_Sobolev_inequality} and Corollary \ref{corollary_mb} we
  obtain that, for every  $\mu_\delta|a|<r<R_0$,
\[ 
\dfrac{1}{r^2}\int_{D_r} \abs{\varphi_j^a}^2\,dx \leq
\left(1+\frac{2}{1-2\delta}\right) \int_{D_r} \abs{(i\nabla +
  A_a)\varphi_j^a}^2\,dx <5 \int_{D_r} \abs{(i\nabla +
  A_a)\varphi_j^a}^2\,dx.
\] 
From above, Lemma \ref{lemma_derivata_H}, Lemma
\ref{lemma_Sobolev_inequality}, recalling that
$R_0<(5\Lambda)^{-1/2}$, for every $\mu_\delta|a|<r<r_0$ we have that
\begin{align*}
  \dfrac{d}{dr}H(\varphi_j^a,r) &= \frac2r \int_{D_r}
  \left(\abs{(i\nabla + A_a)\varphi_j^a}^2
    - \lambda_j^a \abs{\varphi_j^a}^2 \right)\,dx \\
  &\geq \frac2r \left( 1- 5\Lambda r^2 \right)
  \int_{D_r} \abs{(i\nabla + A_a)\varphi_j^a}^2\,dx \\
  &\geq \frac2r \left( 1-5 \Lambda r^2 \right)
  m_{a/r} H(\varphi_j^a,r)\\
  &\geq \frac2r \left( 1-5 \Lambda r^2 \right) \left(\frac12
    -\delta\right) H(\varphi_j^a,r),
\end{align*}
so that, in view of Lemma \ref{l:Hpos}, 
\[ 
\dfrac{d}{dr} \log H(\varphi_j^a,r) \geq \frac{1-2\delta}{r} -
\Lambda (5-10\delta)r \geq \frac{1-2\delta}{r} -
5\Lambda r.
\]
Integrating between $r_1$ and $r_2$ we obtain the desired inequality.
\end{proof}

\begin{Lemma}\label{lemma:Ma}
  For $1\leq j\leq n_0$ and $a\in\Omega$, let $\varphi_j^a\in
  H^{1,a}_{0}(\Omega,\C)$ be a solution to
  \eqref{eq_eigenfunction}--\eqref{eq:23}. Then, for all $|a|<r<R_0$,
  we have that
 \begin{equation}
   \dfrac{d}{dr} E(\varphi_j^a, r,\lambda_j^a,A_a) 
   = 2\int_{\partial D_r} \abs{(i\nabla +A_a)\varphi_j^a \cdot\nu}^2\,ds
   - \frac2r \left( M_j^a + \lambda_j^a \int_{D_r} \abs{\varphi_j^a}^2\,dx \right)
 \end{equation}
where $\nu(x)=\frac{x}{|x|}$ denotes the unit normal vector to
$\partial D_r$ and
\begin{equation}
  M_j^a = \frac14 \Big( a_1 (
  c_{a,j}^2-
  d_{a,j}^2)+2 a_2 c_{a,j} d_{a,j}\Big),
\end{equation}
with $a=(a_1,a_2)$, $c_{a,j}=\beta_1^1(a,\varphi_j^a,\lambda_j^a)$,
and $d_{a,j}=\beta_1^2(a,\varphi_j^a,\lambda_j^a)$, being
$\beta_1^1(a,\varphi_j^a,\lambda_j^a)$,
$\beta_1^2(a,\varphi_j^a,\lambda_j^a)$ the coefficients defined in
\eqref{def_beta}. Furthermore, letting $\mu_\delta$ as in Corollary
\ref{corollary_mb},
\[
\frac{|M_j^a|}{H(\varphi_j^a,\mu_\delta|a|)}\leq C_\delta,
\]
for some $C_\delta>0$ independent of $a$.
\end{Lemma}
\begin{proof}
Since the proof is similar to that of \cite[Lemmas 5.6, 5.7, 5.8, 5.9]{NNT}, we omit it.    
\end{proof}

\begin{Lemma}\label{l:5.5}
  For $1\leq j\leq n_0$ and $a\in\Omega$, let $\varphi_j^a\in
  H^{1,a}_{0}(\Omega,\C)$ be a solution to
  \eqref{eq_eigenfunction}--\eqref{eq:23}.  Let $\delta\in(0,1/4)$,
  $\mu_\delta$ as in Corollary \ref{corollary_mb} and $r_0 \leq R_0$.
  Then there exists $c_{\delta,r_0}>0$ such that, for all
  $\mu>\mu_\delta$, $|a|<\frac{r_0}\mu$, $\mu|a|\leq r<r_0$, and
  $1\leq j\leq n_0$,
\begin{equation}\label{eq:20}
  e^{\frac{\Lambda r^2}{1 - \Lambda {r_0}^2}} \left(N(\varphi_j^a, r,\lambda_j^a,A_a) +1\right)
  \leq e^{\frac{\Lambda {r_0}^2}{1 - \Lambda {r_0}^2}}
  \left(N(\varphi_j^a, r_0,\lambda_j^a,A_a) +1\right) +
  \frac{c_{\delta,r_0}}{\mu^{1-2\delta}}
\end{equation}
and
\begin{equation}\label{eq:27}
  N(\varphi_j^a, r,\lambda_j^a,A_a)+1>0.
\end{equation}
\end{Lemma}
\begin{proof}
By direct computations and Schwarz inequality (see \cite[Lemma
5.11]{NNT}), we obtain that, for all $|a|<r<R_0$,
  \begin{align*}
    &\dfrac{dN(\varphi_j^a, r,\lambda_j^a,A_a)}{dr}\\
    & = \frac{\frac2r\Big( \big(\int_{\partial
        D_r}|(i\nabla+A_a)\varphi_j^a\cdot\nu|^2\,ds\big)\big(
      \int_{\partial D_r}|\varphi_j^a|^2\,ds\big)-\big(i
      \int_{\partial D_r}(i\nabla+A_a)\varphi_j^a\cdot\nu \,
      \overline{\varphi_j^a}\,ds
      \big)^2\Big)}{H^2(\varphi_j^a, r)}\\
    &\quad - \frac{2}{rH(\varphi_j^a, r)} \bigg( M_j^a + \lambda_j^a
    \int_{D_r} \abs{\varphi_j^a}^2 \,dx\bigg)
    \\
    &\geq - \frac{2}{rH(\varphi_j^a, r)} \bigg( |M_j^a| + \lambda_j^a
    \int_{D_r} \abs{\varphi_j^a}^2 \,dx\bigg).
\end{align*}
Via Lemmas \ref{lemma_stima_H_sotto} and \ref{lemma:Ma} we estimate,
for all $\mu_\delta|a|\leq r<r_0$,
  \[ 
  \dfrac{2|M_j^a|}{H(\varphi_j^a, r)} =
  2\dfrac{|M_j^a|}{H(\varphi_j^a,\mu_\delta|a| )} \,
  \dfrac{H(\varphi_j^a,\mu_\delta|a| )}{H(\varphi_j^a,r)} \leq {\rm
    const}_\delta \bigg(\frac{|a|}{r}\bigg)^{\!\!1-2\delta},
 \]
 where ${\rm const}_\delta >0$ is independent of $a$ (but depends on
 $\delta$).  On the other hand, by Lemma \ref{poincare_inequality} we
 have that, for all $\mu_\delta|a|\leq r<r_0$,
  \[ 
  \dfrac{1-\Lambda r^2}{r^2} \int_{D_r} \abs{\varphi_j^a}^2 \leq
  H(\varphi_j^a,r) + E(\varphi_j^a, r,\lambda_j^a,A_a)
\] 
which implies
\[
\frac{2\lambda_j^a}{rH(\varphi_j^a,r)} \int_{D_r} |\varphi_j^a|^2\,dx
\leq \frac{2\Lambda r}{1-\Lambda {r_0}^2} \Big(N(\varphi_j^a,
r,\lambda_j^a,A_a)+1\Big).
\]
Therefore \eqref{eq:27} follows. Moreover,  for all $\mu_\delta|a|\leq r<r_0$,
\[ 
\dfrac{dN(\varphi_j^a, r,\lambda_j^a,A_a)}{dr} \geq - {\rm
  const}_\delta\frac{|a|^{1-2\delta}}{r^{2-2\delta}} - \frac{2\Lambda
  r}{1-\Lambda {r_0}^2} \Big(N(\varphi_j^a,
r,\lambda_j^a,A_a)+1\Big)
 \] 
which is read as
\[ 
\left( e^{\frac{\Lambda r^2}{1 - \Lambda {r_0}^2}}
  \left(N(\varphi_j^a, r,\lambda_j^a,A_a) +1\right)
\right)'e^{-\frac{\Lambda r^2}{1 - \Lambda {r_0}^2}} \geq - {\rm
  const}_\delta\dfrac{|a|^{1-2\delta}}{r^{2-2\delta}}.
\] 
Letting $r\in[\mu_\delta|a|,r_0)$ and integrating from $r$ to $r_0$ we obtain
\[
e^{\frac{\Lambda r^2}{1 - \Lambda {r_0}^2}} \left(N(\varphi_j^a, r,\lambda_j^a,A_a) +1\right)
\leq e^{\frac{\Lambda {r_0}^2}{1 - \Lambda {r_0}^2}}
\left(N(\varphi_j^a, r_0,\lambda_j^a,A_a) +1\right) +e^{\frac{\Lambda {r_0}^2}{1 - \Lambda {r_0}^2}}\frac{ {\rm
  const}_\delta}{1-2\delta}\bigg(\frac{|a|}{r}\bigg)^{\!\!1-2\delta}. 
\]
Letting $\mu>\mu_\delta$, $|a|<\frac{r_0}\mu$ and $\mu|a|\leq r<r_0$,
and taking $c_{\delta,r_0}=e^{\frac{\Lambda {r_0}^2}{1 - \Lambda
    {r_0}^2}}\frac{ {\rm const}_\delta}{1-2\delta}$, the above
estimates yields \eqref{eq:20}.
\end{proof}

A first consequence of Lemma \ref{l:5.5} is the following estimate of
the Almgren quotient of $\varphi_a$ at radii of size $|a|$ in terms of
the order of vanishing of $\varphi_0$ at the pole.
\begin{Lemma}\label{lemma_limitatezza_N_per_blowup}
  For $a\in\Omega$, let $\varphi_a\in H^{1,a}_{0}(\Omega,\C)$ be a
  solution to (\ref{eq:equation_a}-\ref{eq:6}).  For every
  $\delta\in(0,1/4)$ there exist $r_\delta>0$ and
  $K_\delta>\mu_\delta>0$ such that if $\mu\geq K_\delta$,
  $|a|<\frac{r_\delta}\mu$, and $\mu|a|\leq r<r_\delta$
\begin{equation*}
N(\varphi_a, r,\lambda_a,A_a)\leq \frac{k}2+\delta. 
\end{equation*}
\end{Lemma}
\begin{proof}
  Let $p>0$ be sufficiently small so that $p(2+\frac k2+\frac
  p2)<\frac12$. Let $\delta\in(0,\frac14)$.  Since, in view of
  Proposition \ref{prop:fft},
\[
\lim_{r\to 0^+}N(\varphi_0, r,\lambda_0,A_0)=\frac k2,
\]
we can choose $r_\delta>0$ sufficiently small so that $r_\delta<
\min\big\{R_0, (5\Lambda)^{-1/2}\big\}$, $e^{\frac{\Lambda
    {r_\delta}^2}{1 - \Lambda {r_\delta}^2}}\leq 1+\delta p$, and
$N(\varphi_0, r_\delta,\lambda_0,A_0)<\frac k2+\delta p$.

Since, in view of \eqref{eq:21} and \eqref{eq:45}, $N(\varphi_a,
r_\delta,\lambda_a,A_a)\to N(\varphi_0, r_\delta,\lambda_0,A_a)$ as
$|a|\to0$, there exists some $\alpha_\delta>0$ such that if $|a|
<\alpha_\delta$ then $N(\varphi_a, r_\delta,\lambda_a,A_a) <\frac
k2+\delta p$. From Lemma \ref{l:5.5} it follows that, if
$\mu>\mu_\delta$, $|a|<\min\{\frac{r_\delta}{\mu},\alpha_\delta\}$,
and $\mu|a|\leq r<r_\delta$, then
\begin{align*}
  N(\varphi_a, r,\lambda_a,A_a)+1&\leq (1+\delta p)\big(\tfrac
  k2+\delta
  p+1\big)+\frac{c_{\delta,r_\delta}}{\mu^{1-2\delta}}\\
  &=1+\frac k2+\delta\Big(2p+p\frac k2+\delta
  p^2\Big)+\frac{c_{\delta,r_\delta}}{\mu^{1-2\delta}}<1+\frac
  k2+\frac12\delta+\frac{c_{\delta,r_\delta}}{\mu^{1-2\delta}}.
\end{align*}
If $K_\delta>\max\big\{\mu_\delta,\big(\frac{2c_{\delta,r_\delta}}\delta\big)^{\frac1{1-2\delta}},
r_\delta/\alpha_\delta\big\}$, we 
conclude that,  if
$\mu\geq K_\delta$, $|a|<\frac{r_\delta}{\mu}$, and
$\mu|a|\leq r<r_\delta$, then $N(\varphi_a, r,\lambda_a,A_a)<\frac
k2+\delta$, thus concluding the proof.
\end{proof}

A second consequence of Lemma \ref{l:5.5} is the following estimate of
the energy of eigenfunctions $\varphi_j^a$ in disks of radius of order $|a|$.

\begin{Lemma}\label{l:sec_con}
  For $1\leq j\leq n_0$ and $a\in\Omega$, let $\varphi_j^a\in
  H^{1,a}_{0}(\Omega,\C)$ be a solution to
  \eqref{eq_eigenfunction}--\eqref{eq:23}. Let $R_0$ be as in Lemma \ref{l:Hpos}.
  For every $\delta\in(0,1/4)$, there exist $\tilde K_\delta>1$ and
  $\tilde C_\delta>0$ such that, for all $\mu\geq \tilde K_\delta$,
  $a\in\Omega$ with $|a|<\frac{R_0}\mu$, and $1\leq j\leq n_0$,
\begin{align}
\label{eq:34}&\int_{\partial D_{\mu|a|}}|\varphi_j^a|^2\,ds\leq \tilde C_\delta (\mu|a|)^{2-2\delta},\\
\label{eq:35}&\int_{D_{\mu|a|}}|(i\nabla + A_a)\varphi_j^a|^2\,dx\leq
\tilde C_\delta (\mu|a|)^{1-2\delta},\\
\label{eq:36}&\int_{D_{\mu|a|}}|\varphi_j^a|^2\,dx\leq
\tilde C_\delta (\mu|a|)^{3-2\delta}.
\end{align}
\end{Lemma}
\begin{proof}
  Let us fix $\delta\in(0,1/4)$ and let $\mu_\delta$ be as in Corollary
  \ref{corollary_mb}. From Lemma \ref{l:5.5} it follows that, if
  $\mu>\mu_\delta$ and $|a|<\frac{R_0}{\mu}$ then, for all $1\leq
  j\leq n_0$,
\begin{equation}\label{eq:28}
  N(\varphi_j^a, \mu|a|,\lambda_j^a,A_a) 
  \leq e^{\frac{\Lambda {R_0}^2}{1 - \Lambda {R_0}^2}}
  \left(N(\varphi_j^a, R_0,\lambda_j^a,A_a) +1\right) +
  \frac{c_{\delta,R_0}}{\mu_\delta^{1-2\delta}}-1 .
\end{equation}
From \eqref{eq_eigenfunction}, \eqref{eq:23}, and \eqref{eq:19} we
deduce that 
\begin{equation}\label{eq:33}
\int_{D_{R_0}} \abs{(i\nabla + A_a)\varphi_j^a}^2\,dx\leq 
\int_{\Omega} \abs{(i\nabla + A_a)\varphi_j^a}^2\,dx=\lambda_j^a\leq
\Lambda,
\end{equation}
therefore, in view of Lemma \ref{l:Hpos}, if $|a|<\alpha_0$,
\begin{equation}\label{eq:29}
  N(\varphi_j^a, R_0,\lambda_j^a,A_a)=
  \frac{\int_{D_{R_0}} |(i\nabla + A_a)\varphi_j^a|^2\,dx -
    \lambda_j^a \int_{D_{R_0}}|\varphi_j^a|^2\,dx}
  {H(\varphi_j^a, R_0)}\leq \frac{2\Lambda}{C_0}.
\end{equation}
Combining \eqref{eq:28} and \eqref{eq:29} we obtain that, if
$\mu\geq \tilde K_\delta$ with
$\tilde K_\delta>\max\{\mu_\delta,R_0/\alpha_0\}$ and
$|a|<\frac{R_0}\mu$, then
\[
\int_{D_{\mu|a|}} |(i\nabla + A_a)\varphi_j^a|^2\,dx - \lambda_j^a
\int_{D_{\mu|a|}}|\varphi_j^a|^2\,dx
\leq {\rm const}_\delta H(\varphi_j^a, \mu|a|)
\]
for some positive ${\rm const}_\delta>0$ depending on $\delta$. Hence,
from Lemma \ref{poincare_inequality},
\[
(1-\Lambda\mu^2|a|^2) \int_{D_{\mu|a|}} |(i\nabla +
A_a)\varphi_j^a|^2\,dx -\Lambda(\mu|a|)^2 H(\varphi_j^a, \mu|a|)\leq {\rm const}_\delta H(\varphi_j^a, \mu|a|)
\]
which implies 
\begin{equation}\label{eq:30}
  \int_{D_{\mu|a|}} |(i\nabla +
  A_a)\varphi_j^a|^2\,dx\leq \frac{\Lambda R_0^2+{\rm
      const}_\delta}{1-\Lambda R_0^2} H(\varphi_j^a, \mu|a|).
\end{equation}
From Lemma \ref{lemma_stima_H_sotto} it follows that, if
$\mu\geq\tilde K_\delta$ and $|a|<\frac{R_0}\mu$,
\begin{equation}\label{eq:31}
  H(\varphi_j^a,\mu|a|)\leq e^{\frac52 \Lambda R_0^2}
  \left(\dfrac{\mu|a|}{R_0}\right)^{1-2\delta} H(\varphi_j^a,R_0).
\end{equation}
On the other hand, Lemma \ref{lemma_Sobolev_inequality}, Corollary
\ref{corollary_mb}, and \eqref{eq:33} yield
\begin{equation}\label{eq:32}
  H(\varphi_j^a,R_0)\leq \frac1{m_{a/R_0}}\int_{D_{R_0}} \abs{(i\nabla
    + A_a)\varphi_j^a}^2\,dx\leq \frac{2 \Lambda}{1-2\delta}.
\end{equation}
Estimate \eqref{eq:34} follows combining \eqref{eq:31}, and
\eqref{eq:32}, whereas estimate \eqref{eq:35} follows from
\eqref{eq:30}, \eqref{eq:31}, and \eqref{eq:32}.  Finally,
\eqref{eq:36} can be deduced from \eqref{eq:34}, \eqref{eq:35} and
Lemma \ref{poincare_inequality}.
\end{proof}

We blow-up the family of eigenfunctions $\{\varphi_a\}$ with
$a=(|a|,0)$ as $|a|\to0$, i.e. we introduce the family of functions
\begin{equation}\label{def_blowuppate_normalizzate}
 \tilde \varphi_a (x) :=
 \dfrac{\varphi_a(|a|x)}{\sqrt{H(\varphi_a,K_\delta |a|)}},\quad
 a=(|a|,0)=|a|\mathbf e,
\end{equation}
with  $K_\delta$ being as in Lemma \ref{lemma_limitatezza_N_per_blowup}
for some fixed $\delta\in(0,1/4)$. We observe that $\tilde \varphi_a$ weakly
solves
\begin{equation}\label{eq:equations_for_blowedup}
  (i\nabla + A_{\mathbf e})^2 \tilde\varphi_a =|a|^2 \lambda_a
  \tilde\varphi_a,  \quad\text{in }\tfrac1{|a|}\Omega
  =\{x\in\R^2:|a|x\in\Omega\},
\end{equation}
and 
\begin{equation}\label{eq:62}
  \frac{1}{K_\delta}\int_{\partial D_{K_\delta}}|\tilde\varphi_a|^2\,ds=1.
\end{equation}
In section \ref{sec:blow-up-analysis} we will prove that
$\tilde\varphi_a$ converges to a limit profile which is a multiple of
the function $\Psi_k$ introduced in \eqref{eq:10}. To this aim, the
energy estimates below will play a crucial role.

\begin{Theorem}\label{t:stime_blowup}
For all $R\geq K_\delta$,  
\begin{equation}\label{eq:67}
\text{the family of functions }\big\{\tilde
 \varphi_a: a=|a|{\mathbf e}, |a|<\tfrac{r_\delta}{R}\big\}
 \text{ is bounded in $H^{1 ,{\mathbf e}}(D_R,\C)$}.
\end{equation}
In particular, for all $R\geq K_\delta$,  
\begin{align}
\label{eq:46}&\int_{D_{R|a|}} \abs{(i\nabla +
  A_a)\varphi_a}^2dx=O(H(\varphi_a,
K_\delta|a|)),\quad\text{as }|a|\to0^+,\\
\label{eq:47}&\int_{\partial D_{R|a|}} |\varphi_a|^2dx=O(|a|H(\varphi_a,
K_\delta|a|)),\quad\text{as }|a|\to0^+,\\
\label{eq:48}&\int_{D_{R|a|}} |\varphi_a|^2dx=O(|a|^2H(\varphi_a,
K_\delta|a|)),\quad\text{as }|a|\to0^+.  
\end{align}
\end{Theorem}
\begin{proof}
  For $\delta\in(0,1/4)$ fixed, let $r_\delta>0$ and
  $K_\delta>\mu_\delta$ be as in Lemma
  \ref{lemma_limitatezza_N_per_blowup}, so that Lemma
  \ref{lemma_limitatezza_N_per_blowup} yields
\begin{equation}\label{eq:61}
N(\varphi_a, R|a|,\lambda_a,A_a)\leq \frac{k}2+\delta,\quad \text{for all $R\geq
  K_\delta$ and $|a|<\frac{r_\delta}R$}. 
\end{equation}
Let us observe that, by a standard change of variables in the
integrals and \eqref{eq:61},
\begin{align}\label{eq:relazione_N_blowuppate_raggio_fissato}
  N(\varphi_a, R|a|,\lambda_a,A_a) &= \dfrac{R|a|
    \left(\int_{D_{R|a|}} \abs{(i\nabla + A_a)\varphi_a}^2dx -
      \lambda_a \int_{D_{R|a|}} \abs{\varphi_a}^2dx\right)}{\int_{\partial D_{R|a|}}\abs{\varphi_a}^2ds} \\
  \notag &= \dfrac{R \left(\int_{D_R} \abs{(i\nabla + A_{\mathbf
          e})\tilde \varphi_a}^2dx - |a|^2\lambda_a \int_{D_R}
      \abs{\tilde \varphi_a}^2dx\right)}{\int_{\partial
      D_{R}}\abs{\tilde \varphi_a}^2ds} \leq \dfrac{k}2 + \delta.
\end{align}
Thus, via Corollary \ref{corollary_mb}, Lemma
\ref{poincare_inequality} and \eqref{eq:relazione_N_blowuppate_raggio_fissato}, for all $R\geq
  K_\delta$ and $|a|<\frac{r_\delta}R$ there holds
\begin{align}\label{eq:64}
 & (1-5\Lambda r_\delta^2)\int_{D_R}\abs{(i\nabla + A_{\mathbf e})\tilde
      \varphi_a}^2dx\\
\notag&\quad\leq
  \Big(1-\lambda_a|a|^2R^2(1+m_{{\mathbf e}/R}^{-1})\Big)\int_{D_R}\abs{(i\nabla + A_{\mathbf e})\tilde
      \varphi_a}^2dx\\
\notag&\quad\leq
\int_{D_R} \abs{(i\nabla + A_{\mathbf e})\tilde
      \varphi_a}^2dx - |a|^2\lambda_a 
\int_{D_R} \abs{\tilde \varphi_a}^2dx\\
\notag&\quad\leq H(\tilde\varphi_a,R)\Big(\frac{k}2+\delta\Big)=
\frac{H(\varphi_a,R|a|)}{H(\varphi_a,K_\delta|a|)}\bigg(\frac{k}2+\delta\bigg).
\end{align}
From Lemmas \ref{lemma_derivata_H} and
\ref{lemma_limitatezza_N_per_blowup}, there holds that, if $R\geq K_\delta$ and $|a|<\frac{r_\delta}{R}$,
\begin{equation}\label{eq:73}
  \frac1 {H(\varphi_a,r)}\dfrac{d}{dr} H(\varphi_a,r) = \dfrac2r
  N(\varphi_a, r,\lambda_a,A_a)\leq \frac
  2r\bigg(\frac{k}2+\delta\bigg)\quad\text{for all }K_\delta |a|\leq
  r\leq r_\delta,
\end{equation}
hence integration between $K_\delta |a|$ and $R|a|$ yields
\begin{equation}\label{eq:63}
 H(\tilde\varphi_a,R)=\frac{H(\varphi_a,R|a|)}{H(\varphi_a,K_\delta|a|)}
\leq \bigg(\frac{R}{K_\delta}\bigg)^{\!\!k+2\delta}.
\end{equation}
From \eqref{eq:64} and \eqref{eq:63} we obtain that, if $R\geq K_\delta$ and $|a|<\frac{r_\delta}{R}$, 
\begin{equation}\label{eq:65}
  \int_{D_R}\abs{(i\nabla + A_{\mathbf e})\tilde
    \varphi_a}^2dx\leq \frac1{1-5\Lambda
    r_\delta^2}\bigg(\frac{k}2+\delta\bigg)
  \bigg(\frac{R}{K_\delta}\bigg)^{\!\!k+2\delta}.
  \end{equation}
Moreover \eqref{eq:63} yields 
\begin{equation}\label{eq:66}
  \int_{\partial D_R}|\tilde\varphi_a|^2\,ds\leq R
  \bigg(\frac{R}{K_\delta}\bigg)^{\!\!k+2\delta}.
\end{equation}
Estimates \eqref{eq:65} and \eqref{eq:66} together with Lemma
\ref{poincare_inequality} imply \eqref{eq:67}. To conclude, we observe
that \eqref{eq:65} yields \eqref{eq:46}, \eqref{eq:63} imply
\eqref{eq:47}, while \eqref{eq:48} follows from \eqref{eq:46} and
\eqref{eq:47} in view of Lemma \ref{poincare_inequality}.
\end{proof}

\section{Preliminary estimates for the difference of eigenvalues}\label{sec:6}

\subsection{Upper bound for $\lambda_0-\lambda_a$: the Rayleigh quotient for $\lambda_0$}\label{sec:rayl-quot-lambd}
We are now going to estimate the Rayleigh quotient for $\lambda_0$.
By the Courant-Fisher \emph{minimax characterization} of the eigenvalue
 $\lambda_0$, we have that
\begin{equation}\label{eq:91}
  \lambda_0 =\! \min\bigg\{\!\max_{u\in F\setminus \{0\}}
\dfrac{\int_{\Omega} \abs{(i\nabla+A_0) u}^2dx}{\int_{\Omega}
  |u|^2\,dx}:F \text{ is a subspace of $H^{1,0}_0(\Omega,\C)$, $\dim F= n_0$}\bigg\}.
\end{equation}
Let $R>2$. Being $R_0$ as in Lemma \ref{l:Hpos}, for every $a=(|a|,0)$ with $|a|<{R_0}/{R}$ 
we define  the functions
$v_{j,R,a}$  as follows:
\[
 v_{j,R,a}= 
 \begin{cases}
  v_{j,R,a}^{ext}, &\text{in }\Omega \setminus D_{R|a|},\\
  v_{j,R,a}^{int}, &\text{in } D_{R|a|},
 \end{cases}
\quad j=1,\ldots,n_0,
\]
where 
\begin{equation}\label{eq:viext}
  v_{j,R,a}^{ext} := e^{\frac{i}{2}(\theta_0 - \theta_a)}
  \varphi_j^a\quad \text{in }
  \Omega \setminus D_{R|a|}, 
\end{equation}
with $\varphi_j^a$ as in \eqref{eq_eigenfunction}--\eqref{eq:89} and
$\theta_a,\theta_0$ as in \eqref{eq:theta_a} (notice that
$e^{\frac{i}{2}(\theta_0 - \theta_a)}$ is smooth in $\Omega \setminus D_{R|a|}$),  so that it solves
\begin{equation*}
 \begin{cases}
   (i\nabla +A_0)^2 v_{j,R,a}^{ext} = \lambda_j^a v_{j,R,a}^{ext}, &\text{in }\Omega \setminus D_{R|a|},\\
   v_{j,R,a}^{ext} = e^{\frac{i}{2}(\theta_0 - \theta_a)} \varphi_j^a
   &\text{on }\partial (\Omega \setminus D_{R|a|}),
 \end{cases}
\end{equation*}
whereas $v_{j,R,a}^{int}$ is the unique solution to the minimization problem 
\begin{multline}\label{eq:viint}
  \int_{D_{R|a|}} |(i\nabla +A_0)
  v_{j,R,a}^{int}(x)|^2\,dx\\
  = \min\left\{ \int_{D_{R|a|}} |(i\nabla +A_0)u(x)|^2\,dx:\, u\in
    H^{1,0}(D_{R|a|},\C), \ u= e^{\frac{i}{2}(\theta_0 - \theta_a)}
    \varphi_j^a \text{ on }\partial D_{R|a|} \right\},
\end{multline}
so that it solves
\begin{equation}\label{problema_viint}
 \begin{cases}
  (i\nabla +A_0)^2 v_{j,R,a}^{int} = 0, &\text{in }D_{R|a|},\\
  v_{j,R,a}^{int} = e^{\frac{i}{2}(\theta_0 - \theta_a)} \varphi_j^a, &\text{on }\partial D_{R|a|}.
 \end{cases}
\end{equation}
It is easy to verify that 
\begin{equation}\label{eq:90}
\mathop{\rm dim}\big(\mathop{\rm span} \{v_{1,R,a},\ldots,v_{n_0,R,a}\}\big)=n_0.
\end{equation}

\begin{Lemma}\label{lemma_stime_vj}
  For $\delta\in(0,1/4)$, let $\tilde K_\delta>1$ be as in Lemma
  \ref{l:sec_con} and let $R_0$ be as in Lemma \ref{l:Hpos}. For all
  $R> \max\{2,\tilde K_\delta\}$, $a=(|a|,0)\in\Omega$ with
  $|a|<\frac{R_0}R$, and $1\leq j\leq n_0$, let $v_{j,R,a}^{int}$ be
  defined in \eqref{eq:viint}-\eqref{problema_viint}. 
Then
  there exists $\hat C_\delta>0$ (depending only on $\delta$) such that
\begin{align}
\label{eq:35vint}&\int_{D_{R|a|}}|(i\nabla + A_0) v_{j,R,a}^{int}|^2\,dx\leq
\hat C_\delta (R|a|)^{1-2\delta},\\
\label{eq:34vint}&\int_{\partial D_{R|a|}}|v_{j,R,a}^{int}|^2\,ds\leq \hat C_\delta (R|a|)^{2-2\delta},\\
\label{eq:36vint}&\int_{D_{R|a|}}|v_{j,R,a}^{int}|^2\,dx\leq
\hat C_\delta(R|a|)^{3-2\delta}.
\end{align}
\end{Lemma}
\begin{proof}
Let  $\eta_{a,R}:\R^2\to\R$ be a smooth cut-off function such that
$\eta_{a,R}(x)=1$ if $|x|\geq R|a|$, 
$\eta_{a,R}(x)=0$ if $|x|\leq \frac{R|a|}{2}$, 
$0\leq \eta_{a,R}(x)\leq 1$ for all $x$, and $|\nabla
\eta_{a,R}(x)|\leq \frac4{R|a|}$.
From \eqref{eq:viint} it follows that 
\begin{align}\label{eq:12}
&  \int_{D_{R|a|}} |(i\nabla +A_0) v_{j,R,a}^{int}(x)|^2\,dx\leq
  \int_{D_{R|a|}} \Big|(i\nabla +A_0) \big(e^{\frac{i}{2}(\theta_0 -
    \theta_a)} \varphi_j^a \eta_{a,R}\big)(x) \Big|^2\,dx\\
\notag&=\int_{D_{R|a|}} \Big|\eta_{a,R}(x) (i\nabla +A_0) \big(e^{\frac{i}{2}(\theta_0 -
    \theta_a)} \varphi_j^a\big)(x) +ie^{\frac{i}{2}(\theta_0 -
    \theta_a)(x)} \varphi_j^a(x) \nabla\eta_{a,R}(x)\Big|^2\,dx\\
\notag&\leq 2\int_{D_{R|a|}\setminus D_{\frac{R|a|}2} } \Big|(i\nabla +A_0) \big(e^{\frac{i}{2}(\theta_0 -
    \theta_a)} \varphi_j^a\big)(x)\Big|^2\,dx +
2\int_{D_{R|a|}}
|\varphi_j^a(x)|^2 |\nabla\eta_{a,R}(x)|^2\,dx\\
\notag&\leq 2\int_{D_{R|a|}\setminus D_{\frac{R|a|}2} } \big|(i\nabla +A_a)\varphi_j^a(x)\big|^2\,dx +
\frac{32}{R^2|a|^2}\int_{D_{R|a|}}
|\varphi_j^a(x)|^2 \,dx,
\end{align}
which yields \eqref{eq:35vint} in view of estimates \eqref{eq:35} and
\eqref{eq:36}. Estimate \eqref{eq:34vint}
follows directly from \eqref{problema_viint} and \eqref{eq:34}. We
finally conclude by observing that \eqref{eq:36vint} follows from 
Lemma \ref{poincare_inequality} and estimates \eqref{eq:35vint} and \eqref{eq:34vint}.
\end{proof}

For all  $R> 2$ and  $a=(|a|,0)\in\Omega$ with
  $|a|<\frac{R_0}R$, we define 
\begin{align}\label{eq:85} 
Z_a^R(x):=\frac{v_{n_0,R,a}^{int} (|a|x)}{\sqrt{H(\varphi_a, K_\delta|a|)}}.
\end{align}

\begin{Lemma}\label{stime_lemma_blow-up}
For all $R>2$,  
\begin{equation}\label{eq:67_zar}
\text{the family of functions }\big\{Z_a^R: a=|a|{\mathbf e}, |a|<\tfrac{r_\delta}{R}\big\}
 \text{ is bounded in $H^{1 ,0}(D_R,\C)$}.
\end{equation}
In particular, for all $R>2$,  
\begin{align}
\label{eq:46zar}&\int_{D_{R|a|}} \abs{(i\nabla +
  A_0) v_{n_0,R,a}^{int}}^2dx=O(H(\varphi_a,
K_\delta|a|)),\quad\text{as }|a|\to0^+,\\
\label{eq:47zar}&\int_{\partial D_{R|a|}} |v_{n_0,R,a}^{int}|^2dx=O(|a|H(\varphi_a,
K_\delta|a|)),\quad\text{as }|a|\to0^+,\\
\label{eq:48zar}&\int_{D_{R|a|}} |v_{n_0,R,a}^{int}|^2dx=O(|a|^2H(\varphi_a,
K_\delta|a|)),\quad\text{as }|a|\to0^+.  
\end{align}
\end{Lemma}
\begin{proof}
We notice that $Z_a^R$ solves
\begin{equation*}
 \begin{cases}
   (i\nabla +A_0)^2 Z_a^R =0, &\text{in $D_R$}\\
   Z_a^R = e^{\frac{i}{2}(\theta_0-\theta_{\mathbf
       e})}\tilde\varphi_a, &\text{on $\partial D_R$},
 \end{cases}
\end{equation*}
and, by the Dirichlet principle and Theorem \ref{t:stime_blowup},
\begin{align}\label{eq:49}
 \int_{D_R} &|(i\nabla +A_0)Z_a^R|^2dx
\leq \int_{D_R} \abs{(i\nabla +A_0)\big(\eta_R\,
e^{\frac{i}{2}(\theta_0-\theta_{\mathbf e})}\tilde \varphi_a)\big)}^2 dx\\
\notag &\leq 2 \int_{D_R}|\nabla \eta_R|^2|
\tilde \varphi_a|^2dx 
+2\int_{D_R\setminus D_{R/2}}
\eta_R^2\big|(i\nabla +A_{\mathbf e})\tilde \varphi_a\big|^2dx\leq C_R,
\end{align}
for some $C_R>0$, 
where $\eta_R:\R^2\to\R$ is a smooth cut-off function as in \eqref{eq:88}.
Then, taking into account
\eqref{eq:hardy}, we obtain \eqref{eq:67_zar}. Estimate
\eqref{eq:46zar} follows directly from \eqref{eq:49} and
\eqref{eq:85} while \eqref{eq:47zar} is a direct consequence of the
definition of $v_{n_0,R,a}^{int}$ (see \eqref{problema_viint}) and
\eqref{eq:47}.
 \eqref{eq:48zar} follows from \eqref{eq:46zar} and
\eqref{eq:47zar} in view of Lemma \ref{poincare_inequality}.
\end{proof}

\begin{Lemma}\label{l:stima_Lambda0_sopra}
 There exists
  $\tilde R>2$ such that for all $R>\tilde R$ and  $a=(|a|,0)\in\Omega$ with
  $|a|<\frac{R_0}R$, 
\[
\frac{\lambda_0-\lambda_{a}}{H(\varphi_{a}, K_\delta|{a}|)}\leq f_R(a)
\]
where 
\begin{align}\label{eq:69}
&f_R(a)=
\int_{D_R}|(i\nabla+A_0)Z_a^R|^2\,dx-\int_{D_R}|(i\nabla+A_{\mathbf
  e})\tilde \varphi_a|^2\,dx+o(1),\quad\text{as }|a|\to 0^+,\\
\notag &f_R(a)=O(1) ,\quad\text{as }|a|\to 0^+,
\end{align}
with $\tilde \varphi_a$ and  $Z_a^R$ defined in
\eqref{def_blowuppate_normalizzate} and \eqref{eq:85} respectively.
\end{Lemma}

\begin{proof}
  Let $\tilde K_\delta>1$ be as in Lemma
  \ref{l:sec_con} and fix 
  $R> \max\{2,\tilde K_\delta\}$. 
  
  In \eqref{eq:91} we choose $F$ as the space of functions $\{\tilde
  v_{j,R,a}\}$ which result from $\{v_{j,R,a}\}$ by a Gram--Schmidt
  process, that is
\[ 
\tilde v_{j,R,a}:= \dfrac{\hat v_{j,R,a}}{\|\hat v_{j,R,a}\|_{L^2(\Omega,\C)}}, \quad j=1,\ldots, n_0, 
\]
where
\begin{align*}
& \hat v_{n_0,R,a} := v_{n_0,R,a},\\
& \hat v_{j,R,a} := v_{j,R,a} - \sum_{\ell= j+1}^{n_0}\dfrac{\int_\Omega
  v_{j,R,a}  \overline{\hat v_{\ell,R,a}}\,dx}{\|\hat v_{\ell,R,a}\|_{L^2(\Omega,\C)}^2} \hat v_{\ell,R,a} 
\quad \text{for }j=1,\ldots, n_0-1.
\end{align*}
For notation convenience we also set 
\[
 d_{\ell,j}^{R,a} :=\dfrac{\int_\Omega
  v_{j,R,a}
  \overline{\hat v_{\ell,R,a}}\,dx}{\|\hat v_{\ell,R,a}\|_{L^2(\Omega,\C)}^2}.
 \]
From \eqref{eq:23}, Lemmas \ref{l:sec_con} and \ref{lemma_stime_vj},
and an induction argument, il follows that 
\begin{equation}\label{eq:18}
\|\hat v_{j,R,a} \|_{L^2(\Omega,\C)}^2=1+O(|a|^{3-2\delta})
\quad\text{and}\quad 
d_{\ell,j}^{R,a}=O(|a|^{3-2\delta})\text{ for }\ell\neq j
\end{equation}
as $|a|\to 0^+$.
Furthermore, from \eqref{eq:23}, \eqref{eq:48}, and \eqref{eq:48zar} we deduce that 
\begin{equation}\label{eq:13}
\|\hat v_{n_0,R,a} \|_{L^2(\Omega,\C)}^2=\|v_{n_0,R,a}
\|_{L^2(\Omega,\C)}^2=1+O\big(|a|^{2}H(\varphi_a, K_\delta|a|)\big)
\quad\text{as }|a|\to 0^+,
\end{equation}
and 
\begin{equation}\label{eq:76}
\quad 
d_{n_0,j}^{R,a}= O\big(|a|^{\frac52-\delta}\sqrt{H(\varphi_a, K_\delta|a|})\big)
\quad\text{as }|a|\to 0^+,\quad\text{for all }j<n_0.
\end{equation}
From \eqref{eq:91} and \eqref{eq:90} it follows that 
\begin{equation*}
  \lambda_0 \leq \max_
  {\substack{(\alpha_1,\dots, \alpha_{n_0})\in \C^{n_0}\\
      \sum_{j=1}^{{n_0}}|\alpha_j|^2 =1}}
  \int_{\Omega} \bigg|(i\nabla+A_0) \bigg(\sum_{j=1}^{{n_0}}\alpha_j
     \tilde v_{j,R,a}\bigg)\bigg|^2 dx.
\end{equation*}
Hence 
\begin{equation}\label{eq:107}
\lambda_0-\lambda_a\leq \max_
  {\substack{(\alpha_1,\dots, \alpha_{n_0})\in \C^{n_0}\\
      \sum_{j=1}^{{n_0}}|\alpha_j|^2 =1}}\sum_{j,n=1}^{n_0}
m_{j,n}^{a,R}\alpha_j \overline{\alpha_n},
\end{equation}
where 
\[
m_{j,n}^{a,R}= \int_{\Omega} (i\nabla+A_0)    \tilde v_{j,R,a}\cdot
\overline{(i\nabla+A_0)    \tilde v_{n,R,a}}\, dx
-\lambda_a \delta_{jn},
\]
with $\delta_{jn}=1$ if $j=n$ and 
$\delta_{jn}=0$ if $j\neq n$.

Before proceeding, let us observe that integration of \eqref{eq:73}
over the interval $(K_\delta |a|, r_\delta)$ yields
\begin{equation}\label{eq:stima_sotto_radiceH}
H(\varphi_a, K_\delta |a|) \geq C_\delta |a|^{k +
   2\delta},\quad\text{if }|a|<\frac{r_\delta}{K_\delta},
\end{equation}
for some $C_\delta>0$ independent of $a$,
whereas estimate
\eqref{eq:34}  implies that  
\begin{equation}\label{eq:stima_sopra_radiceH}
 H(\varphi_a, K_\delta |a|)=O(|a|^{1-2\delta})\quad\text{ as }|a|\to0.
\end{equation}
From \eqref{eq:13}, \eqref{eq:85},
\eqref{def_blowuppate_normalizzate},  Theorem \ref{t:stime_blowup}, and
Lemma \ref{stime_lemma_blow-up} we deduce that 
\begin{align}\label{eq:95}
  m_{n_0,n_0}^{a,R} &= \dfrac{\lambda_a (1-\|v_{n_0,R,a} \|_{L^2(\Omega,\C)}^2)}{\|v_{n_0,R,a} \|_{L^2(\Omega,\C)}^2} \\
  \notag&\ + \dfrac{1}{\|v_{n_0,R,a} \|_{L^2(\Omega,\C)}^2} \left(
    \int_{D_{R|a|}} \big| (i\nabla +A_0)v_{n_0,R,a}^{int} \big|^2 dx
    - \int_{D_{R|a|}} \big| (i\nabla +A_a)\varphi_{a} \big|^2 dx\right)\\
  \notag &= H(\varphi_a,
  K_\delta|a|)\bigg(\int_{D_R}|(i\nabla+A_0)Z_a^R|^2\,dx-\int_{D_R}|(i\nabla+A_{\mathbf
    e})\tilde \varphi_a|^2\,dx+o(1)\bigg),
\end{align}
as 
$|a|\to0^+$.
From \cite[Theorem 1.1]{BNNNT} (which ensures that $\lambda_j^a\to
\lambda_j^0$ as $|a|\to0$), \eqref{eq:18}, \eqref{eq_eigenfunction}, \eqref{eq:23}, and Lemmas \ref{l:sec_con} and
\ref{lemma_stime_vj}, we obtain that, if $j<n_0$, 
\begin{align}\label{eq:96}
  m_{j,j}^{a,R}&= -\lambda_a \\
  \notag &\ +\dfrac{1}{\|\hat v_{j,R,a}\|_{L^2(\Omega,\C)}^2} \left(
    \lambda_j^a - \int_{D_{R|a|}}\!\! \!\abs{(i\nabla
      +A_a)\varphi_j^a}^2dx
    + \int_{D_{R|a|}}\! \!\!\abs{(i\nabla +A_0)v_{j,R,a}^{int}}^2dx \right) \\
  &\notag \ + \dfrac{1}{\|\hat v_{j,R,a}\|_{L^2(\Omega,\C)}^2}
  \int_{\Omega} \bigg|(i\nabla + A_0)
  \Big(\sum_{\ell>j} d_{\ell,j}^{R,a} \hat v_{\ell,R,a} \Big)\bigg|^2dx\\
  &\notag \ -\dfrac{2}{\|\hat v_{j,R,a}\|_{L^2(\Omega,\C)}^2}
  \Re\bigg(\int_{\Omega} (i\nabla + A_0)v_{j,R,a} \cdot
  \overline{(i\nabla + A_0) \Big( \sum_{\ell>j} d_{\ell,j}^{R,a} \hat v_{\ell,R,a} \Big)}\,dx\bigg)\\
  \notag&=(\lambda_j^0-\lambda_0)+o(1)\quad\text{as }|a|\to0.
\end{align}
From \eqref{eq:18}, \eqref{eq:76},
\eqref{eq_eigenfunction}, \eqref{eq:23}, \eqref{eq:46},
Lemmas \ref{l:sec_con} and \ref{lemma_stime_vj}, and 
\eqref{eq:46zar}, it follows that, for all $j<n_0$,
\begin{align}\label{eq:98}
  m_{j,n_0}^{a,R}&= \dfrac{\int_{D_{R|a|}}\Big((i\nabla+A_0)
    v_{j,R,a}^{int}\cdot \overline{(i\nabla+A_0) v_{n_0,R,a}^{int} }-
    (i\nabla+A_a) \varphi_j^a\cdot
    \overline{(i\nabla+A_a) \varphi_a}\Big)\,dx}{\|\hat v_{j,R,a}\|_{L^2(\Omega,\C)} \|\hat v_{n_0,R,a}\|_{L^2(\Omega,\C)}} \\
  \notag &\ - \dfrac{\int_{\Omega} (i\nabla+A_0)
    \big(\sum_{\ell>j}d_{\ell,j}^{R,a} \hat v_{\ell,R,a}\big) \cdot
    \overline{(i\nabla+A_0) v_{n_0,R,a} }\,dx}{\|\hat
    v_{j,R,a}\|_{L^2(\Omega,\C)} \|\hat
    v_{n_0,R,a}\|_{L^2(\Omega,\C)}}
  \\
  \notag&=O\Big(|a|^{\frac12 -\delta}
  \sqrt{H(\varphi_a,K_\delta|a|})\Big),
\end{align}
and 
\begin{equation}\label{eq:99}
m_{n_0,j}^{a,R}=\overline{ m_{j,n_0}^{a,R}}=O\Big(|a|^{\frac12 -\delta} \sqrt{H(\varphi_a,K_\delta|a|})\Big)
\end{equation}
as 
$|a|\to 0^+$.
From  \eqref{eq:18},
\eqref{eq_eigenfunction}, \eqref{eq:23}, and Lemmas \ref{l:sec_con} and \ref{lemma_stime_vj},
we deduce that, for all $j,n<n_0$ with $j\neq n$,
\begin{align}\label{eq:100} 
  m_{j,n}^{a,R}&=\dfrac{\int_{D_{R|a|}}\Big((i\nabla+A_0)
    v_{j,R,a}^{int}\cdot \overline{(i\nabla+A_0) v_{n,R,a}^{int} }-
    (i\nabla+A_a) \varphi_j^a\cdot \overline{(i\nabla+A_a)
      \varphi_n^a}\Big)\,dx}{
    \|\hat v_{j,R,a}\|_{L^2(\Omega,\C)}\|\hat v_{n,R,a}\|_{L^2(\Omega,\C)}} \\
  \notag &\ + \dfrac{\int_{\Omega} (i\nabla +A_0) \big( \sum_{\ell>j}
    d_{\ell,j}^{R,a} \hat v_{\ell,R,a} \big) \cdot \overline{(i\nabla
      +A_0) \big( \sum_{h>n} d_{h,n}^{R,a} \hat v_{h,R,a}
      \big)}\,dx}{\|\hat v_{j,R,a}\|_{L^2(\Omega,\C)}\|\hat
    v_{n,R,a}\|_{L^2(\Omega,\C)}}
  \\
  \notag &\ - \dfrac{\int_{\Omega} (i\nabla +A_0) \big( \sum_{\ell>j}
    d_{\ell,j}^{R,a} \hat v_{\ell,R,a} \big) \cdot \overline{(i\nabla
      +A_0) v_{n,R,a}} \,dx}{\|\hat v_{j,R,a}\|_{L^2(\Omega,\C)}\|\hat
    v_{n,R,a}\|_{L^2(\Omega,\C)}}
  \\
  \notag &\ - \dfrac{\int_{\Omega} (i\nabla +A_0) v_{j,R,a} \cdot
    \overline{(i\nabla +A_0) \big( \sum_{h>n} d_{h,n}^{R,a} \hat
      v_{h,R,a} \big)}\,dx}{\|\hat v_{j,R,a}\|_{L^2(\Omega,\C)}\|\hat
    v_{n,R,a}\|_{L^2(\Omega,\C)}}
  \\
  \notag&=O(|a|^{1-2\delta}) \quad\text{as }|a|\to0.
\end{align}
Let $\alpha_{j}^{a,R}\in\C$, $j=1,\dots,n_0$, such that
\begin{equation}\label{eq:106}
  \sum_{j=1}^{{n_0}}|\alpha_j^{a,R}|^2 =1\quad\text{and}\quad
  \sum_{j,n=1}^{n_0} m_{j,n}^{a,R}\alpha_j^{a,R}
  \overline{\alpha_n^{a,R}}= \max_
  {\substack{(\alpha_1,\dots, \alpha_{n_0})\in \C^{n_0}\\
      \sum_{j=1}^{{n_0}}|\alpha_j|^2 =1}}\sum_{j,n=1}^{n_0}
  m_{j,n}^{a,R}\alpha_j \overline{\alpha_n}.
\end{equation}
From 
\begin{equation}\label{eq:absubis}
 m_{n_0,n_0}^{a,R} \leq \sum_{j,n=1}^{n_0} m_{j,n}^{a,R}\alpha_j^{a,R}
\overline{\alpha_l^{a,R}}
\end{equation}
it follows that 
\[ 
\big(1-|\alpha_{n_0}^{a,R}|^2\big)\Big(m_{n_0,n_0}^{a,R}-\max_{j<n_0}m_{j,j}^{a,R}\Big)\leq
\sum_{j\neq n}
m_{j,n}^{a,R}\alpha_j^{a,R}
\overline{\alpha_n^{a,R}}
\]
and hence, by \eqref{eq:95} and \eqref{eq:96},
\begin{equation}\label{eq:105}
  \big(1-|\alpha_{n_0}^{a,R}|^2\big)\Big(-\max_{j<n_0}(\lambda_j^0-\lambda_0)+o(1)\Big)\leq
  \sum_{j\neq n}
  m_{j,n}^{a,R}\alpha_j^{a,R}
  \overline{\alpha_n^{a,R}},
\end{equation}
as 
$|a|\to 0^+$.
Due to \eqref{eq:1}, \eqref{eq:stima_sopra_radiceH}, and \eqref{eq:98}--\eqref{eq:100} we then have
\begin{equation}\label{eq:101}
1-|\alpha_{n_0}^{a,R}|^2=O(|a|^{1-2\delta})\quad\text{as
}|a|\to 0^+.
\end{equation}
Since $1-|\alpha_{n_0}^{a,R}|^2=\sum_{j<n_0}|\alpha_{j}^{a,R}|^2$, we
also have that
\begin{equation}\label{eq:102}
|\alpha_{j}^{a,R}|^2=O(|a|^{1-2\delta}),\quad
\text{for all }j<n_0,
\end{equation}
as 
$|a|\to 0^+$.
We claim that 
\begin{equation}\label{eq:claim}
\sum_{\substack{j,n=1 \\ j\neq n}}^{n_0}
m_{j,n}^{a,R}\alpha_j^{a,R} \overline{\alpha_n^{a,R}} =
o\big(H(\varphi_a,K_\delta|a|)\big)  
\end{equation}
as 
$|a|\to 0^+$. To prove \eqref{eq:claim} it is enough to show that
\begin{equation}\label{eq:50}
  \begin{cases}
    &\text{for every sequence $a_l=|a_l|{\mathbf e}\to 0$ there
      exists a subsequence
      $a_{l_p}$ such that}\\
    &\sum_{\substack{j,n=1 \\ j\neq n}}^{n_0}
    m_{j,n}^{a_{l_p},R}\alpha_j^{a_{l_p},R}
    \overline{\alpha_n^{a_{l_p},R}} =
    o\big(H(\varphi_{a_{l_p}},K_\delta|a_{l_p}|)\big) \text{ as
      $p\to+\infty$.}
  \end{cases}
\end{equation}
 Let $a_l=|a_l|{\mathbf e}\to 0$.
Via estimates \eqref{eq:stima_sopra_radiceH}, \eqref{eq:98},
\eqref{eq:99},
\eqref{eq:100}, \eqref{eq:101}, and \eqref{eq:102},  we
have that 
\begin{align*}
  \sum_{\substack{j,n=1 \\ j\neq n}}^{n_0} m_{j,n}^{a,R}\alpha_j^{a,R}
  \overline{\alpha_n^{a,R}} &= \sum_{\substack{j,n=1 \\ j\neq
      n}}^{n_0-1} m_{j,n}^{a,R}\alpha_j^{a,R}
  \overline{\alpha_n^{a,R}} + \sum_{j=1}^{n_0-1}
  m_{j,n_0}^{a,R}\alpha_j^{a,R} \overline{\alpha_{n_0}^{a,R}}
  + \sum_{j=1}^{n_0-1} m_{n_0,j}^{a,R}\alpha_{n_0}^{a,R} \overline{\alpha_{j}^{a,R}} \\
  &= O(|a|^{2-4\delta}) + O(|a|^{1-2\delta}\sqrt{H(\varphi_a,K_\delta|a|)}) \\
  &= O(|a|^{3/2-3\delta})
\end{align*}
as $a=a_{l}$, $l\to\infty$.
If
$|a_{l}|^{3/2-3\delta}=o(H(\varphi_{a_{l}},K_\delta|a_{l}|))$, we have proved claim
\eqref{eq:50}; if not, there holds
\begin{equation}\label{stima_sopra_radiceH_prima}
 H(\varphi_a,K_\delta|a|) = O(|a|^{3/2-3\delta})
\end{equation}
along a subsequence of $a_{l}$ (still denoted as $a_{l}$).
Hence estimate \eqref{eq:102} is improved as
\begin{equation}\label{eq:102'}
|\alpha_{j}^{a,R}|^2=O(|a|^{3/2-3\delta}),\quad
\text{for all }j<n_0,
\end{equation}
along the subsequence.  We now perform a recursive argument, improving
the previous estimates step by step.  Proceeding as above and
exploiting the improved estimate \eqref{stima_sopra_radiceH_prima},
along the subsequence we have
\begin{align*}
  \sum_{\substack{j,n=1 \\ j\neq n}}^{n_0} m_{j,n}^{a,R}\alpha_j^{a,R}
  \overline{\alpha_n^{a,R}} &= \sum_{\substack{j,n=1 \\ j\neq
      n}}^{n_0-1} m_{j,n}^{a,R}\alpha_j^{a,R}
  \overline{\alpha_n^{a,R}} + \sum_{j=1}^{n_0}
  m_{j,n_0}^{a,R}\alpha_j^{a,R} \overline{\alpha_{n_0}^{a,R}}
  + \sum_{j=1}^{n_0-1} m_{n_0,j}^{a,R}\alpha_{n_0}^{a,R} \overline{\alpha_{j}^{a,R}}  \\
  &= O(|a|^{1-2\delta}|a|^{\frac32-3\delta}) + O\big(|a|^{\frac12-\delta}\sqrt{H(\varphi_a,K_\delta|a|)}|a|^{\frac34-\frac32\delta}\big) \\
  &= O(|a|^{2-4\delta})
\end{align*}
via estimates \eqref{eq:98},
\eqref{eq:99}, \eqref{eq:100}, \eqref{stima_sopra_radiceH_prima}, \eqref{eq:102'}.
If $|a|^{2-4\delta}=o(H(\varphi_a,K_\delta|a|))$ along the
subsequence, we have proved claim \eqref{eq:50}; if not, up to
passing to a subsequence again, there holds 
\begin{equation}\label{stima_sopra_radiceH_seconda}
 H(\varphi_a,K_\delta|a|) = O(|a|^{2-4\delta}).
\end{equation}
Hence we  improve estimate \eqref{eq:102'} as 
\begin{equation}\label{eq:102''}
|\alpha_{j}^{a,R}|^2=O(|a|^{2-4\delta}),\quad
\text{for all }j<n_0,
\end{equation}
along the subsequence.
 Repeating the above argument $M$ times, we obtain that, along   a
 subsequence,
\[
\sum_{\substack{j,n=1 \\ j\neq n}}^{n_0}
m_{j,n}^{a,R}\alpha_j^{a,R} \overline{\alpha_n^{a,R}} = O\big(|a|^{1+\frac
  M2-(2+M)\delta}\big).
\]
In view of \eqref{eq:stima_sotto_radiceH}, if $M$ is such that $1+\frac
  M2-(2+M)\delta>k+2\delta$, we obtain claim \eqref{eq:50} and then \eqref{eq:claim}.

From \eqref{eq:105} and \eqref{eq:claim}, it follows
that 
\begin{equation}\label{eq:104}
  |\alpha_{n_0}^{a,R}|^2=1+o(H(\varphi_a,K_\delta|a|))\quad\text{and}\quad 
  |\alpha_{j}^{a,R}|^2=o(H(\varphi_a,K_\delta|a|))\quad\text{for all }j<n_0,
\end{equation}
as $|a|\to 0^+$. From \eqref{eq:95}, \eqref{eq:96}, \eqref{eq:106},
\eqref{eq:claim}, and \eqref{eq:104}, we deduce that
\begin{multline*}
  \max_
  {\substack{(\alpha_1,\dots, \alpha_{n_0})\in \C^{n_0}\\
      \sum_{j=1}^{{n_0}}|\alpha_j|^2 =1}}\sum_{j,n=1}^{n_0}
  m_{j,n}^{a,R}\alpha_j
  \overline{\alpha_n}\\
  =H(\varphi_a,K_\delta|a|)\bigg(
  \int_{D_R}|(i\nabla+A_0)Z_a^R|^2\,dx-\int_{D_R}|(i\nabla+A_{\mathbf
    e})\tilde \varphi_a|^2\,dx+o(1)\bigg)
\end{multline*}
as $|a|\to 0^+$,
which, in view of \eqref{eq:107}, yields
\[
\frac{\lambda_0-\lambda_{a}}{H(\varphi_{a},K_\delta|a|)}\leq f_R(a)
\]
where 
\[
f_R(a)=
\int_{D_R}|(i\nabla+A_0)Z_a^R|^2\,dx-\int_{D_R}|(i\nabla+A_{\mathbf
  e})\tilde \varphi_a|^2\,dx+o(1)\quad\text{as }|a|\to 0^+.
\]
We notice that, from Theorem \ref{t:stime_blowup} and Lemma
\ref{stime_lemma_blow-up}, for all $R>\tilde R$, $f_R(a)=O(1)$ as $|a|\to 0^+$.
The proof is thereby complete.
\end{proof}

As a direct consequence of Lemma \ref{l:stima_Lambda0_sopra} the
following corollary holds.

\begin{Corollary}\label{cor:lower-bound-lambda_a}
There exists  positive constants $C^*,r^*>0$ such that, for all
$a=(|a|,0)\in\Omega$ with $|a|<r^*$,
\[
\lambda_0-\lambda_{a}\leq C^*H(\varphi_{a}, K_\delta|{a}|).
\]
\end{Corollary}

\subsection{Lower bound for $\lambda_0-\lambda_a$: the Rayleigh quotient for $\lambda_a$}
By the Courant-Fisher \emph{minimax characterization} of the eigenvalue
 $\lambda_a$, we have that
\begin{equation}\label{eq:91_la}
  \lambda_a =\! \min\bigg\{\!\max_{u\in F\setminus \{0\}}
\dfrac{\int_{\Omega} \abs{(i\nabla+A_a) u}^2dx}{\int_{\Omega}
  |u|^2\,dx}:F \text{ is a subspace of $H^{1,a}_0(\Omega,\C)$, $\dim F= n_0$}\bigg\}.
\end{equation}
Being $R_0$ as in Lemma \ref{l:Hpos}, for every $R>2$ and
$a=(|a|,0)\in\Omega$ with $|a|<{R_0}/{R}$ we define the functions
$w_{j,R,a}$ as
\[
 w_{j,R,a}= 
 \begin{cases}
  w_{j,R,a}^{ext}, &\text{in }\Omega \setminus D_{R|a|},\\
  w_{j,R,a}^{int}, &\text{in } D_{R|a|},
 \end{cases}
\quad j=1,\ldots,n_0,
\]
where 
\begin{equation*}
  w_{j,R,a}^{ext} := e^{\frac{i}{2}(\theta_a - \theta_0)}
  \varphi_j^0\quad \text{in }
  \Omega \setminus D_{R|a|}, 
\end{equation*}
with $\varphi_j^0$ as in \eqref{eq_eigenfunction}--\eqref{eq:89} with $a=0$,  so that it solves
\begin{equation*}
 \begin{cases}
   (i\nabla +A_a)^2 w_{j,R,a}^{ext} = \lambda_j^0 w_{j,R,a}^{ext}, &\text{in }\Omega \setminus D_{R|a|},\\
   w_{j,R,a}^{ext} = e^{\frac{i}{2}(\theta_a - \theta_0)} \varphi_j^0
   &\text{on }\partial (\Omega \setminus D_{R|a|}),
 \end{cases}
\end{equation*}
whereas $w_{j,R,a}^{int}$ is the unique solution to the minimization problem 
\begin{multline}\label{eq:viint_la}
  \int_{D_{R|a|}} |(i\nabla +A_a)
  w_{j,R,a}^{int}(x)|^2\,dx\\
  = \min\left\{ \int_{D_{R|a|}} |(i\nabla +A_a)u(x)|^2\,dx:\, u\in
    H^{1,a}(D_{R|a|},\C), \ u= e^{\frac{i}{2}(\theta_a - \theta_0)}
    \varphi_j^0 \text{ on }\partial D_{R|a|} \right\},
\end{multline}
thus solving
\begin{equation}\label{problema_viint_la}
 \begin{cases}
  (i\nabla +A_a)^2 w_{j,R,a}^{int} = 0, &\text{in }D_{R|a|},\\
  w_{j,R,a}^{int} = e^{\frac{i}{2}(\theta_a - \theta_0)} \varphi_j^0, &\text{on }\partial D_{R|a|}.
 \end{cases}
\end{equation}
It is easy to verify that 
\begin{equation}\label{eq:90_la}
\mathop{\rm dim}\big(\mathop{\rm span} \{w_{1,R,a},\ldots,w_{n_0,R,a}\}\big)=n_0.
\end{equation}
As a direct consequence of \cite[Theorem 1.3]{FFT} (see also Proposition \ref{prop:fft}), there exists
some $\tilde K>0$ such, for every $R>2$, $a=(|a|,0)\in\Omega$ with
  $|a|<\frac{R_0}R$, 
and $1\leq j\leq n_0$,
\begin{align}
\label{eq:34_la}&\int_{\partial D_{R|a|}}|\varphi_j^0|^2\,ds\leq \tilde K (R|a|)^{2},\\
\label{eq:35_la}&\int_{D_{R|a|}}|(i\nabla + A_0)\varphi_j^0|^2\,dx\leq \tilde K (R|a|),\\
\label{eq:36_la}&\int_{D_{R|a|}}|\varphi_j^0|^2\,dx\leq
\tilde K (R|a|)^{3}.
\end{align}
Arguing as in the proof of Lemma \ref{lemma_stime_vj} (using estimates
\eqref{eq:34_la}, \eqref{eq:35_la} and \eqref{eq:36_la} instead of
\eqref{eq:34}--\eqref{eq:36}) we obtain (up to enlarging the constant
$\tilde K$) that, for every $R>2$, $a=(|a|,0)\in\Omega$ with
  $|a|<\frac{R_0}R$, 
and $1\leq j\leq n_0$,
\begin{align}
\label{eq:35vint_la}&\int_{D_{R|a|}}|(i\nabla + A_a) w_{j,R,a}^{int}|^2\,dx\leq
\tilde K (R|a|),\\
\label{eq:34vint_la}&\int_{\partial
  D_{R|a|}}|w_{j,R,a}^{int}|^2\,ds\leq \tilde K (R|a|)^{2},\\
\label{eq:36vint_la}&\int_{D_{R|a|}}|w_{j,R,a}^{int}|^2\,dx\leq
\tilde K(R|a|)^{3}.
\end{align}
For all  $R> 2$ and  $a=(|a|,0)\in\Omega$ with
  $|a|<\frac{R_0}R$, we define 
\begin{align}\label{eq:85_la} 
U_a^R(x):=\frac{w_{n_0,R,a}^{int} (|a|x)}{|a|^{k/2}}, \quad 
 W_a(x):=\frac{\varphi_0(|a|x)}{|a|^{k/2}}.
\end{align}
Under assumptions \eqref{eq:37} and \eqref{eq:54}, from \cite[Theorem 1.3 and
Lemma 6.1]{FFT} we have  that 
\begin{equation}\label{eq:vkext_la}
W_a\to  \beta e^{\frac i2\theta_0}\psi_k\quad\text{as } |a|\to0
\end{equation}
in $H^{1 ,0}(D_R,\C)$ for
every $R>1$, where $\psi_k$ is defined in \eqref{eq:psi_k} and 
\begin{equation}\label{eq:beta}
\beta:= \frac{\beta_k^2(0,\varphi_0,\lambda_0)}{\sqrt\pi}
\end{equation}
with $\beta_k^2(0,\varphi_0,\lambda_0)$ as in \eqref{def_beta} 
and  \eqref{eq:84}.

We also denote as $w_R$ the unique solution to the minimization problem 
\begin{multline}\label{eq:def_zR_la}
  \int_{D_{R}} |(i\nabla +A_{\mathbf e})
  w_R(x)|^2\,dx\\
  = \min\left\{ \int_{D_{R}} |(i\nabla +A_{\mathbf e})u(x)|^2\,dx:\, u\in
    H^{1,{\mathbf e}}(D_{R},\C), \ u= e^{\frac{i}{2}\theta_{\mathbf e}}
\psi_k \text{ on }\partial D_{R} \right\},
\end{multline}
which then  solves
\begin{equation}\label{eq:equazione_w_R}
 \begin{cases}
  (i\nabla +A_{\mathbf e})^2 w_R = 0, &\text{in }D_{R},\\
  w_R = e^{\frac{i}{2}\theta_{\mathbf e}}
\psi_k, &\text{on }\partial D_{R}.
 \end{cases}
\end{equation}
By the Dirichlet principle and \eqref{eq:vkext_la}, we have that  
\begin{align*}
 \int_{D_R} &\abs{(i\nabla +A_{\mathbf e})(U_a^R - \beta w_R) }^2dx
\leq \int_{D_R} \abs{(i\nabla +A_{\mathbf e})\big(\eta_R\,
e^{\frac{i}{2}(\theta_{\mathbf e}-\theta_0)}(W_a- \beta e^{\frac{i}{2}\theta_{0} }\psi_k\big)}^2 dx\\
&\leq 2 \int_{D_R}|\nabla \eta_R|^2\big|W_a- \beta e^{\frac{i}{2}\theta_{0} }\psi_k\big|^2dx 
+2\int_{D_R\setminus D_{R/2}}
\eta_R^2\big|(i\nabla +A_{0})(W_a- \beta e^{\frac{i}{2}\theta_{0} }\psi_k)\big|^2dx\\
&= o(1)\quad \text{as }
|a|\to0^+,
\end{align*}
where $\eta_R:\R^2\to\R$ is a smooth cut-off function as in \eqref{eq:88}.
Hence, for all $R>2$,
\begin{equation}\label{eq:vkint_la} 
  U_a^R \to  \beta w_R, \quad\text{in  }H^{1,{\mathbf e}}(D_R,\C),
\end{equation}
as $|a|\to 0$, where $\beta$ is defined in \eqref{eq:beta}.

\begin{Lemma}\label{l:conv_w_r}
For every $r>1$, $w_R\to\Psi_k$ in $H^{1,{\mathbf e}}(D_r,\C)$ as $R\to+\infty$.  
\end{Lemma}
\begin{proof}
  Let $r>2$. For every $R>r$, by the Dirichlet Principle, \eqref{eq:17}, and
  \eqref{eq:75} we have that,
  letting $\eta_R$ as in \eqref{eq:88},
  \begin{align*}
    \int_{D_{r}}& |(i\nabla +A_{\mathbf e})(w_R-\Psi_k)(x)|^2\,dx\leq
    \int_{D_{R}} |(i\nabla +A_{\mathbf e})(w_R-\Psi_k)(x)|^2\,dx\\
&\leq \int_{D_{R}} |(i\nabla +A_{\mathbf e})(\eta_R(e^{\frac{i}{2}\theta_{\mathbf e}}\psi_k
-\Psi_k))(x)|^2\,dx\\
&\leq 2 \int_{D_{R}\setminus D_{R/2}} \eta_R^2|(i\nabla +A_{\mathbf e})(e^{\frac{i}{2}\theta_{\mathbf e}}\psi_k
-\Psi_k)|^2\,dx+2
 \int_{D_{R}\setminus D_{R/2}} |\nabla\eta_R|^2|e^{\frac{i}{2}\theta_{\mathbf e}}\psi_k
-\Psi_k|^2\,dx
\\
&\leq  2\int_{\R^2\setminus D_{R/2}} |(i\nabla +A_{\mathbf e})(e^{\frac{i}{2}\theta_{\mathbf e}}\psi_k
-\Psi_k)|^2\,dx+\frac{32}{R^2}
 \int_{D_{R}\setminus D_{R/2}}|e^{\frac{i}{2}\theta_{\mathbf e}}\psi_k
-\Psi_k|^2\,dx=o(1)
  \end{align*}
as $R\to+\infty$.
\end{proof}

\begin{Lemma}\label{l:stima_Lambda0_sotto}
 For $a=(|a|,0)\in\Omega$, let 
$\varphi_a\in
H^{1,a}_{0}(\Omega,\C)$ solve (\ref{eq:equation_a}-\ref{eq:6})
and $\varphi_0\in
H^{1,0}_{0}(\Omega,\C)$ be a solution to
(\ref{eq:equation_lambda0}--\ref{eq:83}). If \eqref{eq:1} and
\eqref{eq:37} hold and \eqref{eq:54} is satisfied, then, for all
  $R>\tilde R$ and $a=(|a|,0)\in\Omega$,
\[
\frac{\lambda_0-\lambda_a}{|a|^k}\geq g_R(a)
\]
where 
\[
\lim_{|a|\to 0}g_R(a)=i|\beta|^2\tilde\kappa_R,
\]
with $\beta$ as in \eqref{eq:beta} and 
\begin{equation}\label{eq:tildekappa_R}
\tilde\kappa_R=\int_{\partial D_R}\Big(e^{-\frac i2 \theta_{\mathbf
    e}} (i\nabla+A_{\mathbf e})w_R\cdot\nu
- (i\nabla) \psi_k\cdot\nu
\Big)\psi_k\,ds
\end{equation}
being $\psi_k$ as in \eqref{eq:psi_k}.
\end{Lemma}
\begin{proof}
In \eqref{eq:91_la} we choose $F$ as the space of functions $\{\tilde w_{j,R,a}\}$ which result from $\{w_{j,R,a}\}$  
by a Gram--Schmidt process, that is 
\[ 
\tilde w_{j,R,a}:= \dfrac{\hat w_{j,R,a}}{\|\hat w_{j,R,a}\|_{L^2(\Omega,\C)}}, \quad j=1,\ldots, n_0, 
\]
where
\begin{align*}
& \hat w_{n_0,R,a} := w_{n_0,R,a},\\
& \hat w_{j,R,a} := w_{j,R,a} - \sum_{\ell= j+1}^{n_0}c_{\ell,j}^{R,a} \hat w_{\ell,R,a} 
\quad \text{for }j=1,\ldots, n_0-1,
\end{align*}
where
\[
c_{\ell,j}^{R,a} :=\dfrac{\int_\Omega
  w_{j,R,a}
  \overline{\hat w_{\ell,R,a}}\,dx}{\|\hat w_{\ell,R,a}\|_{L^2(\Omega,\C)}^2}.
 \]
From \eqref{eq:23}, \eqref{eq:36_la}, and \eqref{eq:36vint_la}
and an induction argument, it follows that 
\begin{equation}\label{eq:18_la}
\|\hat w_{j,R,a} \|_{L^2(\Omega,\C)}^2=1+O(|a|^{3})
\quad\text{and}\quad 
c_{\ell,j}^{R,a}=O(|a|^{3})\text{ for }\ell\neq j
\end{equation}
as $|a|\to 0^+$.
Furthermore, from \eqref{eq:23}, \eqref{eq:vkext_la}, and \eqref{eq:vkint_la} we deduce that 
\begin{equation}\label{eq:13_la}
\|\hat w_{n_0,R,a} \|_{L^2(\Omega,\C)}^2=\|w_{n_0,R,a} \|_{L^2(\Omega,\C)}^2=1+O\big(|a|^{2+k})\big)
\quad\text{as }|a|\to 0^+,
\end{equation}
and 
\begin{equation}\label{eq:76bis}
\quad 
c_{n_0,j}^{R,a}= O\big(|a|^{\frac52+\frac k2})\big)
\quad\text{as }|a|\to 0^+,\quad\text{for all }j<n_0.
\end{equation}
From \eqref{eq:91_la} and \eqref{eq:90_la} it follows that 
\begin{equation*}
  \lambda_a \leq \max_
  {\substack{(\alpha_1,\dots, \alpha_{n_0})\in \C^{n_0}\\
      \sum_{j=1}^{{n_0}}|\alpha_j|^2 =1}}
  \int_{\Omega} \bigg|(i\nabla+A_a) \bigg(\sum_{j=1}^{{n_0}}\alpha_j
     \tilde w_{j,R,a}\bigg)\bigg|^2 dx.
\end{equation*}
Hence 
\begin{equation}\label{eq:107_la}
\lambda_a-\lambda_0\leq \max_
  {\substack{(\alpha_1,\dots, \alpha_{n_0})\in \C^{n_0}\\
      \sum_{j=1}^{{n_0}}|\alpha_j|^2 =1}}\sum_{j,n=1}^{n_0}
p_{j,n}^{a,R}\alpha_j \overline{\alpha_n},
\end{equation}
where 
\[
p_{j,n}^{a,R}= \int_{\Omega} (i\nabla+A_a)    \tilde w_{j,R,a}\cdot
\overline{(i\nabla+A_a)    \tilde w_{n,R,a}}\, dx
-\lambda_0 \delta_{jn}.
\]
By \eqref{eq:85_la}, \eqref{eq:vkext_la}, \eqref{eq:vkint_la}, \eqref{eq:13_la}, and
integration by parts we obtain that
\begin{align*}
 p_{n_0,n_0}^{a,R} &= \dfrac{\lambda_0 (1-\|w_{n_0,R,a} \|_{L^2(\Omega,\C)}^2)}{\|w_{n_0,R,a} \|_{L^2(\Omega,\C)}^2} \\
\notag&\  + \dfrac{1}{\|w_{n_0,R,a} \|_{L^2(\Omega,\C)}^2}
|a|^k\bigg(\int_{D_R}|(i\nabla+A_{\mathbf
  e})U_a^R|^2\,dx-\int_{D_R}|(i\nabla+A_{0})W_a|^2\,dx\bigg)\\ 
&=|a|^k|\beta|^2\bigg(\int_{D_R}|(i\nabla+A_{\mathbf
    e})w_R|^2\,dx-\int_{D_R}|\nabla \psi_k|^2\,dx+o(1)\bigg)\\
  \notag&=-i |a|^k|\beta|^2\big(\tilde\kappa_R+o(1)\big),
\end{align*}
as $|a|\to0$, where 
\[
\tilde\kappa_R=\int_{\partial D_R}\Big(e^{-\frac i2 \theta_{\mathbf
    e}} (i\nabla+A_{\mathbf e})w_R\cdot\nu
- (i\nabla) \psi_k\cdot\nu
\Big)\psi_k\,ds.
\]
From \eqref{eq:35_la}, \eqref{eq:35vint_la}, and \eqref{eq:18_la}, we have that, for all $j<n_0$, 
\begin{align*}
  p_{j,j}^{a,R}&= -\lambda_0 \\
  \notag &\ +\dfrac{1}{\|\hat w_{j,R,a}\|_{L^2(\Omega,\C)}^2} \left(
    \lambda_j^0 - \int_{D_{R|a|}}\!\! \!\abs{(i\nabla
      +A_0)\varphi_j^0}^2dx +
    \int_{D_{R|a|}}\! \!\!\abs{(i\nabla +A_a)w_{j,R,a}^{int}}^2dx \right) \\
  & \ + \dfrac{1}{\|\hat w_{j,R,a}\|_{L^2(\Omega,\C)}^2} \int_{\Omega}
  \bigg|(i\nabla + A_a)
  \Big(\sum_{\ell>j} c_{\ell,j}^{R,a} \hat w_{\ell,R,a} \Big)\bigg|^2dx\\
  &\ - \dfrac{2}{\|\hat w_{j,R,a}\|_{L^2(\Omega,\C)}^2}
  \Re\bigg(\int_{\Omega} (i\nabla + A_a)w_{j,R,a} \cdot
  \overline{(i\nabla + A_a) \Big( \sum_{\ell>j} c_{\ell,j}^{R,a} \hat w_{\ell,R,a} \Big)}\,dx\bigg)\\
  \notag&=(\lambda_j^0-\lambda_0)+o(1)\quad\text{as }|a|\to0.
\end{align*}
From \eqref{eq:35_la}, \eqref{eq:35vint_la}, \eqref{eq:vkext_la}, \eqref{eq:vkint_la},
 \eqref{eq:18_la}, 
and \eqref{eq:76bis} it follows that, for all $j<n_0$, 
\begin{align*}
  p_{j,n_0}^{a,R}&=\overline{p_{n_0,j}^{a,R}}\\
  &= \dfrac{\int_{D_{R|a|}}\Big((i\nabla+A_a) w_{j,R,a}^{int}\cdot
    \overline{(i\nabla+A_a) w_{n_0,R,a}^{int} }- (i\nabla+A_0)
    \varphi_j^0\cdot
    \overline{(i\nabla+A_0) \varphi_0}\Big)\,dx}{\|\hat w_{j,R,a}\|_{L^2(\Omega,\C)} \|\hat w_{n_0,R,a}\|_{L^2(\Omega,\C)}} \\
  &\quad - \dfrac{\int_{\Omega} (i\nabla+A_a)
    \big(\sum_{\ell>j}c_{\ell,j}^{R,a} \hat w_{\ell,R,a}\big) \cdot
    \overline{(i\nabla+A_a) w_{n_0,R,a} }\,dx}{\|\hat
    w_{j,R,a}\|_{L^2(\Omega,\C)} \|\hat
    w_{n_0,R,a}\|_{L^2(\Omega,\C)}}
  \\
  &=O\Big(|a|^{\frac{k+1}2}\Big),
\end{align*}
and, in a similar way, for all
 $j,n<n_0$ with $j\neq n$,
\begin{equation*}
p_{j,n}^{a,R}=O(|a|) \quad\text{as }|a|\to0.
\end{equation*}
Arguing as in \S \ref{sec:rayl-quot-lambd}, we can then prove that 
\[
\max_
{\substack{(\alpha_1,\dots, \alpha_{n_0})\in \C^{n_0}\\
    \sum_{j=1}^{{n_0}}|\alpha_j|^2 =1}}\sum_{j,n=1}^{n_0}
p_{j,n}^{a,R}\alpha_j
\overline{\alpha_n}=|a|^k\Big(-i|\beta|^2\tilde\kappa_R+o(1)\Big),
\quad\text{as }|a|\to0,
\]
which, in view of \eqref{eq:107_la}, yields
$\frac{\lambda_0-\lambda_a}{|a|^k}\geq g_R(a)$ where $\lim_{|a|\to
  0}g_R(a)=i|\beta|^2\tilde\kappa_R$.  The proof is thereby complete.
\end{proof}

\begin{Lemma}\label{l:lim_tilde_kappa_R}
 Let $\tilde \kappa_R$ be as in \eqref{eq:tildekappa_R}. Then
 \begin{equation}
  \lim_{R\to+\infty} \tilde\kappa_R = 4i\mathfrak{m}_k,
 \end{equation}
with $\mathfrak{m}_k$ as in \eqref{eq:Ik}.
\end{Lemma}
\begin{proof}
 We claim that 
\begin{equation}\label{eq:123}
  \tilde \kappa_R=ik\sqrt\pi(\sqrt\pi-\xi(1))+o(1),\quad\text{as }R\to+\infty,
\end{equation}
where
\begin{equation}\label{eq:coeff_fourier_Psi1} 
  \xi(r):= \int_{0}^{2\pi} e^{\frac{i}{2}(\theta_0-\theta_{\mathbf
      e})(r\cos t,r\sin t) } \Psi_k(r\cos t,r\sin t)
  \overline{\psi^k_2(t)}\,dt,\quad r\geq1. 
\end{equation}
To prove claim \eqref{eq:123}, we note that, according to \eqref{coeff_fourier_centrata0} and
\eqref{eq:fourier_centrata0},  the function $\upsilon_R$ defined as 
\begin{equation*}
\upsilon_R(r):= 
\int_0^{2\pi}  w_R(r(\cos t,\sin t))e^{-\frac i2
  \theta_{\mathbf e}(r\cos t,r\sin t)}e^{i\frac t2}
\overline{\psi_2^k(t)}\,dt,\quad r\in[1,R],
\end{equation*}
satisfies, for some $c_{R}\in\C$, $\big( r^{-k/2}\upsilon_R(r) \big)'=
\frac{c_{R}}{r^{1+k}}$ in  $(1,R)$. 
Integrating the previous equation over $(1,r)$ we obtain 
\begin{equation}\label{eq:125}
 r^{-k/2} \upsilon_R(r) -\upsilon_R(1) =
\frac{c_{R}}{k}\bigg(1-\frac1{r^k}\bigg), \quad \text{for all }r\in(1,R]. 
\end{equation}
We notice that, in view of \eqref{eq:9} and \eqref{eq:psi_k},
\begin{equation}\label{eq:126}
\psi_k (r\cos t,r\sin t)= \sqrt\pi r^{k/2} e^{-\frac{i}{2}t}
\psi^k_2(t), \quad \text{for all }t\in(0,2\pi)\text{ and }r>0. 
\end{equation}
Since \eqref{eq:equazione_w_R} and \eqref{eq:126} imply that 
\[
\upsilon_R(R)=\sqrt\pi R^{k/2},
\]
from \eqref{eq:125} we deduce that 
\[
c_R=k\,\frac{R^k}{R^k-1}(\sqrt\pi-\upsilon_R(1))
\]
and then 
\begin{align*}
  \upsilon_R(r) &= r^{k/2} \upsilon_R(1) + r^{k/2} \frac{R^k
    (\sqrt\pi-\upsilon_R(1))}{R^k-1}
  \bigg(1-\frac1{r^k}\bigg)\\
  &=r^{k/2}
  \frac{R^k\sqrt\pi-\upsilon_R(1)}{R^k-1}-r^{-k/2}\frac{R^k(\sqrt\pi-\upsilon_R(1))}{R^k-1},
  \quad \text{for all }r\in(1,R].
\end{align*}
By differentiation of the previous identity, we obtain that 
\begin{equation}\label{eq:127}
\upsilon_R'(R)=\frac k2\,\frac{R^{\frac{k}{2}-1}}{R^k-1}\Big((R^k+1)\sqrt\pi-2\upsilon_R(1)\Big).
\end{equation}
On the other hand, writing $\upsilon_R$ as 
\[
\upsilon_R(r)=\frac1r\int_{\partial D_r}
w_R(x)e^{-\frac i2(\theta_{\mathbf e}-\theta_0)(x)}\,
\overline{\psi^k_2(\theta_0(x))}\,ds(x),
\]
differentiating and using \eqref{eq:126}, we obtain that
\begin{equation}\label{eq:128}
  \upsilon_R'(r)=
  -\frac{i}{\sqrt\pi}r^{-1-\frac k2}
  \int_{\partial D_r}e^{-\frac{i}{2}\theta_{\mathbf
      e}}(i\nabla+A_{\mathbf e})w_R\cdot\nu\, \psi_k\, ds.
\end{equation}
Combination of \eqref{eq:127} and \eqref{eq:128} yields
\begin{equation}\label{eq:129}
  \int_{\partial D_R}e^{-\frac{i}{2}\theta_{\mathbf
      e}}(i\nabla+A_{\mathbf e})w_R\cdot\nu\, \psi_k\, ds
  =\frac{i k\sqrt\pi}{2}\frac{R^k}{R^k-1}\Big((R^k+1)\sqrt\pi-2\upsilon_R(1)\Big).
\end{equation}
Moreover, \eqref{eq:psi_k} directly gives 
\begin{equation}\label{eq:130}
\int_{\partial D_R}(i\nabla) \psi_k\cdot\nu\psi_k\,ds=\frac k2\, iR^k\pi.
\end{equation}
From \eqref{eq:129}, \eqref{eq:130}, and \eqref{eq:tildekappa_R}, it follows that 
\[
\tilde\kappa_R=\frac{ik\sqrt\pi
  R^k}{R^k-1}\big(\sqrt\pi-\upsilon_R(1)\big).
\]
Since Lemma \ref{l:conv_w_r} and \eqref{eq:coeff_fourier_Psi1} imply that 
\[
\lim_{R\to+\infty}\upsilon_R(1)=\xi(1),
\]
we obtain claim \eqref{eq:123}. The conclusion follows by combining
\eqref{eq:123} and the identity  
\begin{equation}\label{eq:43}
  \sqrt\pi-\xi(1)=\frac{4}{k\sqrt\pi} {\mathfrak m}_k,
\end{equation}
which results from Lemma \ref{l:legame_compl}.
\end{proof}
Combining Lemma \ref{l:stima_Lambda0_sotto} and Lemma
\ref{l:lim_tilde_kappa_R} we deduce the following
result.
\begin{Proposition}\label{prop:stima_sopra}
 For $a=(|a|,0)\in\Omega$, let 
$\varphi_a\in
H^{1,a}_{0}(\Omega,\C)$ solve (\ref{eq:equation_a}-\ref{eq:6})
and $\varphi_0\in
H^{1,0}_{0}(\Omega,\C)$ be a solution to
(\ref{eq:equation_lambda0}--\ref{eq:83}). If \eqref{eq:1} and
\eqref{eq:37} hold and \eqref{eq:54} is satisfied, then
\[
\liminf_{|a|\to0}
\frac{\lambda_0-\lambda_a}{|a|^k}\geq -4|\beta|^2 \mathfrak{m}_k>0
\]
with $\beta$ as in \eqref{eq:beta} and 
$\mathfrak{m}_k$ as in (\ref{eq:Ik}-\ref{eq:segno_mk}). 
\end{Proposition}
\begin{remark}\label{rem:segno}
  As a consequence of Proposition \ref{prop:stima_sopra}, we have
  that, if $a\in\Omega$ approaches $0$ along the half-line tangent to
  a nodal line of eigenfunctions associated to the simple eigenvalue
  $\lambda_0$, then
\[
\lambda_a<\lambda_0.
\]
\end{remark}

Combining Corollary \ref{cor:lower-bound-lambda_a} with Proposition
\ref{prop:stima_sopra} we obtain the following upper and lower
estimates for $\lambda_0-\lambda_a$.

\begin{Proposition}\label{prop:quasi_finito1}
  For $a=(|a|,0)=|a|{\mathbf e}\in\Omega$, let $\varphi_a\in
  H^{1,a}_{0}(\Omega,\C)$ solve (\ref{eq:equation_a}-\ref{eq:6}) and
  $\varphi_0\in H^{1,0}_{0}(\Omega,\C)$ be a solution to
  (\ref{eq:equation_lambda0}--\ref{eq:83}). Let \eqref{eq:1},
  \eqref{eq:37}, and \eqref{eq:54} hold. Then there exists a positive
  constant $C^*>0$ such that
\[
-4|\beta|^2 \mathfrak{m}_k\,|a|^k (1+o(1))
\leq \lambda_0 - \lambda_{a}
\leq C^*\, H(\varphi_{a}, K_\delta |a|), 
\quad \text{as }|a|\to 0,
\]
with $\beta$ as in \eqref{eq:beta} and 
$\mathfrak{m}_k<0$ as in (\ref{eq:Ik}-\ref{eq:segno_mk}). 
\end{Proposition}

\section{Energy estimates}\label{sec:7}

To obtain our main result, we aim at proving that the difference of
the eigenvalues $\lambda_0 - \lambda_a$ is estimated even from above
by the rate $|a|^k$, i.e. we have to determine the exact asymptotic
behavior of the normalization term in
\eqref{def_blowuppate_normalizzate}, i.e. of $\sqrt{H(\varphi_{a},
  K_\delta |a|)}$.  To this purpose, in this section we obtain some
preliminary energy estimates of the difference between approximating
and limit eigenfunctions after blow-up, based on the invertibility of
the differential of the function $F$ defined below.

Throughout this section, we will treat the space
$H^{1,0}_0(\Omega,\C)$ defined in \S \ref{sec:introduction} as a real
Hilbert space endowed with the scalar product
\[
(u,v)_{H^{1,0}_{0,\R}(\Omega,\C)}=\Re\bigg(\int_\Omega (i\nabla+A_0) u
\cdot \overline{(i\nabla+A_0) v}\,dx\bigg),
\]
which induces on  $H^{1,0}_0(\Omega,\C)$ the norm
\[
\|u\|_{H^{1,0}_0(\Omega,\C)}=\bigg(
\int_{\Omega}\big|(i\nabla +A_0)u\big|^2dx\bigg)^{\!\!1/2}
\]
which is 
 equivalent to the norm \eqref{eq:norma} (see Lemma \ref{poincare_inequality}).
To emphasize the fact that here $H^{1,0}_0(\Omega,\C)$ is meant as a
vector space over $\R$ we denote it as $H^{1,0}_{0,\R}(\Omega,\C)$. We
will denote as  $(H^{1,0}_{0,\R}(\Omega,\C))^\star$ the real dual space of  $H^{1,0}_{0,\R}(\Omega,\C)$.

Let us consider the function
\begin{align}\label{def_operatore_F}
  F: \C \times H^{1,0}_{0,\R}(\Omega,\C)  &\longrightarrow  \R \times \R \times (H^{1,0}_{0,\R}(\Omega,\C))^\star\\
  \notag (\lambda,\varphi) &\longmapsto \Big( {\textstyle{
      \|\varphi\|_{H^{1,0}_0(\Omega,\C)}^2 -\lambda_0,\
      \mathfrak{Im}\big(\int_{\Omega}
      \varphi\overline{\varphi_0}\,dx\big), \ (i\nabla +A_0)^2
      \varphi-\lambda \varphi}}\Big),
\end{align}
   where  $(i\nabla +A_0)^2  \varphi-\lambda \varphi\in (H^{1,0}_{0,\R}(\Omega,\C))^\star$ acts as 
\[
\phantom{a}_{(H^{1,0}_{0,\R}(\Omega,\C))^\star}\!\Big\langle (i\nabla
+A_0)^2 \varphi-\lambda \varphi , u
\Big\rangle_{\!H^{1,0}_{0,\R}(\Omega,\C)}\!\!=\mathfrak{Re}
\left({\textstyle{\int_{\Omega}(i\nabla+A_0)\varphi\cdot\overline{(i\nabla+A_0)u}\,dx
      -\!\lambda \!\int_{\Omega} \varphi\overline{u} \,dx}}\right)
\]
for all $\varphi\in H^{1,0}_{0,\R}(\Omega,\C)$.  In
\eqref{def_operatore_F} $\C$ is also meant as a vector space over
$\R$. From \eqref{eq:equation_lambda0} and \eqref{eq:83} it follows
that $F(\lambda_0,\varphi_0)=(0,0,0)$.
\begin{Lemma}\label{F_frechet}
  Under assumptions \eqref{eq:1}, \eqref{eq:equation_lambda0} and
  \eqref{eq:83}, the function $F$ defined in \eqref{def_operatore_F}
  is Frech\'{e}t-differentiable at $(\lambda_0,\varphi_0)$ and its
  Frech\'{e}t-differential
\[
dF(\lambda_0,\varphi_0)\in \mathcal L\big( \C \times
H^{1,0}_{0,\R}(\Omega,\C),\R\times \R \times (H^{1,0}_{0,\R}(\Omega,\C))^\star\big)
\]
is invertible.
\end{Lemma}
\begin{proof}
  By direct calculations it is easy to verify that $F$ is
  Frech\'{e}t-differentiable at $(\lambda_0,\varphi_0)$ and
\begin{multline*}
  dF(\lambda_0,\varphi_0)(\lambda,\varphi)\\
  = \bigg( 2
  \,\mathfrak{Re}\left({\textstyle{\int_{\Omega}(i\nabla+A_0)
        \varphi_0\cdot\overline{(i\nabla+A_0)\varphi} \,dx}}\right) ,
  \mathfrak{Im}\left({\textstyle{\int_{\Omega}
        \varphi\overline{\varphi_0}\,dx}}\right), (i\nabla+A_0)^2
  \varphi -\lambda_0\varphi -\lambda \varphi_0 \bigg)
\end{multline*}
for every $(\lambda,\varphi)\in \C \times  H^{1,0}_{0,\R}(\Omega,\C)$.

It remains to prove that $dF(\lambda_0,\varphi_0):\C \times
H^{1,0}_{0,\R}(\Omega,\C)\to\R \times\R \times
(H^{1,0}_{0,\R}(\Omega),\C)^\star$ is invertible.  To this aim, by
exploiting the compactness of the map
\[
T:H^{1,0}_{0,\R}(\Omega,\C)\to (H^{1,0}_{0,\R}(\Omega,\C))^\star,\quad  
u\mapsto \lambda_0 u,
\]
 it is easy to prove that, if
\[
\mathcal R:(H^{1,0}_{0,\R}(\Omega,\C))^\star\to
H^{1,0}_{0,\R}(\Omega,\C)
\]
 is the Riesz isomorphism and
$\mathcal I$ denotes the standard identification of $\R\times\R$ onto $\C$,
then the operator $(\mathcal I \times \mathcal R)\circ
dF(\lambda_0,\varphi_0)\in\mathcal L(\C \times H^{1,0}_{0,\R}(\Omega))$ is a compact
perturbation of the identity. Indeed, 
since by definition
\[
\phantom{a}_{(H^{1,0}_{0,\R}(\Omega))^\star}\big\langle (i\nabla
+A_0)^2 \varphi , u \big\rangle_{H^{1,0}_{0,\R}(\Omega)}=\mathfrak{Re}
\left(\int_{\Omega}(i\nabla+A_0)\varphi\cdot\overline{(i\nabla+A_0)u}\,dx
\right) = \big(\varphi,u\big)_{H^{1,0}_{0,\R}(\Omega,\C)},
\]
we have that
\[ 
\mathcal R \big((i\nabla+A_0)^2\varphi - \lambda_0\varphi -
\lambda\varphi_0 \big)= \varphi - \mathcal R(\lambda_0\varphi) -
\mathcal R(\lambda\varphi_0) 
\] 
being $\mathcal R(\lambda_0\varphi)$
the image of $\varphi$ by a compact operator (composition of the Riesz
isomorphism and the compact operator $T$), as well as $\mathcal R(\lambda\varphi_0)$.

Therefore, from the Fredholm
alternative, $dF(\lambda_0,\varphi_0)$ is invertible if and only if it is
injective. So, to conclude the proof, it is enough to prove that
$\ker(dF(\lambda_0,\varphi_0))=\{0,0\}$.
Let $(\lambda,\varphi)\in \C \times  H^{1,0}_{0,\R}(\Omega,\C)$ be such that 
\begin{equation}\label{eq:inv}
\begin{cases}
  \mathfrak{Re}\left(\int_{\Omega}(i\nabla+A_0) \varphi_0\cdot\overline{(i\nabla+A_0)\varphi} \,dx\right)=0\\
  \mathfrak{Im}\left(\int_{\Omega} \varphi\overline{\varphi_0}\,dx\right)=0\\
  (i\nabla+A_0)^2 \varphi -\lambda_0\varphi -\lambda \varphi_0 =0.
\end{cases}
\end{equation}
The last equation in \eqref{eq:inv} means that
\[
\mathfrak{Re}\left(\int_{\Omega}\Big((i\nabla+A_0)
  \varphi\cdot\overline{(i\nabla+A_0)u}-\lambda_0
 \varphi\overline{u} -\lambda
  \varphi_0\overline{u}\Big) \,dx\right) =0, \quad \text{for
  all }u\in H^{1,0}_{0,\R}(\Omega,\C).
\]
Plugging $u=\varphi_0$ and $u=i\varphi_0$ into the previous
identity and recalling \eqref{eq:equation_lambda0} and \eqref{eq:83},
we obtain $\mathfrak{Re}\lambda=0$ and $\mathfrak{Im}\lambda=0$,
respectively. Then the last equation in \eqref{eq:inv} becomes
\[
(i\nabla+A_0)^2 \varphi -\lambda_0\varphi =0, \quad \text{in }(H^{1,0}_{0,\R}(\Omega,\C))^\star,
\]
which, by assumption \eqref{eq:1}, implies that $\varphi =
(\alpha+i\beta) \varphi_0$ for some $\alpha,\beta\in\R$.  The first
and the second equation in \eqref{eq:inv} imply $\alpha =0$ and
$\beta=0 $, respectively, so that $\varphi=0$. Then we conclude that
the only element in the kernel of $dF(\lambda_0,\varphi_0)$ is
$(0,0)\in \C\times H^{1,0}_{0,\R}(\Omega,\C)$.
\end{proof}

\begin{Theorem}\label{stima_teo_inversione}
  For every $R>2$ and $a=(|a|,0)\in\Omega$ with $|a|<{R_0}/{R}$ and
  $R_0$ being as in Lemma \ref{l:Hpos}, let 
$\varphi_a\in
H^{1,a}_{0}(\Omega,\C)$ solve (\ref{eq:equation_a}-\ref{eq:6}), $\varphi_0\in
H^{1,0}_{0}(\Omega,\C)$ be a solution to
(\ref{eq:equation_lambda0}--\ref{eq:83}) satisfying  \eqref{eq:1},
\eqref{eq:37}, and \eqref{eq:54}, and
$v_{n_0,R,a}$ be as in \S
  \ref{sec:rayl-quot-lambd} (see also \eqref{eq:viext} and
  \eqref{eq:viint}). 
 Then, for every  $R>2$,
\begin{equation*}
\|v_{n_0,R,a} -
\varphi_0\|_{H^{1,0}_0(\Omega,\C)}=O\Big(\sqrt{H(\varphi_{a},K_\delta|a|)}\Big) \quad \text{as }|a|\to0^+,
\end{equation*}
with  $K_\delta$ as in Lemma \ref{lemma_limitatezza_N_per_blowup}
for some fixed $\delta\in(0,1/4)$.
\end{Theorem}
\begin{proof}
Let $R>2$. We first notice that $v_{n_0,R,a}\to \varphi_0$ in
$H^{1,0}_0(\Omega,\C)$ as $|a|\to 0^+$. Indeed, recalling
definitions \eqref{def_blowuppate_normalizzate}, \eqref{eq:85} and \eqref{eq:85_la}, we have that
\begin{align*}
  & \int_{\Omega}\big|(i\nabla+A_0)(v_{n_0,R,a}-\varphi_0)\big|^2\,dx=
  \int_{\Omega}| e^{\frac{i}{2}(\theta_0 - \theta_a)}
(i\nabla+A_a)\varphi_a-
  (i\nabla+A_0)\varphi_0|^2\,dx \\
  &\notag\qquad +
  H(\varphi_a,K_\delta|a|)\int_{D_{R}}\Big|(i\nabla+A_0)\Big(Z_a^{R}-
  \tfrac{|a|^{k/2}}{\sqrt{H(\varphi_a,K_\delta|a|)}} W_a
  \Big)\Big|^2\,dx\\
  &\notag\qquad -
  H(\varphi_a,K_\delta|a|)\int_{D_{R}}\Big| e^{\frac{i}{2}(\theta_0 -
    \theta_{\mathbf e})} (i\nabla+A_{\mathbf
    e})\tilde\varphi_a-\tfrac{|a|^{k/2}}{\sqrt{H(\varphi_a,K_\delta|a|)}}
  (i\nabla+A_0)W_a \Big)\Big|^2\,dx.
\end{align*} 
Estimate \eqref{eq:stima_sopra_radiceH} implies that $H(\varphi_a,
K_\delta |a|)=o(1)$ whereas Proposition \ref{prop:quasi_finito1}
implies that 
$|a|^{k/2}(H(\varphi_a,K_\delta|a|))^{-1/2}=O(1)$ as $|a|\to 0^+$.
Then Theorem \ref{t:stime_blowup},
Lemma \ref{stime_lemma_blow-up}, \eqref{eq:vkext_la}, and \eqref{eq:45}
imply that $v_{n_0,R,a}\to \varphi_0$ in
$H^{1,0}_0(\Omega,\C)$ as $|a|\to 0^+$.

Therefore, from Lemma \ref{F_frechet}, we have that 
\begin{equation}\label{eq:37_bi}
 F(\lambda_a,v_{n_0,R,a}) = dF(\lambda_0,\varphi_0) (\lambda_a - \lambda_0, v_{n_0,R,a} - \varphi_0)
+ o\big(|\lambda_a - \lambda_0| + \|v_{n_0,R,a} -
\varphi_0\|_{H^{1,0}_0(\Omega,\C)}\big)
\end{equation}
 as $|a|\to 0^+$.
In view of Lemma \ref{F_frechet}, the operator $dF(\lambda_0,\varphi_0)$ is
invertible (and its inverse is continuous by the Open Mapping
Theorem), then \eqref{eq:37_bi} implies that 
\begin{multline*}
  |\lambda_{a} - \lambda_0| + \|v_{n_0,R,a} -
  \varphi_0\|_{H^{1,0}_0(\Omega,\C)}\\
  \leq\|(dF(\lambda_0,\varphi_0))^{-1}\|_{ \mathcal L( \R\times \R
    \times (H^{1,0}_{0,\R}(\Omega,\C))^\star,\C \times
    H^{1,0}_{0,\R}(\Omega,\C))} \| F(\lambda_a,v_{n_0,R,a})\|_{
    \R\times\R \times (H^{1,0}_{0,\R}(\Omega))^\star} (1+o(1))
\end{multline*}
as $|a|\to 0^+$.
In order to prove the theorem, it remains to estimate the norm of
\begin{align}\label{eq:42}
  F&(\lambda_a,v_{n_0,R,a}) =\left( \alpha_a, \beta_a, w_a\right)
  \\
  \notag& = \left( \|v_{n_0,R,a}\|_{H^{1,0}_0(\Omega,\C)}^2
    -\lambda_0,
    \mathfrak{Im}\left({\textstyle{\int_{\Omega}v_{n_0,R,a}\overline{\varphi_0}\,dx}}\right),
    (i\nabla+A_0)^2 v_{n_0,R,a} - \lambda_a v_{n_0,R,a} \right)
\end{align}
in $\R\times\R \times (H^{1,0}_{0,\R}(\Omega))^\star$.
As far as $\alpha_a$ is concerned, arguing as in \eqref{eq:95}, we
have  that, in view of  \eqref{def_blowuppate_normalizzate},
\eqref{eq:85}, Theorem \ref{t:stime_blowup},
Lemma \ref{stime_lemma_blow-up}, and Proposition \ref{prop:quasi_finito1},
\begin{align*}
 \alpha_a 
 &= \left(
 \int_{ D_{R|a|}} |(i\nabla+A_0)v_{n_0,R,a}^{int}|^2 \,dx-
 \int_{D_{R|a|}} |(i\nabla+A_a)\varphi_a|^2\,dx 
 \right) +(\lambda_a-\lambda_0)\\
 &=  H(\varphi_a,K_\delta|a|)
\left(
 \int_{ D_{R}} |(i\nabla+A_0)Z_a^R|^2 \,dx-
 \int_{D_{R}} |(i\nabla+A_{\mathbf e})\tilde\varphi_a|^2\,dx 
 \right) +(\lambda_a-\lambda_0)
\\
 &= O(H(\varphi_a,K_\delta|a|)) \quad \text{as }|a|\to0^+.
\end{align*}
As far as $\beta_a$ is concerned, by Theorem \ref{t:stime_blowup}, Lemma
\ref{stime_lemma_blow-up}, \eqref{eq:vkext_la}, and the normalization
condition \eqref{eq:6} required on $\varphi_a$, we have that 
\begin{align*}
  \beta_a &= \Im \left( \int_{D_{R|a|}}
    v_{n_0,R,a}^{int}\overline{\varphi_0} \,dx- \int_{D_{R|a|}}
    e^{\frac i2 (\theta_0-\theta_a)} \varphi_a\overline{\varphi_0}\,dx
    +\int_{\Omega} e^{\frac i2 (\theta_0-\theta_a)}
    \varphi_a\overline{\varphi_0}\,dx
  \right)\\
  &= \sqrt{H(\varphi_a,K_\delta|a|)} |a|^{\frac k2 +2} \Im \left(
    \int_{D_{R}} Z_a^R\overline{W_a}\,dx
    - \int_{D_{R}} e^{\frac i2 (\theta_0-\theta_{\mathbf e})} \tilde\varphi_a \overline{W_a}\,dx \right)\\
  &= o\big(\sqrt{H(\varphi_a,K_\delta|a|)}\big) \quad \text{as }|a|\to
  0^+.
\end{align*}
Let $\eta_{a,R}:\R^2\to\R$ be a smooth cut-off function such that
$\eta_{a,R}(x)=1$ if $|x|\geq R|a|$, $\eta_{a,R}(x)=0$ if $|x|\leq
\frac{R|a|}{2}$, $0\leq \eta_{a,R}(x)\leq 1$ for all $x$, and $|\nabla
\eta_{a,R}(x)|\leq \frac4{R|a|}$.  Then, for every $\varphi\in
H^{1,0}_{0}(\Omega,\C)$ we have that $\eta_{a,R}e^{\frac
  i2(\theta_a-\theta_0)}\varphi\in H^{1,a}_{0}(\Omega,\C)$. Hence
testing equation \eqref{eq:equation_a} with $\eta_{a,R}e^{\frac
  i2(\theta_a-\theta_0)}\varphi$ we obtain that
\begin{multline*}
  \int_{\Omega\setminus D_{R|a|}}e^{\frac i2(\theta_0-\theta_a)}
  (i\nabla+A_a)\varphi_a\cdot\overline{(i\nabla+A_0) \varphi
  }\,dx-\lambda_a \int_{\Omega\setminus D_{R|a|}}e^{\frac
    i2(\theta_0-\theta_a)} \varphi_a\overline{\varphi
  }\,dx \\
  = -
  \int_{D_{R|a|}}(i\nabla+A_a)\varphi_a\cdot\overline{(i\nabla+A_a)
    \big(\eta_{a,R}e^{\frac i2(\theta_a-\theta_0)}\varphi\big) }\,dx
  +\lambda_a \int_{D_{R|a|}}\varphi_a\eta_{a,R}e^{\frac
    i2(\theta_0-\theta_a)}\overline{\varphi}\,dx\\
  = -
  \int_{D_{R|a|}}(i\nabla+A_a)\varphi_a\cdot\overline{(i\nabla+A_0)\varphi}
  \eta_{a,R}e^{\frac i2(\theta_0-\theta_a)}\,dx -i
  \int_{D_{R|a|}}(i\nabla+A_a)\varphi_a\cdot\nabla\eta_{a,R}
  \overline{\varphi}
  e^{\frac i2(\theta_0-\theta_a)}\,dx\\
  +\lambda_a \int_{D_{R|a|}}\varphi_a\eta_{a,R}e^{\frac
    i2(\theta_0-\theta_a)}\overline{\varphi}\,dx
\end{multline*}
and hence, by H\"older inequality and \eqref{eq:hardy}, 
\begin{align}\label{eq:51}
  & \bigg| \int_{\Omega\setminus D_{R|a|}}e^{\frac
    i2(\theta_0-\theta_a)}
  (i\nabla+A_a)\varphi_a\cdot\overline{(i\nabla+A_0) \varphi
  }\,dx-\lambda_a \int_{\Omega\setminus D_{R|a|}}e^{\frac
    i2(\theta_0-\theta_a)} \varphi_a\overline{\varphi
  }\,dx\bigg|\\
  \notag\leq \bigg(&\int_{D_{R|a|}}\!\!|
  (i\nabla+A_a)\varphi_a|^2\,dx\bigg)^{\!\!\frac12} \bigg[
  \bigg(\int_{D_{R|a|}}\!\!|
  (i\nabla+A_0)\varphi|^2\,dx\bigg)^{\!\!\frac12}\!\!+
  \bigg(\int_{D_{R|a|}}\!\!|\nabla\eta_{a,R}|^2|\varphi|^2\,dx\bigg)^{\!\!\frac12}
  \bigg]\\
  \notag&\qquad+2\lambda_a R|a| \bigg(\int_{D_{R|a|}}|
  (i\nabla+A_0)\varphi|^2\,dx\bigg)^{\!\!\frac12}
  \bigg(\int_{D_{R|a|}}|\varphi_a|^2\,dx\bigg)^{\!\!\frac12}
  \\
  \notag&\leq \left(9 \bigg(\int_{D_{R|a|}}|
    (i\nabla+A_a)\varphi_a|^2\,dx\bigg)^{\!\!1/2} +2\lambda_a |a|R
    \bigg(\int_{D_{R|a|}}|\varphi_a|^2\,dx\bigg)^{\!\!1/2}\right)
  \|\varphi\|_{H^{1,0}_0(\Omega,\C)}.
\end{align}
 By H\"older inequality and \eqref{eq:hardy}, we also have that 
\begin{align}\label{eq:52}
  \bigg| \int_{D_{R|a|}}& (i\nabla+A_0)
  v_{n_0,R,a}^{int}\cdot\overline{(i\nabla+A_0) \varphi
  }\,dx-\lambda_a \int_{D_{R|a|}}v_{n_0,R,a}^{int} \overline{\varphi
  }\,dx\bigg|\\
  \notag&\leq \bigg(\int_{D_{R|a|}}| (i\nabla+A_0)
  v_{n_0,R,a}^{int}|^2\,dx\bigg)^{\!\!1/2} (1+4\lambda_a|a|^2R^2)
  \|\varphi\|_{H^{1,0}_0(\Omega,\C)}.
\end{align}
From \eqref{eq:51} and \eqref{eq:52} it follows that, for every
$\varphi\in H^{1,0}_{0}(\Omega,\C)$,
\begin{align*}
  \bigg| \Re
  &\left(\int_{\Omega}(i\nabla+A_0)v_{n_0,R,a}\cdot\overline{(i\nabla+A_0)\varphi}\,dx
    -\lambda_a \int_{\Omega} v_{n_0,R,a}\overline{\varphi}
    \,dx\right)\bigg|\\
  &\leq \|\varphi\|_{H^{1,0}_0(\Omega,\C)} \bigg[9
  \bigg(\int_{D_{R|a|}}| (i\nabla+A_a)\varphi_a|^2\,dx\bigg)^{\!\!1/2}
  +2\lambda_a |a|R
  \bigg(\int_{D_{R|a|}}|\varphi_a|^2\,dx\bigg)^{\!\!1/2}\\
  \notag&\hskip4cm +(1+4\lambda_a|a|^2R^2)\bigg(\int_{D_{R|a|}}|
  (i\nabla+A_0) v_{n_0,R,a}^{int}|^2\,dx\bigg)^{\!\!1/2} \bigg].
\end{align*}
Hence, in view of \eqref{eq:46}, \eqref{eq:48}, and \eqref{eq:46zar},
\begin{align*}
  &\|w_a\|_{(H^{1,0}_{0,\R}(\Omega,\C))^\star} \\
  & = \!\!\sup_{\substack{\varphi\in H^{1,0}_{0}(\Omega,\C) \\
      \|\varphi\|_{H^{1,0}_0(\Omega,\C)}=1}} \!\!\bigg|\Re
  \left(\int_{\Omega}(i\nabla+A_0)v_{n_0,R,a}\cdot\overline{(i\nabla+A_0)\varphi}\,dx
    -\lambda_a \int_{\Omega} v_{n_0,R,a}\overline{\varphi}
    \,dx\right)\bigg|\\
  &\leq 9 \bigg(\int_{D_{R|a|}}|
  (i\nabla+A_a)\varphi_a|^2\,dx\bigg)^{\!\!1/2} +2\lambda_a |a|R
  \bigg(\int_{D_{R|a|}}|\varphi_a|^2\,dx\bigg)^{\!\!1/2}\\
  &\qquad\qquad
  +(1+4\lambda_a|a|^2R^2)\bigg(\int_{D_{R|a|}}| (i\nabla+A_0) v_{n_0,R,a}^{int}|^2\,dx\bigg)^{\!\!1/2}\\
  &=O\big(\sqrt{H(\varphi_a, K_\delta|a|)}\big),\quad\text{as
  }|a|\to0^+.
\end{align*}
The proof is thereby complete.
\end{proof}

As a direct consequence of Theorem \ref{stima_teo_inversione}, we
obtain the following uniform energy estimate for scaled
eigenfunctions.
\begin{Theorem}\label{conseguenza_stima_teo_inversione}
  For $a=(|a|,0)\in\Omega$, let $\varphi_a\in H^{1,a}_{0}(\Omega,\C)$
  solve (\ref{eq:equation_a}-\ref{eq:6}), $\varphi_0\in
  H^{1,0}_{0}(\Omega,\C)$ be a solution to
  (\ref{eq:equation_lambda0}--\ref{eq:83}) satisfying \eqref{eq:1},
  \eqref{eq:37}, and \eqref{eq:54}, $\tilde\varphi_a$ be as in
  \eqref{def_blowuppate_normalizzate} and $W_a$ as in
  \eqref{eq:85_la}.  Then, for every $R>2$,
\begin{equation}\label{eq:41}
 \int_{\big(\frac{1}{|a|}\Omega\big)\setminus D_{R}}\bigg|(i\nabla+A_{\mathbf e})
 \Big(\tilde \varphi_a(x) - e^{\frac i2(\theta_{\mathbf e}-\theta_0)}
\tfrac{|a|^{k/2}}{\sqrt{H(\varphi_a,K_\delta|a|)}}W_a\Big)\bigg|^2dx=O(1),
\quad\text{as }|a|\to0^+.
\end{equation}
\end{Theorem}
\begin{proof}
  The proof follows directly from  scaling  and Theorem \ref{stima_teo_inversione}.
\end{proof}

\section{Blow-up analysis}\label{sec:blow-up-analysis}

In this section we study the limit of the blow-up sequence introduced
in \eqref{def_blowuppate_normalizzate}. 

\begin{Theorem}\label{t:blowup}
  For $a=(|a|,0)\in\Omega$, let $\varphi_a\in H^{1,a}_{0}(\Omega,\C)$
  solve (\ref{eq:equation_a}-\ref{eq:6}) and $\varphi_0\in
  H^{1,0}_{0}(\Omega,\C)$ be a solution to
  (\ref{eq:equation_lambda0}--\ref{eq:83}) satisfying \eqref{eq:1},
  \eqref{eq:37}, and \eqref{eq:54}.  Let $\tilde\varphi_a$ and
  $K_\delta$ be as in \eqref{def_blowuppate_normalizzate},
  $\beta_k^2(0,\varphi_0,\lambda_0)$ as in \eqref{eq:84}, and $\Psi_k$
  be the function defined in \eqref{eq:10}.  Then
\begin{equation}\label{eq:58}
  \lim_{|a|\to 0^+}\frac{|a|^{k/2}}{\sqrt{H(\varphi_{a},K_\delta|a|)}}=
  \frac{\sqrt\pi}{|\beta_k^2(0,\varphi_0,\lambda_0)|}
  \sqrt{\frac{K_\delta}{\int_{\partial D_{K_\delta}}|\Psi_k|^2ds}}
\end{equation}
and
\begin{equation}\label{eq:59}
  \tilde \varphi_{a}\to \frac{\beta_k^2(0,\varphi_0,\lambda_0)}{|\beta_k^2(0,\varphi_0,\lambda_0)|}
  \sqrt{\frac{K_\delta}{\int_{\partial D_{K_\delta}}|\Psi_k|^2ds}} \ \Psi_k \quad\text{as
  }|a|\to0^+,
\end{equation}
in $H^{1 ,{\mathbf e}}(D_R,\C)$ for every $R>1$, almost everywhere and
in $C^{2}_{\rm loc}(\R^2\setminus\{{\mathbf e}\},\C)$.
\end{Theorem}
\begin{proof}
  From Theorem \ref{t:stime_blowup} we know that the family of
  functions $\big\{\tilde \varphi_a: a=|a|{\mathbf e},
  |a|<\tfrac{r_\delta}{R}\big\}$ is bounded in $H^{1 ,{\mathbf
      e}}(D_R,\C)$ for all $R\geq K_\delta$. Furthermore, from
  Proposition \ref{prop:quasi_finito1},
\[
\frac{|a|^{k/2}}{\sqrt{H(\varphi_a,K_\delta|a|)}}=O(1)\quad\text{as }|a|\to0^+.
\]
It follows that, for every sequence $a_n=(|a_n|,0)=|a_n|\mathbf e$
with $|a_n|\to 0$, by a diagonal process there exist $c\in
[0,+\infty)$, $\tilde \Phi\in H^{1 ,{\mathbf e}}_{\rm loc}(\R^2,\C)$,
and a subsequence $a_{n_\ell}$ such that
\[
\lim_{\ell\to+\infty}\frac{|a_{n_\ell}|^{k/2}}{\sqrt{H(\varphi_{a_{n_\ell}},K_\delta|a_{n_\ell}|)}}=c
\]
and
\[
\tilde \varphi_{a_{n_\ell}}\rightharpoonup \tilde \Phi
\quad\text{weakly in }H^{1 ,{\mathbf e}}(D_R,\C)\text{  
 for
every $R>1$ and almost everywhere}.
\] 
We notice that $\tilde \Phi\not\equiv 0$ since 
\begin{equation}\label{eq:56}
\frac1{K_\delta}\int_{\partial D_{K_\delta}}|\tilde\Phi|^2\,ds=1
\end{equation}
thanks to \eqref{eq:62} and the compactness of the trace embedding
$H^{1,{\mathbf e}}(D_{K_{\delta}},\C)\hookrightarrow L^2(\partial
D_{K_{\delta}},\C)$.

Multiplying \eqref{eq:equations_for_blowedup} by $\eta\in
C^\infty_{\rm c}(\R^2\setminus\{{\mathbf e}\},\C)$ and integrating by
parts, we have that, if $|a|$ is sufficiently small so that
$\mathop{\rm supp}\eta\subset\frac1{|a|}\Omega$,
\begin{align*}
 \int_{\R^2} (i\nabla+ A_{\mathbf e} )\tilde\varphi_a
 \cdot\overline{(i\nabla+ A_{\mathbf e} )\eta}\,dx 
 =\lambda_a |a|^2 \int_{\R^2} \tilde\varphi_a \overline\eta\,dx.
\end{align*}
Along $a=a_{n_\ell}$ with $\ell\to\infty$, the left hand side converges to
$\int_{\R^2} (i\nabla+ A_{\mathbf e} )\tilde\Phi
\cdot\overline{(i\nabla+ A_{\mathbf e} )\eta}\,dx $ via the weak $H^{1
  ,{\mathbf e}}(D_R,\C)$-convergence, where $R>1$ is such that
$\mathop{\rm supp}\eta\subset D_R$, whereas, in view of \eqref{eq:67},
the right hand side can be estimated as
\[ 
\abs{ \lambda_{a_{n_\ell}} |{a_{n_\ell}}|^2 \int_{\R^2}
  \tilde\varphi_{a_{n_\ell}} \overline\eta\,dx} \leq
\lambda_{a_{n_\ell}} |{a_{n_\ell}}|^2
\|\tilde\varphi_{a_{n_\ell}}\|_{H^{1,{\mathbf e}}(D_{R},\C)}
\|\eta\|_{L^2(\R^2,\C)} =O(|{a_{n_\ell}}|^2) \quad\text{as
}\ell\to\infty,
\]
thus proving that $\tilde \Phi$ weakly solves 
\begin{equation}\label{eq:tildePhi}
(i\nabla +A_{\mathbf e})^2 \tilde\Phi =0, \quad \text{in }\R^2. 
\end{equation}
We now claim that 
the convergence of  the subsequence $\tilde \varphi_{a_{n_\ell}}$ to
$\tilde \Phi$ is actually strong in $H^{1 ,{\mathbf e}}(D_R,\C)$ for
every $R>1$. 

By classical elliptic estimates, we can easily prove that $\tilde
\varphi_{a_{n_l}} \to \tilde \Phi$ in $C^{2,\alpha}(D_{R_2}\setminus
D_{R_1},\C)$ for every $1< R_1 <R_2$.  Therefore, multiplying by $\tilde \Phi$
equation \eqref{eq:tildePhi} and integrating by parts in $D_R$ for
$R>1$, we obtain 
\begin{equation}\label{eq:conv_forte_tilde_Phi_1}
  -i \int_{\partial D_R} \!\!\big( (i\nabla +A_{\mathbf e})\tilde
  \varphi_{a_{n_\ell}}  \!\cdot \nu \big)
  \overline{\varphi_{a_{n_\ell}}}\,ds
  \to -i \int_{\partial D_R} \!\!\big( (i\nabla +A_{\mathbf e})\tilde
  \Phi \cdot \nu \big) 
  \overline{\tilde \Phi}\,ds
  = \int_{D_R} |(i\nabla + A_{\mathbf e})\tilde\Phi|^2dx
\end{equation}
as $\ell\to\infty$.
On the other hand, multiplying equation
\eqref{eq:equations_for_blowedup} by $\tilde \varphi_{a_{n_\ell}}$ with
$\ell$ large
and integrating by parts in $D_R$ for
$R>1$, we obtain 
\begin{equation}\label{eq:conv_forte_tilde_Phi_2}
  \int_{D_R} |(i\nabla + A_{\mathbf e})\tilde\varphi_{a_{n_\ell}}|^2 dx 
  = \lambda_{a_{n_\ell}} \,|a_{n_\ell}|^2 \int_{D_R} |\tilde
  \varphi_{a_{n_\ell}}|^2 dx 
  -i \int_{\partial D_R} \big( (i\nabla +A_{\mathbf e})\tilde
  \varphi_{a_{n_\ell}} \cdot \nu \big)
  \overline{\tilde \varphi_{a_{n_\ell}}}\,ds.
\end{equation}
Combining \eqref{eq:conv_forte_tilde_Phi_1} and
\eqref{eq:conv_forte_tilde_Phi_2},  we obtain that
\[
\int_{D_R} |(i\nabla + A_{\mathbf e})\tilde\varphi_{a_{n_\ell}}|^2 dx
\to \int_{D_R} |(i\nabla + A_{\mathbf e})\tilde\Phi|^2dx\quad\text{as }\ell\to\infty,
\]
whereas the compactness of the trace embedding
$H^{1,{\mathbf e}}(D_{R},\C)\hookrightarrow
  L^2(\partial D_{R},\C)$ yields 
\[
\int_{\partial D_R} |\tilde\varphi_{a_{n_\ell}}|^2 ds
\to \int_{\partial D_R} |\tilde\Phi|^2ds\quad\text{as }\ell\to\infty,
\]
so that, in view of Lemma \ref{poincare_inequality}, we can conclude that
\[
\|\tilde\varphi_{a_{n_\ell}}\|_{H^{1 ,{\mathbf e}}(D_R,\C)}\to 
\|\tilde\Phi\|_{H^{1 ,{\mathbf e}}(D_R,\C)}\quad\text{as }\ell\to\infty,
\]
and hence $\tilde\varphi_{a_{n_\ell}}\to \tilde\Phi$ strongly in $H^{1
  ,{\mathbf e}}(D_R,\C)$ for every $R>1$ as desired.

Passing to the limit along $a_{n_\ell}$ in \eqref{eq:41} and recalling
\eqref{eq:vkext_la}, we obtain
that, for every $R>2$,
\begin{equation*}
 \int_{\R^2\setminus D_{R}}\bigg|(i\nabla+A_{\mathbf e})
 \Big(\tilde\Phi(x) - c\beta e^{\frac i2\theta_{\mathbf e}}
\psi_k\Big)\bigg|^2dx<+\infty,
\end{equation*}
where $\beta$ is defined in \eqref{eq:beta}. Hence 
\begin{equation}\label{eq:55}
 \int_{\R^2}\bigg|(i\nabla+A_{\mathbf e})
 \Big(\tilde\Phi(x) -c\beta e^{\frac i2\theta_{\mathbf e}}
 \psi_k\Big)\bigg|^2dx<+\infty.
\end{equation}
Estimate \eqref{eq:55} implies that $c>0$. Indeed, $c=0$
would imply that $\int_{\R^2}|(i\nabla+A_{\mathbf e})
\tilde\Phi|^2dx<+\infty$ and then, arguing as in the proof of
Proposition \ref{p:uniqueness}, we could prove that $\tilde\Phi\equiv
0$, thus contradicting \eqref{eq:56}.

Then, from \eqref{eq:tildePhi}, \eqref{eq:55}, and Proposition
\ref{p:uniqueness} we deduce that necessarily
\begin{equation}\label{eq:57}
\tilde\Phi=c\beta\Psi_k
\end{equation}
with   $\Psi_k$ being the function defined in \eqref{eq:10}. 
From \eqref{eq:56}, \eqref{eq:57}, and the fact that $c$ is a positive
real number, it follows that 
\[
c=
\frac{1}{|\beta|}
\sqrt{\frac{K_\delta}{\int_{\partial D_{K_\delta}}|\Psi_k|^2ds}}.
\]
Hence we have that 
\[
\tilde \varphi_{a_{n_\ell}}\to 
\frac{\beta}{|\beta|}
\sqrt{\tfrac{K_\delta}{\int_{\partial D_{K_\delta}}|\Psi_k|^2ds}}
\Psi_k
\quad\text{ in }H^{1 ,{\mathbf e}}(D_R,\C)\text{  
 for
every $R>1$ and a. e.},
\] 
and 
\[
\frac{|a_{n_\ell}|^{k/2}}{\sqrt{H(\varphi_{a_{n_\ell}},K_\delta|a_{n_\ell}|)}}\to \frac{1}{|\beta|}
\sqrt{\frac{K_\delta}{\int_{\partial D_{K_\delta}}|\Psi_k|^2ds}}.
\]
Since the above limits depend neither on the sequence $\{a_{n}\}_n$ nor on the
  subsequence $\{a_{n_\ell}\}_\ell$, we conclude that the above
  convergences hold as $|a|\to 0^+$, thus concluding  the proof of the
  theorem (the convergence in $C^{2}_{\rm
  loc}(\R^2\setminus\{{\mathbf e}\},\C)$ follows easily from classical elliptic estimates).
 \end{proof}

\begin{Theorem}\label{t:blowup2}
  For $a=(|a|,0)\in\Omega$, let $\varphi_a\in H^{1,a}_{0}(\Omega,\C)$
  solve (\ref{eq:equation_a}-\ref{eq:6}) and $\varphi_0\in
  H^{1,0}_{0}(\Omega,\C)$ be a solution to
  (\ref{eq:equation_lambda0}--\ref{eq:83}) satisfying \eqref{eq:1},
  \eqref{eq:37}, and \eqref{eq:54}. Then
\[
\dfrac{\varphi_a(|a|x)}{|a|^{k/2}}\to
\frac{\beta_k^2(0,\varphi_0,\lambda_0)}{\sqrt\pi}\Psi_k \quad\text{as
}|a|\to0^+,
\]
in $H^{1 ,{\mathbf e}}(D_R,\C)$ for
every $R>1$, a.e. and in $C^{2}_{\rm
  loc}(\R^2\setminus\{{\mathbf e}\},\C)$,
with $\beta_k^2(0,\varphi_0,\lambda_0)\neq 0$ as in \eqref{eq:84} and
 $\Psi_k$ being the function defined in \eqref{eq:10}.  
\end{Theorem}
\begin{proof}
  It follows directly from convergences \eqref{eq:58} and
  \eqref{eq:59} established in Theorem \ref{t:blowup}.
\end{proof}

As a consequence of Theorem \ref{t:blowup}, we can prove convergence of
the blow-up family of functions introduced in \eqref{eq:85}.
Let $z_R$ be the unique solution to the minimization problem 
\begin{multline}\label{eq:def_zR}
  \int_{D_{R}} |(i\nabla +A_0)
  z_R(x)|^2\,dx\\
  = \min\left\{ \int_{D_{R}} |(i\nabla +A_0)u(x)|^2\,dx:\, u\in
    H^{1,0}(D_{R},\C), \ u= e^{\frac{i}{2}(\theta_0 - \theta_{\mathbf e})}
\Psi_k \text{ on }\partial D_{R} \right\},
\end{multline}
which then  solves
\begin{equation}\label{eq:117}
 \begin{cases}
  (i\nabla +A_0)^2 z_R = 0, &\text{in }D_{R},\\
  z_R = e^{\frac{i}{2}(\theta_0 - \theta_{\mathbf e})} \Psi_k, &\text{on }\partial D_{R}.
 \end{cases}
\end{equation}

\begin{Lemma}\label{lemma_blow-up}
Under the same assumptions as in Theorem \ref{t:blowup}, let $Z_a^R $ be as in \eqref{eq:85}.
Then, for all $R>2$,
\begin{equation}\label{eq:vkint} 
  Z_{a}^R \to  \frac{\beta_k^2(0,\varphi_0,\lambda_0)}{|\beta_k^2(0,\varphi_0,\lambda_0)|}
  \sqrt{\frac{K_\delta}{\int_{\partial D_{K_\delta}}|\Psi_k|^2ds}} \
  z_R, \quad\text{in  }H^{1,0}(D_R,\C),
\end{equation}
as $|a|\to0^+$.
\end{Lemma}
\begin{proof}
Let us denote
\[
\gamma_\delta=\frac{\beta_k^2(0,\varphi_0,\lambda_0)}{|\beta_k^2(0,\varphi_0,\lambda_0)|}
\sqrt{\frac{K_\delta}{\int_{\partial D_{K_\delta}}|\Psi_k|^2ds}}.
\]
We notice that $Z_a^R -  \gamma_\delta z_R$ solves
\begin{equation*}
 \begin{cases}
   (i\nabla +A_0)^2 \big(Z_a^R -  \gamma_\delta z_R\big) =0, &\text{in $D_R$}\\
   Z_a^R - \gamma_\delta z_R = e^{\frac{i}{2}(\theta_0-\theta_{\mathbf
       e})}\left(\tilde \varphi_a - \gamma_\delta\Psi_k\right), &\text{on $\partial D_R$},
 \end{cases}
\end{equation*}
and, by the Dirichlet principle and Theorem \ref{t:blowup},
\begin{align*}
 \int_{D_R} &\abs{(i\nabla +A_0)(Z_a^R - \gamma_\delta z_R) }^2dx
\leq \int_{D_R} \abs{(i\nabla +A_0)\big(\eta_R\,
e^{\frac{i}{2}(\theta_0-\theta_{\mathbf e})}(\tilde \varphi_a - \gamma_\delta\Psi_k)\big)}^2 dx\\
&\leq 2 \int_{D_R}|\nabla \eta_R|^2\big|
\tilde \varphi_a - \gamma_\delta\Psi_k\big|^2dx 
+2\int_{D_R\setminus D_{R/2}}
\eta_R^2\big|(i\nabla +A_{\mathbf e})(\tilde \varphi_a - \gamma_\delta\Psi_k)\big|^2dx\\
&= o(1)\quad \text{as }
|a|\to0^+,
\end{align*}
where $\eta_R:\R^2\to\R$ is a smooth cut-off function as in \eqref{eq:88}.
Then, taking into account
\eqref{eq:hardy}, we conclude that 
convergence \eqref{eq:vkint} holds.
\end{proof}

\section{Sharp asymptotics for convergence of eigenvalues}\label{sec:9}

In view of the exact asymptotics of the term
$H(\varphi_{a},K_\delta|a|)$ established in \eqref{eq:58},
Proposition \ref{prop:quasi_finito1} yields a control of
$\lambda_0-\lambda_a$ with $|a|^k$ both from above and below. To compute
explicitly the limit of $\frac{\lambda_0-\lambda_a}{|a|^k}$ it remains
to determine the limit of the function $f_R(a)$ in Lemma
\ref{l:stima_Lambda0_sopra} as $|a|\to 0$ and $R\to+\infty$.

\begin{Lemma}\label{l:limite_kappa_R}
For all $R>\tilde R$ and  $a=(|a|,0)\in\Omega$ with
  $|a|<\frac{R_0}R$, let $f_R(a)$ be as in Lemma
  \ref{l:stima_Lambda0_sopra}.
Then 
\begin{equation}\label{eq:60}
\lim_{|a|\to0^+}f_R(a)= -i \frac{K_\delta}{\int_{\partial D_{K_\delta}}|\Psi_k|^2ds}
\kappa_R
\end{equation}
where 
\begin{equation}\label{eq:68}
\kappa_R=\int_{\partial D_R}\Big(e^{-\frac i2 \theta_0}
  (i\nabla+A_0)z_R\cdot\nu -e^{-\frac i2 \theta_{\mathbf e}}
  (i\nabla+A_{\mathbf e})\Psi_k\cdot\nu \Big) e^{\frac i2
    \theta_{\mathbf
      e}}\overline{\Psi_k}\,ds.
\end{equation}
Furthermore
\begin{equation}\label{eq:71}
  \lim_{R\to+\infty}\kappa_R=-4i{\mathfrak m}_k,
\end{equation}
where ${\mathfrak m}_k$ is defined in \eqref{eq:Ik}.
\end{Lemma}
\begin{proof}
We first observe that, by Theorem \ref{t:blowup}, Lemma
\ref{lemma_blow-up}, \eqref{eq:16}, and \eqref{eq:117}, 
\begin{align*}
  &\lim_{|a|\to0^+}\int_{D_R}|(i\nabla+A_0)Z_a^R|^2\,dx-\int_{D_R}|(i\nabla+A_{\mathbf
    e})\tilde \varphi_a|^2\,dx\\
  &\quad= \tfrac{K_\delta}{\int_{\partial
      D_{K_\delta}}\!\!|\Psi_k|^2ds}
  \bigg(\int_{D_R}|(i\nabla+A_0)z_R|^2\,dx-\int_{D_R}|(i\nabla+A_{\mathbf
    e})\tilde \Psi_k|^2\,dx\bigg) = \tfrac{-i K_\delta}{\int_{\partial
      D_{K_\delta}}\!\!|\Psi_k|^2ds} \kappa_R
\end{align*}
with $\kappa_R$ as in \eqref{eq:68}. Hence \eqref{eq:60} follows from 
\eqref{eq:69}.
We divide the proof of \eqref{eq:71} in two steps.

\smallskip\noindent{\bf Step 1.} We claim that 
\begin{equation}\label{eq:108}
  \kappa_R=\int_{\partial D_R}\Big(e^{-\frac i2 \theta_0} (i\nabla+A_0)z_R\cdot\nu
  -e^{-\frac i2 \theta_{\mathbf e}} (i\nabla+A_{\mathbf e})\Psi_k\cdot\nu
  \Big) \psi_k\,ds+o(1)
\end{equation}
as $R\to \infty$.
To prove the claim, we observe that 
\begin{align}\label{eq:109}
  \kappa_R &=\int_{\partial D_R}\Big(e^{-\frac i2 \theta_0}
  (i\nabla+A_0)z_R\cdot\nu -e^{-\frac i2 \theta_{\mathbf e}}
  (i\nabla+A_{\mathbf e})\Psi_k\cdot\nu \Big) \Big(e^{\frac i2
    \theta_{\mathbf
      e}}\overline{\Psi_k}-\psi_k\Big)\,ds\\
  \notag &\quad+ \int_{\partial D_R}\Big(e^{-\frac i2 \theta_0}
  (i\nabla+A_0)z_R\cdot\nu -e^{-\frac i2 \theta_{\mathbf e}}
  (i\nabla+A_{\mathbf e})\Psi_k\cdot\nu
  \Big) \psi_k\\
  \notag &=\int_{\partial D_R}\Big(e^{-\frac i2
    (\theta_0-\theta_{\mathbf
      e})}\overline{\Psi_k}-e^{-\frac{i}{2}\theta_0}\psi_k\Big)
  (i\nabla+A_0)\Big(z_R-e^{\frac i2 (\theta_0-\theta_{\mathbf
      e})}\Psi_k\Big) \cdot\nu
  \,ds\\
  \notag&\quad+ \int_{\partial D_R}\Big(e^{-\frac i2 \theta_0}
  (i\nabla+A_0)z_R\cdot\nu -e^{-\frac i2 \theta_{\mathbf e}}
  (i\nabla+A_{\mathbf e})\Psi_k\cdot\nu \Big) \psi_k\,ds\\
  \notag&=\int_{\partial D_R}\Big(e^{-\frac i2 \theta_0}
  (i\nabla+A_0)z_R\cdot\nu -e^{-\frac i2 \theta_{\mathbf e}}
  (i\nabla+A_{\mathbf e})\Psi_k\cdot\nu \Big) \psi_k\,ds+I_1(R)+I_2(R)
\end{align}
where 
\begin{align*}
  I_1(R)&=\int_{\partial D_R}\Big(e^{-\frac i2
    (\theta_0-\theta_{\mathbf
      e})}\overline{\Psi_k}-e^{-\frac{i}{2}\theta_0}\psi_k\Big)
  (i\nabla+A_0)\Big(e^{\frac i2 \theta_0}\psi_k-e^{\frac i2
    (\theta_0-\theta_{\mathbf e})}\Psi_k\Big) \cdot\nu
  \,ds\\
  &= \int_{\partial
    D_R}\Big(\overline{\Psi_k}-e^{-\frac{i}{2}\theta_{\mathbf
      e}}\psi_k\Big) (i\nabla+A_{\mathbf e})\Big(e^{\frac i2
    \theta_{\mathbf e}}\psi_k-\Psi_k\Big) \cdot\nu
  \,ds\\
  I_2(R)&=\int_{\partial D_R}\Big(e^{-\frac i2
    (\theta_0-\theta_{\mathbf
      e})}\overline{\Psi_k}-e^{-\frac{i}{2}\theta_0}\psi_k\Big)
  (i\nabla+A_0)\Big(z_R-e^{\frac i2 \theta_0}\psi_k \Big) \cdot\nu
  \,ds.
\end{align*} 
Testing the equation
\[ 
(i\nabla +A_{\mathbf e})^2
\big(e^{\frac{i}{2}\theta_{\mathbf e}}\psi_k-{ \Psi_k}\big) =0, 
\]
which is satisfied in $\R^2\setminus D_R$, with the function 
$\big(e^{\frac{i}{2}\theta_{\mathbf e}}\psi_k-{
  \Psi_k}\big)(1-\eta_{2R})^2$ (being $\eta_{2R}$ as in \eqref{eq:88}), we obtain that
\begin{align*}
  I_1(R)&= i \int_{\R^2 \setminus D_R} (i\nabla +A_{\mathbf
    e})\big(e^{\frac{i}{2}\theta_{\mathbf e}}\psi_k-{
    \Psi_k}\big)\cdot \overline{(i\nabla +A_{\mathbf
      e})\big((1-\eta_{2R})^2 (e^{\frac{i}{2}\theta_{\mathbf
        e}}\psi_k-{ \Psi_k)}\big)}
  \,dx \\
  &= i \int_{\R^2 \setminus D_R} |(i\nabla +A_{\mathbf
    e})\big(e^{\frac{i}{2}\theta_{\mathbf e}}\psi_k-{
    \Psi_k}\big)|^2(1-\eta_{2R})^2
  \,dx  \\
  &\quad +2\int_{\R^2 \setminus D_R}(1-\eta_{2R})
  (e^{-\frac{i}{2}\theta_{\mathbf e}}\psi_k- \overline{\Psi_k})
  (i\nabla +A_{\mathbf e})\big(e^{\frac{i}{2}\theta_{\mathbf
      e}}\psi_k-{ \Psi_k}\big)\cdot\nabla\eta_{2R} \,dx,
\end{align*}
and hence, thanks to \eqref{eq:17} and \eqref{eq:75},
\begin{align}\label{eq:110}
  &|I_1(R)|\leq 2\int_{\R^2 \setminus D_R} |(i\nabla +A_{\mathbf
    e})\big(e^{\frac{i}{2}\theta_{\mathbf e}}\psi_k-{ \Psi_k}\big)|^2
  \,dx \\
  \notag&\qquad\qquad + \int_{D_{2R} \setminus D_R}
  |e^{\frac{i}{2}\theta_{\mathbf e}}\psi_k-
  \Psi_k|^2|\nabla\eta_{2R}|^2 \,dx \\
  \notag&\leq 2\int_{\R^2 \setminus D_R} |(i\nabla +A_{\mathbf
    e})\big(e^{\frac{i}{2}\theta_{\mathbf e}}\psi_k-{ \Psi_k}\big)|^2
  \,dx +\frac{4}{R^2} \int_{D_{2R} \setminus D_R}
  |e^{\frac{i}{2}\theta_{\mathbf e}}\psi_k- \Psi_k|^2 \,dx \to0
\end{align}
as $R\to +\infty$.

On the other hand, testing the equation $(i\nabla +A_0)^2
(e^{\frac{i}{2}\theta_0}\psi_k - z_R)=0$ in $D_R$ with the function
$\eta_R \big(e^{\frac{i}{2}(\theta_0 - \theta_{\mathbf e})}{ \Psi_k}
-e^{\frac{i}{2}\theta_0}\psi_k\big)$ (with $\eta_R$ as in
\eqref{eq:88}) and using the Dirichlet Principle, we have that 
\begin{align*}
  &|I_2(R)|=\bigg|-i \int_{D_R} (i\nabla +A_0)(e^{\frac{i}{2}\theta_0}\psi_k -
  z_R)\cdot \overline{(i\nabla +A_0)\Big(\eta_R
    \big(e^{\frac{i}{2}(\theta_0 -
      \theta_{\mathbf e})}\Psi_k -e^{\frac{i}{2}\theta_0}\psi_k\big)\Big)}\,dx\bigg|\\
  &\leq \left(\int_{D_R} \big|(i\nabla +A_0) \big
    (e^{\frac{i}{2}\theta_0}\psi_k - z_R \big)
    \big|^2\,dx\right)^{\!\!1/2}\times\\
  &\hskip3cm\times \left(\int_{D_R} \Big|(i\nabla +A_0) \Big(\eta_R
    \big(e^{\frac{i}{2}(\theta_0 -
      \theta_{\mathbf e})}\Psi_k -e^{\frac{i}{2}\theta_0}\psi_k\big)\Big)\Big|^2dx\right)^{1/2} \\
  & \leq \int_{D_R} \Big|(i\nabla +A_0) \Big(\eta_R
  \big(e^{\frac{i}{2}(\theta_0 -
    \theta_{\mathbf e})}\Psi_k -e^{\frac{i}{2}\theta_0}\psi_k\big)\Big)\Big|^2dx\\
  & \leq 2 \int_{D_R \setminus D_{\frac R2}} \!\!\Big|(i\nabla +A_0)
  \Big( e^{\frac{i}{2}(\theta_0 - \theta_{\mathbf e})}\Psi_k
  -e^{\frac{i}{2}\theta_0}\psi_k\Big)\Big|^2dx+2\int_{D_R \setminus
    D_{\frac R2}}\!\! |\nabla \eta_R|^2
  \big| \Psi_k -e^{\frac{i}{2}\theta_{\mathbf e}}\psi_k \big|^2 dx\\
  & \leq 2 \int_{D_R \setminus D_{\frac R2}} \Big|(i\nabla +A_{\mathbf
    e}) \Big( \Psi_k -e^{\frac{i}{2}\theta_{\mathbf
      e}}\psi_k\Big)\Big|^2dx+\frac{32}{R^2} \int_{D_R \setminus
    D_{\frac R2}}\big| \Psi_k -e^{\frac{i}{2}\theta_{\mathbf e}}\psi_k
  \big|^2\,dx
\end{align*}
which, in view of \eqref{eq:17} and estimate
\eqref{eq:75}, yields that 
\begin{equation}\label{eq:111}
  I_2(R)\to 0
\end{equation}
as $R\to +\infty$. Claim \eqref{eq:108} then follows from \eqref{eq:109},
\eqref{eq:110}, and \eqref{eq:111}.

\smallskip\noindent{\bf Step 2.} We claim that
\begin{equation}\label{eq:step2}
  \int_{\partial D_R}\Big(e^{-\frac i2 \theta_0} (i\nabla+A_0)z_R\cdot\nu
-e^{-\frac i2 \theta_{\mathbf e}} (i\nabla+A_{\mathbf e})\Psi_k\cdot\nu
\Big) \psi_k\,ds=ik\sqrt\pi\big(\xi(1)-\sqrt \pi \big),
\end{equation}
where the function $\xi$ is defined in \eqref{eq:coeff_fourier_Psi1}. 
From \eqref{coeff_fourier_centrata0} and
\eqref{eq:fourier_centrata0},  the function $\xi$
satisfies
\[ 
\big( r^{-k/2}\xi(r) \big)'= \dfrac{C_\xi}{r^{1+k}}, \quad \text{in
}(1,+\infty),
 \]
 for some $C_\xi\in\C$.
Integrating the previous equation over $(1,r)$ we obtain that
\begin{equation}\label{eq:114} 
r^{-k/2} \xi(r) -\xi(1) = \frac{C_\xi}{k}\left(1-\frac1{r^k}\right). 
\end{equation}
From \eqref{eq:psi_k} and estimate \eqref{eq:75} it follows that 
\begin{align*}
  \xi(r)&=\frac1{\sqrt\pi}\int_0^{2\pi}\psi_k(r\cos t,r\sin t)
\sin \big(\tfrac k2 t\big)\,dt\\
&+\int_{0}^{2\pi} e^{\frac{i}{2}(\theta_0-\theta_{\mathbf
      e})(r\cos t,r\sin t) } \Big(\Psi_k(r\cos t,r\sin t)-e^{\frac{i}{2}\theta_{\mathbf
      e}(r\cos t,r\sin t) }\psi_k(r\cos t,r\sin t) \Big)
  \overline{\psi^k_2(t)}\,dt\\
&=\sqrt{\pi}\,r^{k/2}+O(r^{-1/2}),\quad\text{as }r\to+\infty,
\end{align*}
and hence 
$r^{-k/2} \xi(r) \to \sqrt\pi$ as $r\to+\infty$. 
Letting $r\to+\infty$ in \eqref{eq:114}, this implies that $\tfrac{C_\xi}{k}=\sqrt\pi -\xi(1)$, so that 
\begin{align}
 \xi(r)&= \sqrt\pi \,r^{k/2} +  r^{-k/2} \big( \xi(1)-\sqrt\pi
 \big),\quad r>1, \label{eq:xi}\\
 \xi'(r)&=\frac{k}{2}\sqrt\pi r^{k/2-1} + (\sqrt\pi - \xi(1))\frac{k}{2} r^{-k/2-1},\quad r>1. \label{eq:xi'}
\end{align}
In particular, from \eqref{eq:xi} we have that 
\begin{equation}\label{eq:sqrtpi-xi1}
 \sqrt\pi - \xi(1) = \sqrt\pi r^k - r^{k/2}\xi(r),\quad \text{for
   all }r>1,
\end{equation}
whose substitution into \eqref{eq:xi'} yields
\begin{equation}\label{eq:115}
  \xi'(r)=k\sqrt{\pi}r^{k/2-1}-\frac k2\,\frac{\xi(r)}{r},\quad \text{for
   all }r>1.
\end{equation}
On the other hand, writing $\xi$ as 
\[
\xi(r)=\frac1r\int_{\partial D_r}
e^{\frac{i}{2}(\theta_0-\theta_{\mathbf e})}\Psi_k(x)\,
\overline{\psi^k_2(\theta_0(x))}\,ds(x),
\]
differentiating and taking into account \eqref{eq:9}, \eqref{eq:psi_k}
and the fact that $A_0\cdot \nu=0$ on $\partial D_r$, we obtain also that 
\begin{align}\label{eq:112}
  \xi'(r)&=\frac1r\int_{\partial D_r}
  \nabla\Big(e^{\frac{i}{2}(\theta_0-\theta_{\mathbf
      e})}\Psi_k\Big)\cdot\nu \,
  \overline{\psi^k_2(\theta_0(x))}\,ds\\
  \notag&=-\frac{i}{\sqrt\pi}r^{-\frac k2-1}\int_{\partial
    D_r}e^{-\frac{i}{2}\theta_{\mathbf e}}(i\nabla+A_{\mathbf
    e})\Psi_k\cdot\nu\, \psi_k\, ds.
\end{align}
Combination of  \eqref{eq:115} and \eqref{eq:112} yields that
\begin{equation}\label{eq:113}
  \int_{\partial D_r}e^{-\frac{i}{2}\theta_{\mathbf
      e}}(i\nabla+A_{\mathbf e})\Psi_k\cdot\nu\, \psi_k\, ds=i\sqrt\pi
  r^{k/2+1}\xi'(r)=i\sqrt\pi
  r^{k/2+1}\bigg(k\sqrt{\pi}r^{k/2-1}-\frac k2\,\frac{\xi(r)}{r} \bigg),
\end{equation}
for all $r>1$.

According to \eqref{coeff_fourier_centrata0} and
\eqref{eq:fourier_centrata0},  the function $\zeta_R$ defined as 
\begin{equation*}
\zeta_R(r):= \int_{0}^{2\pi} z_R(r\cos t,r\sin t) \overline{\psi^k_2(t)}\,dt
\end{equation*}
satisfies, for some $C_{R}\in\C$,
\[ 
\big( r^{-k/2}\zeta_R(r) \big)'= \dfrac{C_{R}}{r^{1+k}}, \quad r\in(0,R). 
\]
Integrating the previous equation over $(r,R)$ we obtain 
\[
 R^{-k/2} \zeta_R(R) -r^{-k/2} \zeta_R(r) =
\frac{C_{R}}{k} \left(\frac{1}{r^k}-\frac{1}{R^k}\right), \quad \text{for all }r\in(0,R]. 
 \]
Since by Proposition \ref{prop:fft} $\zeta_R(r)=O(r^{1/2})$ as $r\to 0^+$,
we necessarily have $C_{R}=0$. Hence 
\begin{align}
 \notag \zeta_R(r)&= \dfrac{\zeta_R(R)}{R^{k/2}} r^{k/2}, \quad \text{for all }r\in(0,R],\\
 \label{eq:116}\zeta_R'(r)&=\frac{k}{2}\dfrac{\zeta_R(R)}{R^{k/2}} r^{k/2-1}, \quad \text{for all }r\in(0,R].
\end{align}
On the other hand, writing $\zeta_R$ as 
\[
\zeta_R(r)=\frac1r\int_{\partial D_r}
z_R(x)\,
\overline{\psi^k_2(\theta_0(x))}\,ds(x),
\]
differentiating and using  \eqref{eq:9}, \eqref{eq:psi_k}
and $A_0\cdot \nu=0$ on $\partial D_r$, we obtain that
\begin{equation}\label{eq:112_la}
  \zeta_R'(r)=\frac1r\int_{\partial D_r}
  \nabla z_R\cdot\nu \,
  \overline{\psi^k_2(\theta_0(x))}\,ds=-\frac{i}{\sqrt\pi}r^{-\frac k2-1}
  \int_{\partial D_r}e^{-\frac{i}{2}\theta_0}(i\nabla+A_{0})z_R\cdot\nu\, \psi_k\, ds.
\end{equation}
Combination of  \eqref{eq:116} and \eqref{eq:112_la} yields that
\begin{equation}\label{eq:secondocontributo}
  \int_{\partial
    D_r}e^{-\frac{i}{2}\theta_0}(i\nabla+A_{0})z_R\cdot\nu\, \psi_k\,
  ds=
  \frac{ik}2 \sqrt\pi \frac{\zeta_R(R)}{R^{k/2}}r^k
\end{equation}
for all $r\in(0,R]$.

From the boundary condition in \eqref{eq:117} it follows that
$\xi(R)=\zeta_R(R)$. Hence, collecting \eqref{eq:113},
\eqref{eq:secondocontributo}, and \eqref{eq:sqrtpi-xi1} we obtain that 
\begin{align*}
  \int_{\partial D_R}&\Big(e^{-\frac i2 \theta_0}
  (i\nabla+A_0)z_R\cdot\nu -e^{-\frac i2 \theta_{\mathbf e}}
  (i\nabla+A_{\mathbf e})\Psi_k\cdot\nu
  \Big) \psi_k\,ds\\
  &= \frac{ik}2 \sqrt\pi \frac{\zeta_R(R)}{R^{k/2}}R^k- i\sqrt\pi
  R^{k/2+1}\bigg(k\sqrt{\pi}R^{k/2-1}-\frac k2\,\frac{\xi(R)}{R} \bigg) \\
  &=ik\sqrt\pi \big(\xi(R)R^{k/2}-\sqrt \pi
  R^k\big)=ik\sqrt\pi\big(\xi(1)-\sqrt \pi \big),
\end{align*}
thus proving claim \eqref{eq:step2}.

Combining \eqref{eq:108} with
\eqref{eq:step2} we obtain that 
\begin{equation}\label{eq:124}
  \kappa_R=ik\sqrt\pi(\xi(1)-\sqrt\pi)+o(1),\quad\text{as }R\to+\infty.
\end{equation}
The conclusion \eqref{eq:71} follows by \eqref{eq:124}  and Lemma
\ref{l:legame_compl} (see also \eqref{eq:43}).
\end{proof}

We are now in position to prove our main result.
\begin{proof}[Proof of Theorem \ref{t:main_asy_eige}]
From Proposition \ref{prop:quasi_finito1}, Lemma
\ref{l:stima_Lambda0_sopra}, Lemma \ref{l:limite_kappa_R}, and \eqref{eq:58}
it follows that, for every $R>\tilde R$, 
\begin{align*}
  -4 \frac{|\beta_k^2(0,\varphi_0,\lambda_0)|^2}{\pi}
  &\mathfrak{m}_k\, (1+o(1)) \leq \frac{\lambda_0 -
    \lambda_{a}}{|a|^k} \leq f_R(a)
  \frac{H(\varphi_{a},K_\delta|a|)}{|a|^{k}}\\
  &=\bigg( -i \frac{K_\delta}{\int_{\partial
      D_{K_\delta}}|\Psi_k|^2ds} \kappa_R+o(1)\bigg)
  \bigg(\frac{|\beta_k^2(0,\varphi_0,\lambda_0)|^2}{\pi}
  \frac{\int_{\partial D_{K_\delta}}|\Psi_k|^2ds}{K_\delta}+o(1)\bigg)
\end{align*}
as $|a|\to0^+$.
Hence 
\begin{equation}\label{eq:72}
  -4 \frac{|\beta_k^2(0,\varphi_0,\lambda_0)|^2}{\pi} \mathfrak{m}_k\leq \liminf_{|a|\to0^+}\frac{\lambda_0 -
    \lambda_{a}}{|a|^k} \leq 
  \limsup_{|a|\to0^+}\frac{\lambda_0 -
    \lambda_{a}}{|a|^k} \leq 
  -i 
  \kappa_R
  \frac{|\beta_k^2(0,\varphi_0,\lambda_0)|^2}{\pi}
\end{equation}
 for every $R>\tilde R$.  From \eqref{eq:71}, letting $R\to+\infty$ in
 \eqref{eq:72} we obtain the chain of inequalities
\[
-4 \frac{|\beta_k^2(0,\varphi_0,\lambda_0)|^2}{\pi} \mathfrak{m}_k\leq
\liminf_{|a|\to0^+}\frac{\lambda_0 - \lambda_{a}}{|a|^k} \leq
\limsup_{|a|\to0^+}\frac{\lambda_0 - \lambda_{a}}{|a|^k} \leq -
4{\mathfrak m}_k \frac{|\beta_k^2(0,\varphi_0,\lambda_0)|^2}{\pi}
\]
which yields the conclusion (see Remark \ref{rem_beta1=0}).
\end{proof}

\bigskip\noindent {\bf Acknowledgments.}  The authors would like to
thank Susanna Terracini for her encouragement and for fruitful
discussions.

\end{document}